\documentclass[aop]{imsart}

\RequirePackage{amsthm,amsmath,amsfonts,amssymb}
\RequirePackage[numbers]{natbib}
\RequirePackage[colorlinks,citecolor=blue,urlcolor=blue]{hyperref}
\RequirePackage{graphicx}

\usepackage{mathtools}
\usepackage{enumerate}
\usepackage[all]{xy}
\usepackage{mathrsfs}
\usepackage{fancyhdr}
\usepackage{listings}
\usepackage{hyperref}
\usepackage{cleveref}
\usepackage{soul}
\usepackage{xcolor}
\usepackage{dsfont}
\usepackage{stmaryrd}
\usepackage{cleveref}
\newcommand{\red}[1]{\textcolor{red}{#1}}

\allowdisplaybreaks

\usepackage{tikz}
\usetikzlibrary{shapes,arrows, petri, topaths, automata}

\DeclareMathOperator{\im}{{\mathrm{Im}}}
\newcommand{\Pol}{{\mathbf{M}}}
\renewcommand{\Dot}{\mathbf {Dot}}

\hypersetup{
	colorlinks=true
}
\def\P{\mathbb{P}}

\def\E{\mathbb{E}}
\newcommand{\gE}[1]{\left\langle #1 \right \rangle}
\def\Z{\mathbb{Z}}
\def\R{\mathbb{R}}
\def\N{\mathbb{N}}

\def\bx{\mathbf{x}}
\def\11{{\mathbf{1}}}

\newcommand{\abs}{{\rm abs}}

\newcommand{\ord}{{\mathrm{ord}}}
\newcommand{\size}{{\rm{size}}}
\newcommand{\gh}{{\rm{gh}}}
\newcommand{\fr}{{\rm{fr}}}
\newcommand{\C}{\mathbb C}
\newcommand{\Iso}{{\mathcal I}}
\newcommand{\wt}{\widetilde}
\newcommand{\wh}{\widehat}

\newcommand{\al}{\alpha}

\newcommand{\nonuni}{{\text{complete $T$-expansion}}}
\newcommand{\Nonuni}{{\text{Complete $T$-expansion}}}

\newcommand{\dashed}{\text{diffusive}}
\newcommand{\Dashed}{\text{Diffusive}}
\newcommand{\self}{\text{self-energy}}
\newcommand{\selfs}{\text{self-energies}}

\newcommand{\Selfs}{\text{Self-energies}}

\newcommand{\incomp}{{\text{$T$-equation}}}
\newcommand{\fa}{{\mathfrak a}}
\newcommand{\fb}{{\mathfrak b}}

\newcommand{\QGn}{\mathcal Q^{(n)}}
\newcommand{\Err}{{\mathcal Err}}

\newcommand{\PT}{\mathcal R}

\newcommand{\PTn}{\mathcal R^{(n)}}

\newcommand{\QT}{\mathcal Q}
\newcommand{\QTn}{\mathcal Q^{(n)}}

\newcommand{\AT}{\mathcal A}
\newcommand{\ATn}{\mathcal A^{(>n)}}

\DeclareMathOperator{\OO}{O}
\DeclareMathOperator{\oo}{o}

\newcommand{\Wn}{\mathcal W^{(n)}}

\newcommand{\wtSdelta}{\Sigma}

\newcommand{\Sele}{{\mathcal E}}

\newcommand{\Selek}{{\mathcal E}}

\newcommand{\conc}{\mathfrak g}
\newcommand{\soe}{d_\eta}
\newcommand{\etas}{W^{-5+\delta_0}L^{5-d}}

\newcommand{\tr}[1]{\mathrm{tr}\left(#1\right)}
\newcommand{\be}{\begin{equation}}
	\newcommand{\ee}{\end{equation}}
\newcommand{\ii}{\mathrm{i}}
\newcommand{\dd}{\mathrm{d}}
\newcommand{\e}{{\varepsilon}}
\renewcommand{\cal}{\mathcal}

\newcommand{\cob}{\color{blue}}

\newcommand\dif{\mathop{}\!\mathrm{d}}

\startlocaldefs
\newtheorem{thm}{Theorem}[section]
\newtheorem{assumption}[thm]{Assumption}

\newtheorem{proposition}[thm]{Proposition}

\newtheorem{lemma}[thm]{Lemma}
\newtheorem{corollary}[thm]{Corollary}
\newtheorem{defn}[thm]{Definition}

\newtheorem{claim}[thm]{Claim}

\newtheorem{strategy}[thm]{Strategy}

\theoremstyle{definition}
\newtheorem{remark}[thm]{Remark}
\numberwithin{equation}{section}
\usepackage{environ}

\endlocaldefs

\begin{document}

\begin{frontmatter}
\title{Bulk universality and quantum unique ergodicity for random band matrices in high dimensions}
\runtitle{Bulk universality and QUE for random band matrices}

\begin{aug}
\author[A]{\fnms{Changji}~\snm{Xu}\ead[label=e1]{cxu@cmsa.fas.harvard.edu}},
\author[B]{\fnms{Fan}~\snm{Yang}\ead[label=e2]{fyangmath@mail.tsinghua.edu.cn}}, 
\author[C]{\fnms{Horng-Tzer}~\snm{Yau}\ead[label=e3]{htyau@math.harvard.edu}}
\and
\author[D]{\fnms{Jun}~\snm{Yin}\ead[label=e4]{jyin@math.ucla.edu}}
\address[A]{Center of Mathematical Sciences and Applications, Harvard University\printead[presep={,\ }]{e1}}

\address[B]{Yau Mathematical Sciences Center, Tsinghua University\printead[presep={,\ }]{e2}}

\address[C]{Department of Mathematics, Harvard University\printead[presep={,\ }]{e3}}

\address[D]{Department of Mathematics, University of California, Los Angeles\printead[presep={,\ }]{e4}}
\end{aug}

\begin{abstract}
We consider Hermitian random band matrices {\scriptsize$H=(h_{xy})$} on the {\it d}-dimensional lattice {\scriptsize$(\Z/L\Z)^d$}, where the entries {\scriptsize$h_{xy}=\overline h_{yx}$} are independent centered complex Gaussian random variables with variances {\scriptsize$s_{xy}=\mathbb E|h_{xy}|^2$}. The variance matrix {\scriptsize$S=(s_{xy})$} has a banded profile so that {\scriptsize$s_{xy}$} is negligible if {\scriptsize$|x-y|$} exceeds the band width {\scriptsize$W$}. For dimensions {\scriptsize$d\ge 7$}, we prove the bulk eigenvalue universality of {\scriptsize$H$} under the condition {\scriptsize$W \gg L^{95/(d+95)}$}. Assuming that {\scriptsize$W\geq L^\e $} for a small constant {\scriptsize$\e >0$}, we also prove the quantum unique ergodicity for the bulk eigenvectors of {\scriptsize$H$} and a sharp local law for the Green's function {\scriptsize$G(z)=(H-z)^{-1}$} up to {\scriptsize$\im \, z \gg W^{-5}L^{5-d}$}. The local law implies that the bulk eigenvector entries of $H$ are of order {\scriptsize$\OO(W^{-5/2}L^{-d/2+5/2})$} with high probability. 
\end{abstract}



\end{frontmatter}
\tableofcontents

\section{Introduction}

A $d$-dimensional random band matrix describes a Hamiltonian on a $d$-dimensional lattice with random hoppings in a band of width $W$. In this paper, we consider a large finite lattice $\Z_L^d:=\{1,2, \cdots, L\}^d$ with $N:= L^d$ many lattice sites and define a random band matrix ensemble $H=(h_{xy}:x,y\in \Z_L^d)$ on it. (Here, for the matrix indices, we fix an arbitrary ordering $1,\ldots, N$ of the vertices in $\Z_L^d$.) The entries of $H$ are independent (up to the Hermitian condition) centered random variables with a translation invariant variance profile $s_{xy}:=\mathbb E|h_{xy}|^2 = f(|x-y|/W)$ for a rapidly decaying function $f$. In particular, the variance is negligible when $|x-y|\gg W$. Without loss of generality, we can normalize $f$ so that 
\be\label{fxy}
\sum_{x}s_{xy}=\sum_{y}s_{xy}=1.
\ee
It is well-known that under \eqref{fxy}, the global eigenvalue distribution of $H$ converges weakly to the Wigner's semicircle law supported in $[-2,2]$ \cite{Wigner}. The random band matrix ensemble naturally interpolates between the mean-field Wigner ensemble \cite{Wigner} and the famous Anderson model \cite{Anderson} as $W$ varies. 
In particular, a sharp Anderson metal-insulator transition is  conjectured to occur when $W$ crosses a critical band width $W_c$:
if $W \ll W_c$, $H$ has localized bulk eigenvectors and Poisson statistics for bulk eigenvalues; if $W \gg W_c$, $H$ has delocalized bulk eigenvectors and GOE/GUE  statistics for bulk eigenvalues. Based on simulations and some non-rigorous supersymmetry arguments, the critical band width is conjectured to be $W_c = \sqrt L$ when $d=1$, $W_c = \sqrt{\log L}$ when $d=2$, and $W_c\sim 1$ when $d\ge 3$ \cite{PB_review, ConJ-Ref2, ConJ-Ref1, ConJ-Ref4, fy, Spencer3, Spencer2, Spencer1,ConJ-Ref6}.

Despite the importance of the Anderson metal-insulator transition in physics, establishing its rigorous theory for a concrete model (including the Anderson model and random band matrices) remains one of the major open problems in mathematical physics. There have been many partial results concerning the localization or delocalization of one-dimensional (1D) random band matrices  \cite{BaoErd2015,BGP_Band,BouErdYauYin2017,PartII,PartI,CS1_4,CPSS1_4,ErdKno2013,ErdKno2011,delocal,Semicircle,HeMa2018,PelSchShaSod,Sch2009,SchMT,Sch3,1Dchara,Sch1,Sch2014,Sch2,Sod2010,Band1D_III} and random band matrices in dimensions $d\ge 2$ \cite{DL_2D,DisPinSpe2002,ErdKno2013,ErdKno2011,delocal, HeMa2018,BandI,BandII,Band1D_III}. We refer the reader to \cite{PB_review,PartI,CPSS1_4,BandI} for a brief review of some of these results. 
So far, the best delocalization result for high-dimensional band matrices was obtained in \cite{BandI,BandII}, which proved that if $d\ge 8$ and $W\ge L^\e$ for an arbitrarily small constant $\e>0$, the bulk eigenvectors of random band matrices are (weakly) delocalized in the following senses:  with high probability, most bulk eigenvectors have localization length of order $L$ and the $\ell^\infty$-norm of every bulk eigenvector of $H$ is at most $W/L$. On the other hand, the eigenvalue statistics of random band matrices are much harder to study. To the best of our knowledge, the GOE/GUE statistics of band matrices with general variance profiles have only been proved in 1D under $W\ge cL$ for a constant $c>0$ in \cite{BouErdYauYin2017} and under $W\gg L^{3/4}$ in  \cite{PartI}. 

One main goal of this paper is to answer this important question for high-dimensional band matrix ensembles. We show that the bulk eigenvalue universality and quantum unique ergodicity (QUE) hold for random band matrices in dimensions $d\ge 7$ as long as the band width $W$ is reasonably large. More precisely, we obtain the following results in this paper:
\begin{itemize}
	\item[(i)] If $W\gg L^{95/(d+95)}$, the bulk eigenvalue statistics of $H$ match those of GOE/GUE asymptotically.
	
	\item[(ii)] If $W\ge L^\delta$ for an arbitrarily small constant $\delta>0$, the QUE of bulk eigenvectors holds.
	
	\item[(iii)] If $W\ge L^\delta$ for an arbitrarily small constant $\delta>0$, an optimal local law holds for the Green's function (or resolvent) of $H$,
	\be\nonumber
	G(z)=(H-z)^{-1},\quad z\in \C_+:=\{z\in \C: \im z>0\},
	\ee
	down to the scale $\im z \gg \eta_*:=W^{-5}L^{5-d}$.
\end{itemize}
The bulk universality established here gives a rare example of random matrix statistics arising from a ``local" model. Roughly speaking, the band width is viewed as an equivalent of the range of interaction, so for a band matrix ensemble to be considered a local model, $W$ needs to be finite. In very high dimensions $d$, our assumption  $W\gg L^{95/(d+95)}$ is only slightly worse than the conjectured optimal assumption that $W$ is large but finite. Notice that QUE is a stronger notion than delocalization by saying that every bulk eigenvector of $H$ is asymptotically uniformly distributed on subsets of microscopic scales $\oo(N)$. 
Equivalently, it states that given an eigenvector $u_\al$ and a real diagonal matrix $\Pi$ with $\tr\Pi=0$, $N|\langle u_\al, \Pi u_\al\rangle|$ is much smaller than its typical size $\sum_x |\Pi_{xx}|$, where the inner product $\langle \cdot, \cdot\rangle$ is defined as $\langle  u,  v\rangle:= u^* v$. 
The local law of $G(z)$ says that the resolvent entries $G_{xy}$ are well-approximated by $m(z)\delta_{xy}$ for $z=E+\ii \eta$ with $E$ in  the bulk and $\eta\gg \eta_*$, where $m(z)$ is the Stieltjes transform of Wigner's semicircle law,
\be\label{msc}
m(z):=\frac{-z+\sqrt{z^2-4}}{2} = \frac{1}{2\pi}\int_{-2}^2 \frac{\sqrt{4-\xi^2}}{\xi-z}\dd\xi.
\ee
An immediate consequence of the local law is the following delocalization estimate: in dimensions $d\ge 7$, the $\ell^\infty$-norm of every bulk eigenvector of $H$ is at most $W^{-5/2}L^{-d/2+5/2}$ with high probability provided that $W\ge L^\delta$. Compared with \cite{BandI,BandII}, this result covers a new dimension $d=7$ and improves the $\ell^\infty$ bound $W/L$ to $W^{-5/2}L^{-d/2+5/2}$, which is already very close to the optimal bound $L^{-d/2+\e}$---the extra factor is only $(L/W)^{5/2}$.

One main strategy for the proof of the bulk universality of $H$ roughly follows the three-step strategy initiated in \cite{ErdPecRamSchYau2010,ErdSchYau2011}. The reader can refer to \cite{erdHos2017dynamical} for an overview of this strategy. Our local law in Theorem \ref{thm: locallaw} below completes the first step for $\eta\gg \eta_*$. In the second step, we consider the matrix Brownian motion $H_t:=\sqrt{1-t} H+B_t$, where $B_t$ is a Hermitian matrix whose entries are independent Brownian motions with variances $t/N$ (i.e., for any fixed time $t$, $B_t$ has the law of $\sqrt{t}\text{GOE}$ or $\sqrt{t}\text{GUE}$). With the local law as the input, it was proved in \cite{LSY2019,LY17} that the local bulk statistics of $H_t$ converge to those of GOE/GUE at any time $t\gg \eta_*$, which provides the optimal second step given the local law. In the third step, we need to show that the local bulk statistics of the original matrix $H$ are well-approximated by those of $H_t$, which is achieved by comparing the moments of $\im g(z) $ with those of $\im g_t(z) $, where $ g(z):= \frac{1}{N}\tr{G(z)}$ and $g_t(z):= \frac{1}{N}\tr{G_t(z)}$ with $G_t(z):=(H_t-z)^{-1}$. This presents a main obstacle to the three-step strategy because previous Green's function comparison arguments for this step require matching the variances of the entries of $H$ with those of $H_t$. However, the variances of the entries of a random band matrix (almost) vanish outside the band, so the variance matching will never hold.

To overcome the obstacle in the third step mentioned above, a mean-field reduction method was introduced in \cite{BouErdYauYin2017,PartI}. Roughly speaking, it reduces the study of some spectral properties of  $H$ to that of a $W\times W$ mean-field random matrix $Q$ expressed as a rational function of certain blocks of $H$. A key observation in \cite{BouErdYauYin2017,PartI} is that the QUE, bulk universality, and local law of $Q$ imply the QUE and bulk universality of the original band matrix $H$. However, the local law of $Q$ was established via the \emph{generalized resolvent} of $H$ \cite{PartII,Band1D_III}. Besides the fact that generalized resolvents seem to be extremely difficult to estimate in high dimensions, it looks like the mean-field reduction approach is applicable only when the band width is large, while we aim to deal with narrow band matrices in this work.

In this paper, we prove the bulk universality of $H$ using the three-step strategy but without using the mean-field reduction argument. Roughly speaking, a key observation is that in controlling $\E (\im g_t(z) )^k - \E (\im g(z))^k$, the leading quantities governing its evolution with respect to $t$ will contain factors of the form $N|\langle u_\al, \Pi u_\al\rangle|^2$, where $u_\al$ denotes a bulk eigenvector and $\Pi$ is a real diagonal matrix with $\tr\Pi=0$ and $\sum_x |\Pi_{xx}|=1$. Applying the QUE for $H$, these factors are much smaller than 1, which implies that, from $t=0$ to $t=L^\e\eta_*$, $g_t(z)$ changes by a power of $W^{-(d+95)}L^{95+C\e}$ for some absolute constant $C>0$. Hence, as long as $W^{-(d+95)}L^{95+C\e}\ll 1$, the local spectral statistics essentially do not change along the DBM from $t=0$ to $t=L^\e\eta_*$. Together with the bulk universality for $H_t$, it gives the bulk universality of the original band matrix $H$.

Regarding the proof of QUE, previous arguments are either based on an eigenvector moment flow method \cite{BouYau2017,PartI} 
or some multi-resolvent local laws for $G(z)$ \cite{CES_QUE1,CES_QUE3,CES_QUE2}. Our current proof of QUE is instead based on an extension of the ideas in \cite{BandI,BandII}. More precisely, proving the QUE amounts to bounding high moments of $\cal P:=\tr{(\im G) \Pi (\im G) \Pi}$ for $\im G:=\frac{1}{2\ii}(G-G^*)$. A key tool developed in \cite{BandI,BandII} is a $T$-expansion of the  $T$-variable, defined by $T_{xy}:=|m|^2\sum_\al s_{x\al}|G_{\al y}|^2$ \cite{delocal}. In this paper, we extend it to expansions of high moments of $\cal P$, that is, we can expand any $\cal P^k$ into a sum of deterministic terms with arbitrarily high order error. Then, we can explore some important cancellations for these deterministic terms, which lead to the proof of the QUE.

Some key ideas used in the proof will be briefly discussed in Section \ref{sec_idea} below. We expect that the argument in this paper can be extended to random band matrices in lower dimensions and with smaller band width, and we will pursue this direction in future works. 


Finally, we mention several advantages of the current proof of bulk universality and QUE. (1) The fact that the QUE of $H$ implies the bulk universality of  $H$ is very  transparent in our current argument. In particular, we do not need to introduce an extra mean-field matrix $Q$ to establish such a connection. As a consequence, the proofs of the QUE and bulk universality are simplified greatly. 
(2) Our current proof only uses the conventional resolvent $G(z)$, whose local law is much easier to prove than the generalized resolvent in the mean-field reduction approach. 
(3)  If we can prove the local law of $G(z)$ up to the optimal scale $\im z\gg \eta_*= L^{-d}$ (instead of the current $\eta_*=W^{-5}L^{5-d}$), then our current argument yields the bulk universality of random band matrices with band width $W\ge L^\delta$. (This is because the change of $g_t(z)$ from $t=0$ to $t=L^\e\eta_*$ would become  $W^{-d}L^{C\e}$ instead of $W^{-(d+95)}L^{95+C\e}$.)  In addition, optimal QUE and complete delocalization of bulk eigenvectors would also follow. Hence, proving the local law of $G(z)$ down to $\im z\gg L^{-d}$ is now the key open problem in the study of band matrices.

\subsection{Main results}

In this subsection, we state the main results of this paper, including the bulk universality, QUE, local law, and delocalization of bulk eigenvectors. We consider $d$-dimensional random band matrices indexed by a cube of linear size \(L\) in \(\mathbb{Z}^{d}\), i.e., 
\be\label{ZLd}
\Z_L^d:=\left( \Z\cap ( -L/2 , L/2]\right) ^d. 
\ee
We will view $\Z_L^d$ as a torus and denote by $[x-y]_L$ the representative of $x-y$ in $\Z_L^d$, i.e.,  
\be\label{representativeL}[x-y]_L:= [(x-y)+L\Z^d]\cap \Z_L^d.\ee
Clearly, $\|x-y\|_L:=\| [x-y]_L \|$ is a {periodic} distance on $\Z_L^d$ for any norm $\|\cdot\|$ on $\Z^d$. For definiteness, we use the $\ell^\infty$-norm in this paper, i.e. $\|x-y\|_L:=\|[x-y]_L\|_\infty$. In this paper, we consider the following class of $d$-dimensional random band matrices. 

\begin{assumption}[Random band matrix] \label{assmH}
	Fix any $d\in \N$. For $L\gg W\gg 1$ and $N:=L^d$, we assume that $ H\equiv H_{d,f,W,L}$ is an $N\times N$ complex Hermitian random matrix whose entries $(\mathrm{Re}\, h_{xy}, \mathrm{Im }\,  h_{xy}: x,y \in \Z_L^d)$ are independent Gaussian random variables (up to symmetry $H=H^\dagger$) such that  
	\be\label{bandcw0}
	\mathbb E h_{xy} = 0, \quad \E (\mathrm{Re}\, h_{xy})^2 =  \E (\im \, h_{xy})^2 = s_{xy}/2, \quad x , y \in \Z_L^d,
	\ee
	where the variances $s_{xy}$ satisfy that
	\be\label{sxyf}s_{xy}= f_{W,L}\left( [x-y]_L \right)\ee
	for a positive symmetric function $f_{W,L}$ satisfying Assumption \ref{var profile} below. We say that $H$ is a $d$-dimensional random band matrix with linear size $L$, band width $W$, and variance profile $f_{W,L}$. Denote the variance matrix by $S : = (s_{xy})_{x,y\in \Z_L^d}$, which is an $N \times N$ doubly stochastic symmetric matrix. 
\end{assumption}   

\begin{assumption}[Variance profile]\label{var profile}
	We assume that $f_{W,L}:\Z_L^d\to \mathbb R_+$ is a positive symmetric function on $\Z_L^d$ that can be expressed by the Fourier transform 
	\be\label{choicef}
	f_{W,L}(x):= \frac{1}{(2\pi)^d Z_{W,L}}\int \psi(Wp)e^{\ii p\cdot x} \dif p.  \ee
	Here, $\Z_{W,L}$ is the  normalization constant so that $\sum_{x\in \Z_L^d} f_{W,L}(x)=1$, and $\psi\in C^\infty(\R^d)$ is a symmetric smooth function independent of $W$ and $L$ and satisfies the following properties:
	\begin{itemize}
		\item[(i)] $\psi(0)=1$ and $\|\psi\|_\infty \le 1$;  
		\item[(ii)] $\psi(p)\le \max\{1 - c_\psi |p|^2 , 1-c_\psi  \}$ for a constant $c_\psi>0$;
		\item[(iii)]  $\psi$ is in the Schwartz space, i.e.,
		\be\label{schwarzpsi} \lim_{|p|\to \infty}(1+|p|)^{k}|\psi^{(l)}(p)| =0, \quad \text{for any }k,l\in \N.\ee
	\end{itemize}
\end{assumption}

Clearly, $f_{W,L}$ is of order $\OO(W^{-d})$ and decays faster than any polynomial, that is, for any fixed $k\in \N$, there exists a constant $C_k>0$ so that
\be\label{subpoly}
|f_{W,L}(x)|\le C_k W^{-d}\left( {\|x\|_L}/{W}\right)^{-k}.
\ee
In other words, $f_{W,L}$ is a Schwartz function of $x/W$. Hence, the variance profile $S$ defined in \eqref{sxyf}  has a banded structure, namely,  for any constants $\tau,D>0$,
\be\label{app compact f}
\mathbf 1_{|x-y|\ge W^{1+\tau}}|s_{xy}|\le W^{-D}.
\ee
Combining \eqref{schwarzpsi} and \eqref{subpoly} with the Poisson summation formula, we obtain that  
\be\label{bandcw1} 
Z_{W,L} =   \psi(0) + \OO(W^{-D})=1+ \OO(W^{-D}),
\ee
for any large constant $D>0$ as long as $L\ge W^{1+\e}$ for a constant $\e>0$.

Denote the eigenvalues and normalized eigenvectors of $H$ by $\{\lambda_\al\}$ and $\{u_\alpha\}$. Our first main result gives the bulk universality of random band matrices in dimensions $d\ge 7$. Define the $k$-point correlation function of $H$ by $$ \rho_H^{(k)}(\alpha_1,\alpha_2,...,\alpha_k):= \int_{\R^{N-k}}\rho_H^{(N)}(\alpha_1,\alpha_2,...,\alpha_N) \dif \alpha_{k+1}\cdots \dif \alpha_N,$$
where $\rho_H^{(N)}(\alpha_1,\alpha_2,...,\alpha_N)$ is the joint density of all unordered eigenvalues of $H$.
\begin{thm}[Bulk universality]
	\label{main thm-uni}
	Under Assumptions \ref{assmH} and \ref{var profile}, fix any $d\ge 7$ and an arbitrarily small constant $c_0>0$. If $W^{d+95} \geq L^{95 + c_0}$, then for any fixed $k>0$, the $k$-point correlation function of $H$ converges to that of \textrm{GUE} in the following sense. 
 For any $|E| \leq 2 - \kappa$ and smooth test function $\cal O$ with compact support, we have
 $$\lim_{N \to \infty} \int_{\R^k}\dif \boldsymbol{\alpha} ~ \cal O( \boldsymbol{\alpha}) \left\{\left(\rho_H^{(k)} - \rho_{\text{GUE}}^{(k)}\right)\left(E + \frac{\boldsymbol{\alpha}}{N}\right) \right\} =0\,.$$
\end{thm}

As we will explain at the end of this subsection, this result can be readily extended to non-Gaussian random band matrices after some straightforward technical modifications. Hence, this theorem indeed proves the important bulk universality conjecture for random band matrices, i.e., the local bulk eigenvalue statistics for large random matrices with independent entries are universal and do not depend on the particular distribution of matrix elements and the variance profile. The bulk universality was first proved for Wigner matrices \cite{BouErdYauYin2015,ErdPecRamSchYau2010,ErdSchYau2011,EYYbernoulli,ErdYauYin2012Univ,ErdYauYin2012Rig,TaoVu2011} and later extended to 1D random band matrices \cite{BouErdYauYin2017,PartI} and many other mean-field random matrix and random graph ensembles. We refer the reader to \cite{erdHos2017dynamical} for a more detailed review of the universality conjecture and related results in the literature.

As mentioned before, the bulk universality is a consequence of the QUE of bulk eigenvectors and the local law of the Green's function. We state the QUE as our second main result. 
Roughly speaking, it shows that \emph{all} bulk eigenvectors in dimensions $d \geq 12$ and \emph{most} bulk eigenvectors in dimensions $7\le d \le 11$ are asymptotically uniformly distributed on microscopic scales. In particular, it implies that the bulk eigenvectors are not localized in any small subset of volume $\oo(N)$, so their localization length must be $L$. 

\begin{thm}[Quantum unique ergodicity]\label{thm:QUE}
	Let $\kappa,\delta>0$ be arbitrary small constants. Suppose $W\ge L^\delta$ and Assumptions \ref{assmH} and \ref{var profile} hold. 
	\begin{itemize}
		\item[(i)] For $d \geq {12}$ and any $I_N \subset \Z_L^d$ with $|I_N| \geq \left(L/W\right)^{{10d}/(d-2)}$, the following event occurs with probability tending to one:
		\begin{equation}
			\label{eq:que}
			\frac{1}{|I_N|}\sum_{x \in I_N} (N|u_\alpha(x)|^2 -1)\to 0 \quad \text{uniformly for all $\alpha$ such that $ |\lambda_\alpha| \leq 2 - \kappa$.}
		\end{equation}
		
		\item[(ii)] For $d \geq {7}$, the following event occurs with probability tending to one for any small constant $\epsilon>0$ and $\ell\ge W^{-3/d} L^{\frac{5+d}{2d-2}}$:
		\begin{equation}
			\label{eq:weakque}
			\frac{1}{N}\,\left|\left\{ \alpha: |\lambda_\alpha|<2 - \kappa,\Big|\frac{1}{|I|}\sum_{x \in I} (N|u_\alpha(x)|^2 -1)\Big| \geq \epsilon \text{ for some $I \in \mathcal I$}\right\}\right| \to 0\,,
		\end{equation}
		where $\mathcal I: = \left\{I_{\mathbf k,\ell}: \mathbf k\in \Z^d, -L/\ell \le k_i \le L/\ell \right\}$ is a collection of boxes that covers $\Z^d_L$, with $I_{\mathbf k,\ell}:=\{[y]_L: y_i \in [(k_i-1)\ell/2  ,(k_i+1)\ell/2) \cap \Z^d\}$ (recall the notation \eqref{representativeL}).
	\end{itemize}
\end{thm}

Notice that the left-hand side of \eqref{eq:que} is a special case of $|\langle u_\al, \Pi u_\al\rangle|^2$ with the diagonal matrix $\Pi$ defined by $\Pi_{xx}=({N}/{|I_N|})\11_{x \in I_N} -1$. In part (ii), $\cal I$ can be any cover of $\Z_L^d$ consisting of subsets of size $\ell^d$ and with cardinality $|\cal I|=\OO(N/\ell^d)$.
The above probabilistic QUE was first proved for Wigner matrices \cite{BouYau2017} and later extended to 1D random band matrices \cite{BouErdYauYin2017,PartI} and many other types of mean-field random matrices and random graphs \cite{ALM_Levy,BHY2019,Benigni2020,Benigni2021,LP2021,principal,BouHuaYau2017,Marcinek_thesis}, to name a few. Recently, a stronger notion of QUE called the eigenstate thermalization hypothesis was also established for Wigner matrices \cite{CES_QUE1,CES_QUE3,CES_QUE2}. 

Both the proofs of bulk universality and QUE are crucially based on the following (essentially sharp) local law up to the scale $\im z\gg W^{-5}L^{5-d}$. 

\begin{thm}[Local law]\label{thm: locallaw}
	Let $\kappa,\delta,\delta_0\in (0,1)$ be arbitrary small constants. Under Assumptions \ref{assmH} and \ref{var profile}, fix any $d \geq 7$ and suppose $W\ge L^\delta$. For any constants $\tau,D>0$, we have the following estimate on $G(z)$ for $z=E+\ii \eta$ and all $x,y \in \Z_{L}^d$: 
	\begin{equation}
		\label{locallaw}
		\P\bigg(\sup_{|E|\le 2- \kappa}\sup_{\eta_*\le \eta\le 1} |G_{xy} (z) -m(z)\delta_{xy}|^2 \le  W^\tau \left(B_{xy}+ \frac{1}{N\eta}\right)\bigg) \ge 1- L^{-D} 
	\end{equation}
	for large enough $L$, where $\eta_*:=\etas$ and we denote 
	\be\label{defnBxy}B_{xy}:=W^{-2}\left(\|x-y\|_{ L}+W\right)^{-d+2}.\ee 
\end{thm}

An immediate corollary of this local law is the following delocalization estimate on bulk eigenvectors.

\begin{corollary}[Delocalization]\label{thm:supu}
	Under the assumptions of Theorem \ref{thm: locallaw}, for any constants $\tau, D>0$, we have that 
	\begin{equation}
		\P\Big(\sup_{\alpha: |\lambda_\alpha | \leq 2 - \kappa} \|u_\alpha\|_\infty^2 \leq W^{-5+\tau}L^{5-d}\Big) \ge 1- L^{-D} 
	\end{equation}
	for large enough $L$.
\end{corollary}
\begin{proof}
	Taking $E=\lambda_\al$ and $\eta=\etas$ with $\delta_0<\tau$ in the spectral decomposition 
	$$ \im G_{xx}(E+\ii \eta) = \sum_\alpha \frac{\eta}{|\lambda_\alpha - E|^2 + \eta^2} |u_{\alpha}(x)|^2$$
	gives that $|u_{\alpha}(x)|^2 \le \eta \im G_{xx}(E+\ii \eta).$ Then, using the local law \eqref{locallaw} and $\im m(z) \le 1$, we conclude the proof.
\end{proof}

Theorem \ref{thm: locallaw} improves the local law, Theorem 1.4, in \cite{BandI} from $d\ge 8$ to a lower dimension $d\ge 7$ and from $\eta\gg W^2/L^2$ to a  smaller scale $\eta\gg W^{-5}L^{5-d}$ which has the correct leading dependence in $L^{-d}$. As a consequence, it yields a better delocalization estimate than \cite{BandI}, which proved that $\|u_\al\|_\infty^2\le W^{2+\e}/L^2$ when $d\ge 8$. We believe that the local law \eqref{locallaw} should hold for all $\eta\gg L^{-d}$, which will lead to the following complete delocalization of bulk eigenvectors:
$$\mathbb P \Big(\sup_{\alpha: |\lambda_\alpha | \leq 2 - \kappa} \|u_\al\|_\infty^2 \le L^{-d +\tau}\Big)\ge  1-L^{-D} .
$$
Such complete delocalization was first proved for Wigner matrices in \cite{ESY_local,ESY1,ErdYauYin2012Univ,ErdYauYin2012Rig} and later extended to many other classes of mean-field random matrices and random graphs (see e.g., \cite{ALY_Levy,ADK2021,ADK2022,BHY2019,BKH2017,LP2020,isotropic,PartI,EKYY_ER1,HKM2019,HY_Regulard,LT2020,RV2016}).

As has been explained in \cite{BandI,BandII}, our proof can be readily adapted to non-Gaussian random band matrices after some technical modifications. More precisely, we have used Gaussian integration by parts in expanding resolvent entries (see Section \ref{sec_operations}), and this can be replaced by certain cumulant expansion formulas (see e.g., \cite[Proposition 3.1]{Cumulant1} and \cite[Section II]{Cumulant2}) for general distributions. For simplicity of presentation, we choose to stick to the complex Gaussian cases to avoid technical complexities associated with non-Gaussian distributions. The proof for the real Gaussian case is also similar to the complex case except that the number of terms will double whenever applying the Gaussian integration by parts formula to resolvents (which is due to the fact that $\E h_{xy}^2 = 0$ in the complex case but not in the real case).

\subsection{Some key ideas}\label{sec_idea}

As discussed before, one main challenge in proving the bulk universality for $H$ is how to compare $g(z)$ with $g_t(z)$ for some $t\gg W^{-5}L^{5-d}$. When $H$ is a mean-field Wigner matrix, we can find another Wigner matrix $H'$ so that the first four moments of $H$ and $H'+B_t$ are almost identical entrywise. Then, in the Green's function comparison, the matching moments imply that if we replace the entries of $H$ with those of $H'+B_t$ one by one, the Green's function will not change by much.
The moment matching condition, however, is no longer true for the setting of band matrices. Our proof is instead based on the key observation that the stability of the Green's function still holds, provided with QUE. As a simple example, we compare $\E \im g_t$ with $\E \im g$. With {It{\^o}'s} formula, we find that $\partial_t \E(\im g_t -\im g )$ is dominated by terms containing QUE factors $|\langle u_\al, \Pi u_\al\rangle|^2$ for a real diagonal $\Pi$ with zero trace. For example, one of these terms is
\begin{equation}
	-\frac{1}{2N}\E \sum_{a,b} (G_t)_{aa}(G^2_t)_{bb}\left(s_{ab} - \frac{1}{N}\right)  .
\end{equation}
On the right hand, each summand does not vanish unless $s_{ab} = N^{-1}$, but they are expected to be small after averaging due to  QUE:
\begin{equation}
	\sum_a (G_t)_{aa} \left(s_{ab} - \frac{1}{N}\right)  =   \sum_{\alpha} \frac{1}{\lambda_\alpha(t) - z} \sum_a |u_\alpha(a)|^2 \left(s_{ab} - \frac{1}{N}\right) \, ,
\end{equation}
where, fixing $b$, the diagonal matrix with $(a,a)$-entry given by $s_{ab} - N^{-1}$ has zero trace.


With the spectral decomposition of $\im G$, it is easy to see that proving the QUE amounts to bounding high moments of $\cal P=\tr{(\im G) \Pi (\im G) \Pi}$. A key ingredient for the proof of QUE is a \emph{complete expansion} of the $k$-moment of $\cal P$ for any fixed $k\in \N$. Roughly speaking, with a carefully designed expansion strategy,  we can express $\E \cal P^k$ 
as a sum of deterministic expressions, which satisfy some ``nice graphical properties". We can then control these deterministic expressions by exploring some cancellations in them.

Similar to \cite{BandI,BandII}, our general complete expansions are based on a simpler and more basic $T$-expansion of the $T$-variable. Roughly speaking, we will show that
\be\label{Teq_intro3}T_{xy} = |m|^2 \Theta_{xy}^{\circ} + |m|^2 \frac{\im  m}{ N\eta} + (\text{fluctuations and higher order errors})\, ,
\ee
where, roughly speaking, $\Theta_{xy}^{\circ}$ (cf. \eqref{Theta-S-circ})  is the Green's function of a random walk on $\Z^d_L$ with transition probability matrix $S$ and with zero mode removed, and the second term on the right-hand side gives the zero mode. The propagator \smash{$\Theta^\circ_{xy}$} describes the diffusive behavior of the random walk and is of order $B_{xy}$ (cf. \eqref{thetaxy}). Hence, when $\im z \gg W^2/L^{2}$, the zero mode is negligible and can be absorbed into the diffusive term in \cite{BandI,BandII}. On the other hand, when $\im z$ gets below $W^2/L^{2}$, the zero mode will play a crucial role and make some expressions in \cite{BandI,BandII} diverge. An important step in the $T$-expansion of this paper is to separate out the zero mode so that all expressions are under control as long as $\eta\gg W^{-5}L^{5-d}$. This key idea will be explained in Section \ref{sec_second_T}. 
If we pick out the zero mode $\frac{\im m}{N\eta}$ from $T_{xy}$, then in the expansions of $\E \cal P^k$, the ``leading terms" will contain the vanishing factor $\sum_{x}\Pi_{xx} \frac{\im m}{N\eta}=0.$
This type of cancellation is crucial in the proof of QUE and is proved rigorously by a careful analysis of the deterministic expressions from complete expansions.


\vspace{5pt}

The rest of this paper is organized as follows. 
In Section \ref{sec_T_exp}, we introduce some graphical notations representing expressions of resolvent entries and the concept of $\selfs$. We then use them to define the $T$-expansion. In Section \ref{sec_complete_T}, we define the concept of complete $T$-expansions and complete expansions of general graphs. The main results, Theorems \ref{main thm-uni}, \ref{thm:QUE} and \ref{thm: locallaw}, will be proved in Section \ref{sec_pf_main} based on the tools developed in Sections \ref{sec_T_exp} and \ref{sec_complete_T}. The proof of the local law requires a ``sum zero property" for $\selfs$, whose proof is presented in Section \ref{sec_pf_sumzero}. Two key lemmas in Section \ref{sec_complete_T} regarding bounding complete expansions of general graphs will be proved in Section \ref{sec_pf_complete}. Finally, constructions of the $T$-expansion, the $\nonuni$, and complete expansions of general graphs are included in Appendix \ref{sec notation}--\ref{sec_pf_mG}.

\section{$T$-expansion}\label{sec_T_exp}

Similar to \cite{BandI,BandII}, our proofs in this paper are crucially based on an important tool---the $T$-expansion up to arbitrarily high order. However, in order to decrease $\eta$ from $\eta\gg W^2/L^2$ in \cite{BandI,BandII} to $\eta\ge \etas$, we need to derive a different $T$-expansion with zero mode removed. We now use the second order $T$-expansion to explain this key idea.

\subsection{Second order $T$-expansion}\label{sec_second_T}

Our $T$-expansion will be formulated in terms of the following $T$-variables with three subscripts: 
\be\label{general_T}
T_{x,yy'}:=|m|^2\sum_\al s_{x\al}G_{\al y}\overline G_{\al y'},\quad \text{and}\quad T_{yy',x}:=|m|^2\sum_\al G_{ y\al }\overline G_{y' \al}s_{\al x}. 
\ee
When $y=y'$, we have the conventional $T$-variable $T_{xy}\equiv T_{x,yy}$. 
In this subsection, we derive the second order $T$-expansion of $T_{x,yy'}$ using the expansion
\be\label{GmH}
G = - \frac 1 { z+m} + \frac 1 { z+m} (H+m) G \quad \Rightarrow \quad G-   m = - m  (H+m) G .
\ee
The formula \eqref{GmH} is derived easily from the definition of $G$ and the self-consistent equation $(z+ m) m = -1$.   
Define $\E_x$ as the partial expectation with respect to the $x$-th row and column of $H$, i.e., $\E_x(\cdot) := \E(\cdot|H^{(x)}),$	where $H^{(x)}$ denotes the $(N-1)\times(N-1)$ minor of $H$ obtained by removing the $x$-th row and column. 
For simplicity, in this paper, we will use the notations 
$$P_x :=\E_x , \quad Q_x := 1-\E_x.$$
Taking the partial expectation of $HG$ and applying Gaussian integration by parts to the $H$ entries, we can obtain the following lemma. The corresponding $T$-expansion of $T_{yy',x}$ can be obtained by considering the transposition of $T_{x,yy'}$. 

\begin{lemma}[Second order $T$-expansion]
	\label{2ndExp}
	Under Assumption \ref{assmH}, we have 
	\begin{equation}
		\label{eq:2nd}
		\begin{split}
			T_{\fa,\fb_1 \fb_2}&= m  \Theta^{\circ}_{\fa \fb_1}\overline G_{\fb_1\fb_2} + \frac{|m|^2}{2\ii N\eta} (G_{\fb_2 \fb_1} - \overline G_{\fb_1 \fb_2}) + \sum_x \Theta^{\circ}_{\fa x}\left[ \mathcal A^{(>2)}_{x,\fb_1\fb_2}  + \mathcal Q^{(2)}_{x,\fb_1\fb_2}  \right]\,,
		\end{split}
	\end{equation}
	for $z=E+\ii \eta$ and $\fa,\fb_1,\fb_2\in \Z_L^d$, where 
	\be\label{Theta-S-circ}\Theta^\circ := \frac{|m|^2S^\circ}{1 - |m|^2 S^\circ},\quad S^{\circ}:= P^\perp SP^\perp,
	\ee
	with $P^\perp:= I - \mathbf e\mathbf e^\top$ and $\mathbf e:= {N}^{-1/2}(1,\ldots, 1)^\top$, and 
	\begin{align*}
		\mathcal A^{(>2)}_{x,\fb_1\fb_2}&:= m \sum_{y} s_{xy}   (G_{yy}-m) G_{x \fb_1} \overline G_{x \fb_2} \nonumber+m \sum_{y} s_{xy} ( \overline G_{xx} -\overline m)G_{y\fb_1}\overline G_{y\fb_2}\,,\\
		\mathcal Q^{(2)}_{x,\fb_1\fb_2}&:= Q_x \left(G_{x\fb_1} \overline G_{x\fb_2} \right)- m  \delta_{x\fb_1}  Q_{\fb_1}\left(  \overline G_{\fb_1\fb_2} \right) - m\sum_{y}  s_{xy}  Q_x \left[   (G_{yy}-m) G_{x\fb_1} \overline G_{x\fb_2} \right]  \\
		& - m\sum_{y}  s_{xy}  Q_x\left(    \overline G_{xx} G_{y\fb_1}\overline G_{y\fb_2}  \right)\,.
	\end{align*}
\end{lemma}

\begin{proof}
	Note that $\mathbf e$ is the Perron-Frobenius eigenvector of $S$ with eigenvalue 1. Hence, we have $s_{xy}^{\circ} = s_{x y}- N^{-1}$ and $\Theta^\circ\mathbf e=0$. Furthermore, we recall the following classical Ward's identities, which are derived from a simple algebraic calculation:
	\be\label{eq_Ward0}
	\sum_x \overline G_{xy'} G_{xy} = \frac{G_{y'y} -\overline{G_{yy'} }}{2\ii \eta},\quad \sum_x \overline G_{y' x} G_{yx} = \frac{G_{yy'} -\overline{G_{y'y} }}{2\ii \eta}.
	\ee
	As a special case, if $y=y'$, we have
	\be\label{eq_Ward}
	\sum_x |G_{xy}( z)|^2 =\sum_x |G_{yx}( z)|^2 = \frac{\im G_{yy}(z) }{ \eta}.
	\ee

	Now, using \eqref{eq_Ward0}, we get that
	\begin{align}
		T_{\fa,\fb_1\fb_2}& =  |m|^2\sum_{x} s_{\fa x}G_{x\fb_1} \overline G_{x\fb_2} =  |m|^2\sum_{x} s_{\fa x }^{\circ}G_{x \fb_1} \overline G_{x \fb_2} + |m|^2\frac{1}{N}\sum_{x} G_{x \fb_1} \overline G_{x \fb_2}\nonumber\\
		&= |m|^2\sum_{x} s_{\fa x }^{\circ}G_{x \fb_1} \overline G_{x \fb_2} + \frac{|m|^2}{2\ii N\eta} (G_{\fb_2\fb_1} - \overline G_{\fb_1\fb_2}) \nonumber\\
		&= T_{\fa,\fb_1\fb_2}^\circ+  \frac{|m|^2}{2\ii N\eta} (G_{\fb_2\fb_1} - \overline G_{\fb_1\fb_2})\,, \label{eq:TTC}	 
	\end{align}
	where 
	\be\label{Tcirc}
	T^{\circ}_{\fa,\fb_1\fb_2}:=\sum_{x}P^{\perp}_{\fa x}T_{x,\fb_1\fb_2}  = |m|^2\sum_{x} s_{\fa x }^{\circ}G_{x \fb_1} \overline G_{x \fb_2}. 
	\ee
	Using \eqref{Theta-S-circ} and \eqref{GmH}, we can write that 
	\begin{align*}
		 T^{\circ}_{\fa,\fb_1\fb_2} & = \sum_{x}[\Theta^\circ_{\fa x} -|m|^2(\Theta^\circ S)_{\fa x}] G_{x \fb_1}\overline G_{x \fb_2} \\
		&= \sum_{x}\Theta^\circ_{\fa x}  P_x\left[(m\delta_{x\fb_1}-m(HG)_{x \fb_1}-m^2 G_{x\fb_1})\overline G_{x \fb_2}\right]   \\
		&+ \sum_{x}\Theta^\circ_{\fa x}  Q_x\left(G_{x \fb_1}\overline G_{x \fb_2}\right) -  |m|^2\sum_{x}(\Theta^\circ S)_{\fa x} G_{x \fb_1}\overline G_{x \fb_2}.
	\end{align*}
	Then, applying Gaussian integration by parts to $P_x[(HG)_{x \fb_1} \overline G_{x \fb_2}]$, we can get that 
	\begin{align*}
		\sum_{x}[\Theta^\circ_{\fa x} -|m|^2(\Theta^\circ S)_{\fa x} ]G_{x \fb_1}\overline G_{x \fb_2} = m  \Theta^\circ_{x\fb_1}   \overline G_{\fb_1\fb_2}   + \mathcal A^{(>2)}_{x,\fb_1\fb_2}+\mathcal Q^{(2)}_{x,\fb_1\fb_2}.
	\end{align*} 
	Since the calculation is exactly the same as the one for Lemma 2.4 of \cite{BandI}, we omit the details. This completes the proof of Lemma \ref{2ndExp}.  
\end{proof}

Compared with the second order $T$-expansion in Lemma 2.5 of \cite{BandI}, there are two differences: the second term on the RHS of \eqref{eq:2nd} is new and the other two terms use $\Theta^{\circ}_{\fa x}$ instead of $\Theta_{\fa x}$ as in \cite{BandI}, where $\Theta := {|m|^2S}/({1 - |m|^2 S}).$ Note that $\Theta^\circ$ differs from $\Theta$ by a zero mode, i.e., 
\be\label{Theta-Thetacirc}
\Theta^\circ_{xy}=(P^\perp \Theta P^\perp)_{xy}= \Theta_{xy}-\frac{|m|^2}{N(1-|m|^2)}= \Theta_{xy}-\frac{\im m}{N\eta},
\ee
where we used the identity 
\be\label{insert m^2}\frac{|m(z)|^2}{1-|m(z)|^2}  = \frac{ \im m(z)}{\eta} ,\ee
which can be derived by taking the imaginary part of the equation $z=-m(z)-m^{-1}(z)$. 
It is well-known that for $z=E+\ii \eta$ with $E\in (-2+\kappa,2-\kappa)$ and $\eta >0 $,
\be\label{thetaxy}
|\Theta_{xy}^{\circ}(z)| \le   \frac{W^\tau \mathbf 1_{|x-y|\le   \eta^{-1/2}W^{1+\tau}}}{W^2\langle x-y\rangle^{d-2}}   + \frac{1}{ \langle x-y\rangle^{D}}  \le W^{\tau} B_{xy},
\ee
for any constants $\tau, D>0$, where we recall that $B_{xy}$ was defined in \eqref{defnBxy}. For simplicity, here and throughout the rest of this paper, we abbreviate
\be\label{Japanesebracket} |x-y|\equiv \|x-y\|_L,\quad \langle x-y \rangle \equiv \|x-y\|_L + W.\ee
The reader can refer to e.g., \cite{delocal,Band1D_III} for a proof of \eqref{thetaxy}. (For example, the proof in \cite{Band1D_III} was written for $\Theta(z)$ with $\eta\ge W^2/L^2$, in which case we can use $\|S\|\le 1$ and $1-|m(z)|^2\gtrsim W^2/L^2$. For $\Theta^{\circ}(z)$ with smaller $\eta$, the same proof still works except that we will use $|m(z)|\le 1$ and the spectral gap of $S$ at $1$, i.e., $1-\|S^{\circ}\|\gtrsim W^2/L^2$.) With \eqref{thetaxy} and \eqref{Theta-Thetacirc}, we get that 
\be\label{thetaxy2}
\Theta_{xy}(z)\le W^{\tau} B_{xy} + \frac{\im m}{N\eta}.
\ee

To explain heuristically why the new second order $T$-expansion can be applied to smaller $\eta$ and what trouble the $T$-expansion in \cite{BandI} may have, we take $\fb_1=\fb_2=\fa$ and look at the \smash{$\sum_x \Theta^{\circ}_{\fa x} \mathcal A^{(>2)}_{x,\fa\fa}$} term in \eqref{eq:2nd}. Suppose we already know that the local law \eqref{locallaw} holds. Then, using \eqref{thetaxy}, we can bound that 
\be\nonumber \sum_x \Theta^{\circ}_{\fa x} \mathcal A^{(>2)}_{x,\fa\fa} \prec \sum_x B_{\fa x}\left(B_{x\fa} + \frac{1}{N\eta} \right) \lesssim W^{-d} + \frac{1}{N\eta}\frac{L^2}{W^2}\ll 1,\ee
as long as \smash{$\eta\gg W^{-2}L^{2-d}$}. Here, ``$\prec$" denotes stochastic domination, which will be defined in Definition \ref{stoch_domination} below. On the other hand, for the $T$-expansion in \cite{BandI}, we need to control $\sum_x \Theta_{\fa x} \mathcal A^{(>2)}_{x,\fa\fa}$. Using \eqref{thetaxy2}, we can bound that 
$$ \sum_x \Theta_{\fa x} \mathcal A^{(>2)}_{x,\fa\fa} \prec \sum_x  \left(B_{x\fa} + \frac{1}{N\eta} \right)^2 \lesssim W^{-d} + \frac{1}{N\eta}\frac{L^2}{W^2} + \frac{1}{N\eta^2}, $$
where the RHS diverges when $\eta\gg L^{-d/2}$. Hence, the $T$-expansion in \cite{BandI} does not work well for $\eta\ll L^{-d/2}$. Of course,  in the proof, we do not have the local law \eqref{locallaw} a priori, and we can only use a worse continuity estimate for $G$ entries. Hence, in the current work, we cannot reach the level $W^{-2} L^{2-d}$ yet. 

In this paper, we use the following notion of stochastic domination introduced in \cite{EKY_Average}. 

\begin{defn}[Stochastic domination and high probability event]\label{stoch_domination}
	{\rm{(i)}} Let
	\[\xi=\left(\xi^{(W)}(u):W\in\mathbb N, u\in U^{(W)}\right),\hskip 10pt \zeta=\left(\zeta^{(W)}(u):W\in\mathbb N, u\in U^{(W)}\right),\]
	be two families of non-negative random variables, where $U^{(W)}$ is a possibly $W$-dependent parameter set. We say $\xi$ is stochastically dominated by $\zeta$, uniformly in $u$, if for any fixed (small) $\tau>0$ and (large) $D>0$, 
	\[\mathbb P\bigg[\bigcup_{u\in U^{(W)}}\left\{\xi^{(W)}(u)>W^\tau\zeta^{(W)}(u)\right\}\bigg]\le W^{-D}\]
	for large enough $W\ge W_0(\tau, D)$, and we will use the notation $\xi\prec\zeta$. 
	If for some complex family $\xi$ we have $|\xi|\prec\zeta$, then we will also write $\xi \prec \zeta$ or $\xi=\OO_\prec(\zeta)$. 
	
	\vspace{5pt}
	\noindent {\rm{(ii)}} As a convention, for two deterministic non-negative quantities $\xi$ and $\zeta$, we will write $\xi\prec\zeta$ if and only if $\xi\le W^\tau \zeta$ for any constant $\tau>0$. 
	
	
	\vspace{5pt}
	\noindent {\rm{(iii)}} We say that an event $\Xi$ holds with high probability (w.h.p.) if for any constant $D>0$, $\mathbb P(\Xi)\ge 1- W^{-D}$ for large enough $W$. More generally, we say that an event $\Omega$ holds $w.h.p.$ in $\Xi$ if for any constant $D>0$,
	$\P( \Xi\setminus \Omega)\le W^{-D}$ for large enough $W$.
\end{defn}

\subsection{Graphs, scaling order and doubly connected property}

Similar to some previous works (e.g., \cite{EKY_Average,BandI,BandII,Band1D_III}) on random matrices, we will organize our proofs using graphs. In this subsection, we introduce the basic concepts of atomic graphs,  molecular graphs, and the doubly connected property. Many graphical notations used in this paper have been defined in \cite{BandI,BandII}, but we repeat them for the convenience of the reader.

Our graphs will consist of matrix indices as vertices and matrix entries as edges. In particular, the entries of $S$, $\Theta^{\circ}$ and $G$ will be represented by different types of edges. In addition, we also have edges representing the following two deterministic matrices 
\be\label{Thetapm2}
S^+(z):=\frac{m^2(z) S}{1-m^2(z) S}, \quad S^-(z):=\overline S^+(z),  
\ee
which satisfy the following estimate \eqref{S+xy}. It is a folklore result and we omit its proof (a formal proof for the $d=1$ case is given in equation (4.21) of \cite{PartII}). 
\begin{lemma} \label{lem deter}
	Suppose Assumptions \ref{assmH} and \ref{var profile} hold and $z=E+\ii \eta$ with $E\in (-2+\kappa,2-\kappa)$ for a constant $\kappa>0$. Then, for any constants $\tau, D>0$, we have that
	\be\label{S+xy}|S^\pm_{xy}(z)| \lesssim    W^{-d}\mathbf 1_{|x-y|\le W^{1+\tau}} +  \langle x-y\rangle^{-D} . \ee 
\end{lemma}

We first introduce the most basic atomic graphs. 

\begin{defn}[Atomic graphs] \label{def_graph1} 
	
	Given a standard oriented graph with vertices and edges, we assign the following structures and call the resulting graph an atomic graph.  
	
	\begin{itemize}
		%
		%
		
		\item {\bf Atoms:} We will call the vertices atoms. Every graph has some external atoms and internal atoms. 
		The external atoms represent external indices whose values are fixed, while internal atoms represent summation indices that will be summed over. 
		
		
		\item \noindent{\bf Weights}: A regular weight on an atom $x $ represents a $G_{xx}$ factor, drawn as a blue solid $\Delta$, or a $\overline G_{xx}$ factor, drawn as a red solid $\Delta$. 
		A light weight on atom $x$ represents a $G_{xx}-m$ factor, drawn as a blue hollow $\Delta$, or a $\overline G_{xx}-\overline m$ factor, drawn as a red hollow $\Delta$. 
	Note that every $\Delta$ actually represents a self-loop of $G$ or $(G-m)$ edge (see the definition of solid edges below).
		
		\item {\bf Edges:} The edges are divided into the following types. 
		
		\begin{enumerate}

			\item{\bf  Solid edges:} 
			\begin{itemize}
				\item a blue oriented edge from atom $x$ to atom $y$ represents a $G_{xy}$ factor; 
				\item a red oriented edge from atom $x$ to atom $y$ represents a $\overline G_{xy}$ factor.
			\end{itemize}

			\item {\bf Waved edges:}
			\begin{itemize}
				\item a neutral waved edge between atoms $x$ and $y$ represents an $s_{xy}$ factor; 
				
				\item a positive blue waved edge 
				between atoms $x$ and $y$ represents an $S^+_{xy}$ factor;
				
				\item a negative red waved edge 
				between atoms $x$ and $y$ represents an $S^-_{xy}$ factor.
			\end{itemize}

			\item {\bf $\Dashed$ edges:} A $\dashed$ edge between atoms $x$ and $y$ represents a ${\Theta}_{xy}^{\circ}$ factor. We draw it as a black double-line edge between atoms $x$ and $y$. 
			
			\item {\bf Free edges:} A purple solid edge between atoms $x$ and $y$ represents a $(N \eta)^{-1}$ factor, and we call it a free edge. 
			
			\item {\bf Dotted edges:} A dotted edge connecting atoms $\al$ and $\beta$  represents a factor $\mathbf 1_{\al=\beta}\equiv \delta_{\al\beta}$, and a $\times$-dotted edge represents a factor $\mathbf 1_{\al\ne \beta} \equiv  1-\delta_{\al\beta} $. There is at most one dotted or $\times$-dotted edge between each pair of atoms. 
			
			
		\end{enumerate}
		The orientations of non-solid edges do not matter. Edges between internal atoms are called {internal edges}, while edges with at least one end at an external atom are called {external edges}.

		\item{\bf $P$ and $Q$ labels:} A solid edges or a weight may have a label $P_x$ or $Q_x$ for some atom $x$ in the graph. Moreover, every edge or weight can have at most one $P$ or $Q$ label.

		\item{\bf Coefficients:} There is a coefficient 
		associated with each graph. Unless otherwise specified, the coefficient is of order $\OO(1)$. 
		
		
	\end{itemize}
	
\end{defn}

\begin{defn}[Subgraphs]\label{def_sub}
	A graph $\cal G_1$ is said to be a subgraph of $\cal G_2$, denoted by $\cal G_1\subset \cal G_2$, if every graphical component (except the coefficient) of $\cal G_1$ is also in $\cal G_2$. Moreover, $\cal G_1 $ is a proper subgraph of $\cal G_2$ if  $\cal G_1\subset \cal G_2$ and $\cal G_1\ne \cal G_2$. Given a subset $\cal S$ of atoms in a graph $\cal G$, the subgraph $\cal G|_{\cal S}$ induced on $\cal S$ refers to the subgraph of $\cal G$ with atoms in $\cal S$ as vertices, the edges between these atoms, and the weights on these atoms. Given a subgraph $\cal G$, its closure $\overline{\cal G}$ refers to $\cal G$ plus its external edges. 
\end{defn}

Along the proof, we will introduce some other types of weights and edges. To each graph, we assign a \emph{value}  as follows.

\begin{defn}[Values of graphs]\label{ValG} 
	Given an atomic graph $\mathcal G$, we define its value as an expression obtained as follows. We first take the product of all the edges, all the weights, and the coefficient of the graph $\cal G$. Then, for the edges and weights with the same $P_x$ or $Q_x$ label, we apply $P_x$ or $Q_x$ to their product. Finally, we sum over all the internal indices represented by the internal atoms. The values of the external indices are fixed by their given values. The value of a linear combination of graphs  $\sum_i c_i \cal G_i$ is naturally defined as the linear combination of the graph values of $\cal G_i$.
\end{defn}

For simplicity, throughout this paper, we will always abuse the notation by identifying a graph (a geometric object) with its value (an analytic expression). In this sense, noticing that free edges represent $(N\eta)^{-1}$ factors without indices, two graphs are equivalent if they have the same number of free edges and all other graph components are the same. In other words, we can move a free edge freely to another place without changing the graph. 

\begin{defn}[Normal graphs]  \label{defnlvl0}  
	We say an atomic graph $\cal G$ is \emph{normal} if it satisfies the following properties:
	\begin{itemize}
		\item[(i)] it contains at most $\OO(1)$ many atoms and edges;
		\item[(ii)] all internal atoms are connected together or to external atoms through paths of waved, $\dashed$ or dotted edges;
		\item[(iii)] there are no dotted edges between internal atoms;
		\item[(iv)] every pair of atoms $\al$ and $\beta$ in the graph are connected by a $\times$-dotted edge \emph{if and only if} they are connected by a $G$ edge.
	\end{itemize}
\end{defn}

All graphs appearing in our proof are normal (after some simple graph operations related to dotted edges, see Definition \ref{dot_operation}). By this definition, every $G$ edge in a normal graph is off-diagonal, while all diagonal $G$ factors will be represented by weights. Given a normal graph, we define its scaling order as follows. 

\begin{defn} [Scaling order] \label{def scaling}
	Given a normal graph $\cal G$, we define its scaling order as 
	\begin{align}
		\mathrm{ord}(\cal G): = &\ \#\{\text{off-diagonal }  G  \text{ edges}\}  + 2\#\{ \text{waved, free or diffusive edges}\} \nonumber\\
		& + \#\{\text{light weights}\}   - 2\left[ \#\{\text{internal atoms}\}- \#\{\text{dotted edges}\} \right]. \label{eq_deforderrandom2}
	\end{align}
	Here, every dotted edge in a normal graph means that an internal atom is equal to an external atom, so we lose one free summation index. The concept of scaling order can also be defined for subgraphs. 
\end{defn}  

In this paper, whenever we say the order of a graph, we are referring to its scaling order. 
We arrange the graphs in $T$-expansions according to their scaling orders, and an $n$-th order $T$-expansion indicates that the ``error term" contains graphs of scaling order $>n$. 
 
The motivation behind \Cref{def scaling} is as follows: consider the special case with $\eta\ge W^2/L^2$ and $W\ge cL$, i.e., $H$ is a generalized Wigner matrix. First, each summation over an internal atom leads to a $\OO(W^d)$ factor. Second, by \eqref{thetaxy} and \eqref{S+xy}, every waved or diffusive edge is of order $\OO_\prec(W^{-d})$. Third, every free edge provides a $(N\eta)^{-1}=\OO(W^{-d})$ factor. Finally, if we know that $|G_{xy}-m\delta_{xy}|\prec W^{-d/2}$ (recall the local law \eqref{locallaw}), then every off-diagonal $G$ edge or light weight is bounded by $\OO_\prec(W^{-d/2})$. From the above arguments, we see that 
$$|\cal G|\prec W^{-\text{ord}(\cal G)\cdot d/2}.$$ 
However, if we have $L\gg W$, the scaling order does not imply the real size of the graph value directly. In order to establish such a connection for a graph, it needs to be \emph{doubly connected}. To define the doubly connected property, we first define molecular graphs.

	Our graphs have a two-level structure, that is, a local structure varying on scales of order $W$, and a global structure varying on scales up to $L$. To explain this, we introduce the following concept of molecules. 
	
	
	\begin{defn}[Molecules]\label{def_poly}
		We partition the set of all atoms into a disjoint union of subsets $\{\text{all atoms}\}=\cup_j \cal M_j$, where every $\cal M_j$ is called a molecule. Two atoms belong to the same molecule if and only if they are connected by a path of neutral/plus/minus \emph{waved edges} and \emph{dotted edges}. Every molecule containing at least one external atom is called an external molecule.  
	\end{defn}
	
	By \eqref{subpoly} and \eqref{S+xy}, if two atoms $x$ and $y$ are in the same molecule, then we essentially have $|x-y|\le W^{1+\tau}$ up to a negligible error $\OO(W^{-D})$. Given an atomic graph, we will call the subgraph inside a molecule the {\it local structure} of this molecule. On the other hand, the {\it global structure} of an atomic graph refers to its \emph{molecular graph}, which is the quotient graph with each molecule regarded as a vertex. 
	
	\begin{defn}[Molecular graphs] \label{def moleg}  
		Molecular graphs are graphs consisting of 
		\begin{itemize}
			\item external and internal molecules;
			\item solid,  $\dashed$ and free edges between molecules.
		\end{itemize}
		Given an atomic graph $\cal G$, we define its molecular graph, denoted by $ \cal G_{\cal M}$, as follows:
		\begin{itemize}
			\item merge all atoms in the same molecule and represent them by a vertex;
			
			\item keep solid, $\dashed$ and free edges between molecules;

			\item discard all the other components in $\cal G$ (including weights, dotted or $\times$-dotted edges, and edges inside molecules).
		\end{itemize}
	\end{defn}
	
	In this paper, molecular graphs are used solely to help with the analysis of graph structures, while all graph expansions are applied to atomic graphs only. 
	The following \emph{doubly connected property} is a crucial property defined using molecular graphs. It will allow us to establish a direct connection between the scaling order of a graph and a bound on its value. All graphs in the $T$-expansion will satisfy this property. 
	
	\begin{defn} [Doubly connected property]\label{def 2net}  
		An atomic graph $\cal G$ without external molecules is said to be doubly connected if its molecular graph $\cal G_{\cal M}$ satisfies the following property. There exists a collection, say $\cal B_{black}$, of $\dashed$ edges, and another collection, say $\cal B_{blue}$, of blue solid, $\dashed$ and free edges such that 
		\begin{itemize}
			\item[(a)] $\cal B_{black}\cap \cal B_{blue}=\emptyset$,
			\item[(b)] each of $\cal B_{black}$ and $\cal B_{blue}$ contains a spanning tree that connects all molecules. 
		\end{itemize}
		(Note that red solid edges are not used in either collection.) For simplicity of notations, we will call the edges in $\cal B_{black}$ as black edges and the edges in $\cal B_{blue}$ as blue edges. Correspondingly, $\cal B_{black}$ and $\cal B_{blue}$ are referred to as black net and blue net, respectively.
		
		An atomic graph $\cal G$ with external molecules is called doubly connected if its subgraph with all external molecules removed is doubly connected.
	\end{defn}

	
	\subsection{$\Selfs$ and renormalized $\dashed$ edges}\label{sec_selfc}

	In higher order $T$-expansions, new types of $\dashed$ edges will also appear. They have the same behavior as $\Theta^{\circ}$, but ``renormalized by $\selfs$" in the following sense: 
	\be\label{theta_renormal1}
	\Theta^{(n)}:= \frac{1}{1-\Theta^{\circ} \wtSdelta^{(n)} }\Theta^{\circ},\quad \text{where}\quad  \wtSdelta^{(n)}(z):=\sum_{2l=4}^n \Sele_{2l}(z),
	\ee
 where each $\Sele_{2l}(z)$ is a deterministic matrix called \emph{$\self$ with scaling order $2l$} that will be defined in \Cref{collection elements}. Then, we perform a truncated Taylor expansion of $\Theta^{(n)}$ and pick out the $k$-th order part $\Theta^{(n)}_k$. More precisely, we define $\Theta^{(n)}_2:=\Theta^{\circ}$ and 
	\be\label{chain S2k}
	\Theta^{(n)}_k:=\Theta^{\circ}\Sele_{k}\Theta^{\circ} + \sum_{ l =2}^k \sum\limits_{ \mathbf k=(k_1,\cdots, k_l) \in \Omega_{n,k}^{(l)} } \Theta^{\circ}\Sele_{k_1}\Theta^{\circ}  \Sele_{k_2}\Theta^{\circ} \cdots \Theta^{\circ}  \Sele_{k_l}\Theta^{\circ},\quad k\ge 4, 
	\ee
	where $\Omega_{n,k}^{(l)} \subset \N^l$ is the subset of vectors $\mathbf k$ satisfying that 
	\be\label{omegakl}
	4\le k_i \le n,\quad \text{ and }\quad  \sum_{i=1}^l  k_i  -2(l-1)=k. 
	\ee
We will discuss the meaning of condition \eqref{omegakl} below \Cref{collection elements}.

	\begin{defn}[$\Selfs$]\label{collection elements}
		
		Fix any $l\in \N$. $\Sele_{2l}(z)$ is a deterministic matrix 
		satisfying the following properties.
		
		\begin{itemize}
			\item[(i)] For any $x,y \in \Z_L^d$, $  (\Selek_{2l})_{xy}$ is a sum of $\OO(1)$ many deterministic graphs of scaling order $2l$, with external atoms $x$ and $y$. 
			These graphs consist of waved edges, dotted edges (there may be a dotted edge between $x$ and $y$), diffusive edges, and \smash{$\Theta^{(n)}_k$} edges with $2\le k \le n \le 2l-1$, but do not contain free or $\Theta^{(n)}$ edges with $n \geq 3$. 
			
			
			\item[(ii)] $ \Selek_{2l} (z)$ satisfies the following properties 
			for $z= E+ \ii \eta$ with $|E|\le  2- \kappa$ and $\eta \in [ W^{-5}L^{5-d}, 1]$:
			\be\label{two_properties0}
			\Selek_{2l} (x, x+a) =  \Selek_{2l} (0,a), \quad   \Selek_{2l} (0, a) = \Selek_{2l}(0,-a), \quad \forall \ x,a\in \Z_L^d,
			\ee 
			\be\label{4th_property0}
			\left|  (\Selek_{2l})_{0x}(z) \right| \prec  W^{-(l-2)d} B_{0x}^2 , \quad \forall \ x\in \Z_L^d,
			\ee
			\be\label{3rd_property0}
			\Big|\sum_{x\in \Z_L^d} (\Selek_{2l})_{0x}(z)\Big|   \prec   \left[ \left(\eta + \frac {W^{2d-6}}  {L^{2d-6}}  \right) W^{-d}+   \frac{1}{N} \frac{L^2}{W^2}  \right] W^{-(l-2)d} . 
			\ee
			
		\end{itemize} 
	\end{defn}

	By Definition \ref{def scaling}, the scaling order of a deterministic graph can only be even. Moreover, every nontrivial $\self$ $\Selek_{l}$ in this paper has scaling order $\ge 4$. Hence, as a convention, we set 
	\be\label{trivial conv}
	\Sele_1=\Sele_2=\Sele_3=0, \quad \text{and}\quad \Sele_{2l+1}:=0,\quad l\in \N.
	\ee
	 We will call $\Theta^{(n)} $ in \eqref{theta_renormal1} the \emph{$n$-th order renormalized $\dashed$ matrix}. Note the second condition in \eqref{omegakl} means that the subgraph $(\Theta^{\circ}\Sele_{k_1}\Theta \cdots \Theta  \Sele_{k_l}\Theta^{\circ})_{xy}$ has scaling order $k$, so \smash{$\Theta^{(n)}_k$} is the sum of all graphs in the truncated Taylor expansion of $\Theta^{(n)}$ that have scaling order $k$. We also remark that Definition \ref{collection elements} (i) is a recursive condition, that is, after defining lower order $\selfs$, we then use them to define \smash{$\Theta^{(n)}_k$}, $2\le k \le n \le 2l-1$, which are further used to define $\cal E_{2l}$.
	
	We will regard $\Theta^{(n)}$ and $\Theta^{(n)}_k$ as diffusive edges with labels. 
	
	\begin{defn}[Labeled $\dashed$ edges] \label{def_graph comp} 
		We represent $\Theta^{(n)}_{xy}$ and $(\Theta^{(n)}_k)_{xy}$ by $\dashed$ edges between atoms $x$ and $y$ with labels $(n)$ and $(n;k)$, respectively. In graphs, every labeled $\dashed$ edge is drawn as a double-line edge with a label but without any internal structure. Moreover, when calculating scaling orders, a $\Theta^{(n)}$ edge is counted as an edge of scaling order 2, and a $\Theta^{(n)}_k$ edge is counted as an edge of scaling order $k$.
	\end{defn}
	
	Using the properties \eqref{two_properties0}--\eqref{3rd_property0}, we can show that the labeled $\dashed$ edges defined above satisfy a similar bound as $\Theta^\circ$ (recall \eqref{thetaxy}). 
	
	
	
	\begin{lemma}\label{lem redundantagain}
		Fix $d\ge 6$ and $n\in \N$. Let $\Selek_{2l}$, $4\le 2l\le n$, be a sequence of $\selfs$ satisfying Definition \ref{collection elements}. We have that for any $x,y \in \Z_L^d$ and fixed $k\in \N$,
		\be\label{BRB} 
		|\Theta^{(n)}_{xy}(z)|\prec  B_{xy},\quad  \left|(\Theta^{(n)}_k)_{xy}(z)\right|\prec W^{-(k-2)d/2}B_{xy}  , \ee
		for all $z=E+\ii \eta$ with $|E|\le  2- \kappa$ and $\eta \in [ W^{-5}L^{5-d}, 1]$.
	\end{lemma}
	\begin{proof}
		The estimate \eqref{BRB} was proved in Lemma 6.2 of \cite{BandI} with $\eta\ge W^2/L^2$ and $\Theta^{\circ}$ replaced by $\Theta$. But, as we remarked below \eqref{Japanesebracket}, the same proof also works for our setting by using $|m(z)|\le 1$ and the spectral gap of $S^\circ$ at $1$. So we omit the details. 
	\end{proof}

	\subsection{Definition of the $T$-expansion}
	
	With the above preparations, in this subsection, we define the $n$-th order $T$-expansion for any fixed $n$.  We first introduce the concepts of recollision graphs and $Q$-graphs. 
	
	\begin{defn}[Recollision graphs and $Q$-graphs]\label{Def_recoll}
		{(i)} We say a graph is a recollision graph, if there is at least one dotted edge connecting an internal atom to an external atom. By this definition, a recollision graph 
  represents an expression where we set at least one summation (or free) index 
  to be equal to a fixed index value. 
		
		\vspace{5pt}
		\noindent{(ii)} We say a graph is a $Q$-graph if all  $G$ edges and weights in the graph have the same $Q_x$ label for a specific atom $x$. 
		In other words, the value of a $Q$-graph can be written as $Q_x(\Gamma)$ for an external atom $x$ or $\sum_x Q_x(\Gamma)$ for an internal atom $x$, where $\Gamma$ is a graph without $P/Q$ labels. 
	\end{defn}

 To give some simple examples, we consider the second order $T$-expansion in \eqref{eq:2nd}. If we set $x=\fb_1$ or $\fb_2$ in the graph $\sum_{x,y} m\Theta^{\circ}_{\fa x} s_{xy}   (G_{yy}-m) G_{x \fb_1} \overline G_{x \fb_2}$, we get a recollision graph. For example,
 \be\label{eq:simplerecoll} 
 \sum_{x,y} m\Theta^{\circ}_{\fa x} s_{xy}   (G_{yy}-m) \delta_{x\fb_1} (1-\delta_{x\fb_2})G_{x \fb_1} \overline G_{x \fb_2} 
 \ee
 is a $\fb_1$-recollision graph. Also notice that the graphs in $\sum_x \Theta^{\circ}_{\fa x} \mathcal Q^{(2)}_{x,\fb_1\fb_2} $ are all $Q$-graphs.

	\begin{defn} [$n$-th order $T$-expansion]\label{defn genuni}
		Fix any $n\in \N$ and a large constant $D>n$. For $\fa, \fb_1,\fb_2 \in \Z_L^d$, an $n$-th order $T$-expansion of $T_{\fa,\fb_1\fb_2}$ with $D$-th order error is an expression of the following form: 
		\begin{equation}
			\label{mlevelTgdef}
			\begin{split}
				T_{\fa,\fb_1 \fb_2}&= 
				m \Theta^{(n)}_{\fa \fb_1}\overline G_{\fb_1\fb_2} + \frac{|m|^2}{2\ii N\eta} (G_{\fb_2 \fb_1} - \overline G_{\fb_1 \fb_2}) \\
				&+ \sum_x \Theta^{(n)}_{\fa x}\left[\PTn_{x,\fb_1 \fb_2} +  \ATn_{x,\fb_1\fb_2}  + \Wn_{x,\fb_1\fb_2}  + \QTn_{x,\fb_1\fb_2}  +  (\Err_{n,D})_{x,\fb_1\fb_2}\right]\, ,
			\end{split}
		\end{equation}
		where the right side is a sum of $\OO(1)$ many graphs satisfying the following properties.

		\begin{enumerate}

			\item Every graph in $\PTn_{x,\fb_1\fb_2}$, $\ATn_{x,\fb_1\fb_2}$, $\Wn_{x,\fb_1\fb_2}$, $\QTn_{x,\fb_1\fb_2} $ and $(\Err_{n,D})_{x,\fb_1\fb_2}$ is a normal graph (recall Definition \ref{defnlvl0}) with external atoms $x, \fb_1,\fb_2$. Furthermore, in every graph,
			\begin{itemize}
				\item 		there is an edge, blue solid or $\dashed$ or dotted, connected to  $\fb_1$;
				
				\item there is an edge, red solid or $\dashed$ or dotted, connected to $\fb_2$.
			\end{itemize} 
			
			\item The sequence of $\selfs$ $\Sele_{k}$, $4\le k \le n$, satisfy Definition \ref{collection elements}.

			\item $\PTn_{x,\fb_1\fb_2}$ is a sum of recollision graphs (cf. Definition \ref{Def_recoll}) of scaling order $\ge 3$ and without any $P/Q$ label or free edge. 

			\item  $\ATn_{x,\fb_1\fb_2}$ is a sum of higher order graphs of scaling order $> n$ and without any  $P/Q$ label or free edge.  
			
			\item  $\Wn_{x,\fb_1\fb_2}$ is a sum of graphs of scaling order $\ge 3$, without any  $P/Q$ label and with exactly one free edge. 
			
			\item $\QTn_{x,\fb_1\fb_2} $ is a sum of $Q$-graphs (cf. Definition \ref{Def_recoll}) without any free edge.  
			
			
			
			\item $(\Err_{n,D})_{x,\fb_1\fb_2}$ is a sum of error graphs of scaling order $> D$ (these graphs may contain $P/Q$ labels and hence are not included into $\ATn_{x,\fb_1\fb_2}$).

			\item The graphs in $\PTn_{x,\fb_1\fb_2}$, $\ATn_{x,\fb_1\fb_2}$, $\Wn_{x,\fb_1\fb_2}$, $\QTn_{x,\fb_1\fb_2}$ and $(\Err_{n,D})_{x,\fb_1\fb_2}$ are doubly connected in the sense of Definition \ref{def 2net}. Moreover, the free edge in every graph of $\Wn_{x,\fb_1\fb_2}$ is redundant, that is, removing the free edge does not break the doubly connected property.
			
			
			%
			
		\end{enumerate}
		The graphs in $\PTn$, $\ATn$, $\Wn$, and $\QTn$ actually satisfy some additional graphical properties, which will be given later in Definition \ref{def genuni2}.
	\end{defn}
	
We now explain heuristically why we expect to have these terms $\PTn$, $\ATn$, $\Wn$, $\QTn$ and $\Err_{n,D}$ in the $T$-expansion. All these graphs come from further expanding the graphs in \eqref{eq:2nd} (with local and global expansions that will be given in Appendices \ref{sec_operations} and \ref{sec_goperations}). First, we will get a recollision graph in $\PTn$ if we add a dotted edge between a pair of internal and external atoms that are connected by a $G$ edge (see e.g., the graph in \eqref{eq:simplerecoll}). Second, the $Q$-graphs will appear when we apply Gaussian integration by parts, that is, given a graph, say $\cal G=P_x(\cal G)+Q_x(\cal G)$, we will apply Gaussian integration by parts to $P_x(\cal G)$, while $Q_x(\cal G)$ will be included in $\QTn$. Third, when constructing \eqref{mlevelTgdef}, we sometimes need to apply a lower order $T$-expansion (e.g., the second order $T$-expansion \eqref{eq:2nd}) to a $T$-variable. Then, the graph obtained by replacing the $T$-variable with the second term in \eqref{eq:2nd} is a graph with one free edge and will be included in $\Wn$. Finally, the remaining graphs that have no ``special structures" are of sufficiently high scaling order and will be included in $\ATn$. The error graphs in $\Err_{n,D}$ appear when we truncate an expansion at a very large order $D$. For example, in the $Q$-expansion that will be discussed in \Cref{sec_goperations}, we need to perform truncated Taylor expansions of weights $(G_{xx})^{-1}$. The graphs with truncated errors will be included in $\Err_{n,D}$ and they may contain $P/Q$ labels.

	To prove the local law, Theorem \ref{thm: locallaw}, we need to construct the $T$-expansion up to arbitrarily high order.  
	
	\begin{thm}[Construction of $T$-expansions]   \label{completeTexp} 
		Given any $M\in \N$, 
		we can construct a sequence of $n$-th order $T$-expansions satisfying Definition \ref{defn genuni} for all $2\le n \le M$.
	\end{thm}

	As a special case, when $\fb_1=\fb_2=\fb$, \eqref{mlevelTgdef} gives an expansion of the $T$-variable $T_{\fa\fb}$. Later in Section \ref{sec nondeter}, using the doubly connected property of the graphs in \eqref{mlevelTgdef}, we will show that very roughly speaking, the second term on the right side of \eqref{mlevelTgdef} provides a factor $(N\eta)^{-1}$ and the other terms can be bounded by $B_{\fa\fb}$.
	This leads to the local law, Theorem \ref{thm: locallaw}, since we know that  $T_{\fa\fb}$ controls the size of $|G_{\fa\fb}|^2$ (see Lemma \ref{lem G<T} below).

	\section{Complete expansions}\label{sec_complete_T}
	
	Letting $\fb_1=\fb_2=\fb$, we notice that \eqref{mlevelTgdef} is actually an incomplete expansion of $\fb_1=\fb_2=\fb$
	due to the random graphs in $\PTn$, $\ATn$, and $\Wn$. In fact, since the diagonal $G$ entries are close to $m$ by the local law \eqref{locallaw}, the first two terms on the right side of \eqref{mlevelTgdef} give the two leading deterministic terms:
	$ |m|^2  \Theta^{(n)}_{\fa \fb}  + |m|^2 \frac{ \im m}{ N\eta} .$ 
	The $Q$-graphs in $\QTn$ are fluctuation terms with zero expectation, and the error graphs in $\Err_{n,D}$ are always negligible as long as $D$ is taken sufficiently large. We need to further expand the graphs in $\PTn$, $\ATn$, and $\Wn$ into sums of \emph{deterministic graphs}, \emph{fluctuation graphs}, and \emph{error graphs}. We call the expansion thus obtained a \emph{complete $T$-expansion}. In other words, the complete $T$-expansion gives the exact expectation of $T_{\fa\fb}$ in terms of a sum of deterministic graphs plus a negligible error. 
	
	\subsection{$\Nonuni$}
	
	Before defining the complete $T$-expansion, we introduce another type of edge. 
	
	\begin{defn}[Ghost edges]
		We use a dashed edge between atoms $x$ and $y$ to represent a $W^2/L^2$ factor and call it a \emph{ghost edge}. We do not count ghost edges when calculating the scaling order of a graph, i.e., the scaling order of a ghost edge is $0$. Moreover, the doubly connected property in Definition \ref{def 2net} is extended to graphs with ghost edges by including these edges in the blue net. Finally, every ghost edge is associated with a $L^2/W^2$ factor in the coefficient of a graph.
	\end{defn}
	
	Like free edges, ghost edges have no indices and can be moved freely to other places without changing the graph. Both ghost and free edges are introduced to maintain the doubly connected property and some other graphical properties (e.g., the SPD property in Definition \ref{def seqPDG} and the generalized doubly connected property in Definition \ref{def 2netex}). 
	
	For any graph $\mathcal G$, let $k_\gh(\mathcal G) $ denote the number of ghost edges in $\cal G$. Define
	\begin{equation}
		\label{eq:def-size}
		\size(\mathcal G): = \Big( \frac{L^2}{W^2}\Big)^{k_\gh(\mathcal G)} W^{-\ord(\mathcal G) \soe}\,,
	\end{equation}
	where (recall that $\delta_0$ is the constant introduced in \Cref{thm: locallaw})
	\begin{equation}
		\label{eq:def-soe}
		\soe:= \begin{cases}
			\delta_0/2,& \etas\le \eta <(W/L)^2\,,\\
			d/2, & \eta \geq (W/L)^2\,.
		\end{cases}
	\end{equation}
	Roughly speaking, $\soe$ is chosen such that $B_{xy} + (N\eta)^{-1} =\OO(W^{-2\soe}) $ for $\eta \ge \etas$. 
	The following lemma shows that as long as there is a $T$-expansion of sufficiently high order, we can construct a $\nonuni$.
	
	\begin{lemma}[$\Nonuni$]\label{def nonuni-T}
		Under the assumptions of Theorem \ref{thm: locallaw}, suppose the local law 
		\begin{equation}\label{locallaw0}
			|G_{xy} (z) -m(z)\delta_{xy}|^2 \prec B_{xy} + \frac{1}{N\eta}
		\end{equation}
		holds for a fixed $z= E+\ii\eta$ with $|E|\le  2-\kappa$ and $\eta \in [\etas,1]$. Fix any $n \in \N$ sufficiently large such that 
		\begin{equation}\label{Lcondition10}  
			{L^2}/{W^2}  \le W^{(n-1) \soe -c_0} 
		\end{equation}
		for some constant $c_0>0$. Suppose we have an $n$-th order $T$-expansion satisfying Definition \ref{defn genuni}. Then, for any large constant $D>0$,  $T_{\fa,\fb_1 \fb_2}$ can be expanded into a sum of $\OO(1)$ many normal graphs (which may contain ghost and free edges): 
		\begin{equation}
			\label{mlevelTgdef weak}
			\begin{split}
				T_{\fa,\fb_1 \fb_2} & =   m  \wt \Theta_{\fa \fb_1}\overline G_{\fb_1\fb_2}  + \frac{|m|^2}{2\ii N\eta} (G_{\fb_2 \fb_1} - \overline G_{\fb_1 \fb_2})+ \Err_{\fa,\fb_1 \fb_2}\\
				&+\sum_{\mu} \sum_x \wt \Theta_{\fa x}\mathcal D^{(\mu)}_{x \fb_1}\overline G_{\fb_1\fb_2} f_\mu (G_{\fb_1\fb_1})+   \sum_\nu \sum_{x} \wt \Theta_{\fa x}\mathcal D^{(\nu)}_{x \fb_2} G_{\fb_2\fb_1} \wt f_\nu(G_{\fb_2\fb_2})  \\
				&+  \sum_{\gamma} \sum_{x} \wt \Theta_{\fa x}\mathcal D^{(\gamma)}_{x , \fb_1 \fb_2}g_\gamma(G_{\fb_1\fb_1},G_{\fb_2\fb_2},\overline G_{\fb_1\fb_2},  G_{\fb_2\fb_1})+ \sum_{\omega}\sum_{x}\wt \Theta_{\fa x} \mathcal Q^{(\omega)}_{x,\fb_1\fb_2}   .
			\end{split}
		\end{equation} 
		The graphs on the right-hand side satisfy the following properties. 
		\begin{enumerate}
			\item $\Err_{\fa,\fb_1 \fb_2}$ is an error term satisfying $\Err_{\fa,\fb_1 \fb_2}\prec W^{-D}.$ 
			
			\item $|\wt \Theta_{xy}| \prec B_{xy}$ and $\mathcal D^{(\mu)}_{x \fb_1}$, $\mathcal D^{(\nu)}_{x \fb_2}$, $\mathcal D^{(\gamma)}_{x , \fb_1 \fb_2}$ are deterministic doubly connected graphs with size $\leq W^{-c_0}$. 
			\item $f_\mu(\cdot)$, $\wt f_\nu(\cdot)$, and $g_\mu(\cdot)$ are monomials.
			\item $\mathcal Q^{(\omega)}_{\fa,\fb_1\fb_2} $ are doubly connected $Q$-graphs. (They actually satisfy some additional graphical properties, which will be given in Lemma \ref{def nonuni-T extra}.)
		\end{enumerate}
		
	\end{lemma}
	
	Roughly speaking, taking $\fb_1=\fb_2=\fb$, the above lemma shows that we can write $T_{\fa\fb}$ as a sum of two leading terms (i.e., the first two terms on the right side of \eqref{mlevelTgdef weak}), a fluctuation term of mean zero, a negligible error term, and some deterministic terms (except for some external weights $G_{\fb\fb}$ and $\overline G_{\fb\fb}$).

	
	\subsection{Complete expansions of graphs with multiple external atoms}
	
	We can also extend the complete $T$-expansion to complete expansions of more general graphs with multiple external atoms. In other words, given a graph consisting of external atoms and edges between them, we want to expand it into a sum of deterministic graphs, $Q$-graphs, and error graphs. Moreover, we will show that the deterministic graphs are properly bounded. 
	
	To state the main result, Lemma \ref{mGep}, of this subsection, we need to introduce a new concept of simple auxiliary graphs. 
	
	\begin{defn}
		A graph is a simple auxiliary graph if it only contains external atoms and the following few types of edges: 
		\begin{itemize}
			\item pseudo-waved edges: a waved edge between $x$ and $y$ represents a $W^{-d}\mathbf 1_{|x-y|\le W^{1+\tau}}$ factor for an arbitrary small constant $\tau>0$; 
			\item pseudo-diffusive edges: a double-line edge between $x$ and $y$ represents a $B_{xy}$ factor; 
			\item silent diffusive edges: a green double-line edge between $x$ and $y$ represents a factor
			$\wt B_{xy}:= W^{-4}\langle x-y\rangle^{-(d-4)};$
			\item free edges;
			\item silent free edges: a green solid edge between $x$ and $y$ represents a factor 
			$\frac{1}{N\eta} \frac{L^2}{W^2}.$
		\end{itemize}
		All these edges are counted as edges of scaling order $2$. 
	\end{defn} 
	Pseudo-waved edges give bounds for waved edges by \eqref{app compact f} and \eqref{S+xy}, and pseudo-diffusive edges give bounds for diffusive and labeled diffusive edges by \eqref{thetaxy} and \eqref{BRB}. Silent pseudo-diffusive and free edges come from summations over internal atoms in the deterministic graphs obtained from complete expansions. Essentially, they come from the following estimates:
	$$\sum_{w}B_{xw}B_{wy}\lesssim \wt B_{xy},\quad \sum_{w}B_{xw}\frac{1}{N\eta}\lesssim \frac{1}{N\eta} \frac{L^2}{W^2}.$$

	\begin{lemma}
		\label{mGep}
		Under the setting of Lemma \ref{def nonuni-T}, suppose we have a $\nonuni$ \eqref{mlevelTgdef weak} for a fixed $z= E+\ii\eta$ with $|E|\le  2-\kappa$ and $\eta \in [\etas,1]$. Let $\mathcal G_{\mathbf x}(z)$ be a graph consisting of external atoms $\bx=(x_1,\cdots, x_p)$, all taking different values, and (non-ghost and non-silent) edges between them. 
		For any constant $D>0$, we have that
		\begin{equation}\label{EGx}
			\E[\mathcal G_{\mathbf x}] = \sum_\mu\mathcal G_{\mathbf x}^{(\mu)} + \OO(W^{-D})\,,
		\end{equation}
		where the right-hand side is a sum of $\OO(1)$ many deterministic normal graphs $\mathcal G_{\mathbf x}^{(\mu)}$ with internal atoms and without silent edges such that if $x_i$ and $x_j$ are connected in $\mathcal G_{\mathbf x}$, then they are also connected in $\mathcal G_{\mathbf x}^{(\mu)}$ through non-ghost edges.
		Furthermore, every $\mathcal G_{\mathbf x}^{(\mu)}$ is bounded by a sum of $\OO(1)$ many simple auxiliary graphs: 
		\begin{equation}\label{EGx2}
			|\mathcal G_{\mathbf x}^{(\mu)}|\prec \sum_{\gamma}c_\gamma(W,L) \mathcal G_{\mathbf x}^{(\mu,\gamma)}+ W^{-D}\, ,
		\end{equation}
		where $c_\gamma(W,L)$ are positive $(W,L)$-dependent coefficients. These simple auxiliary graphs satisfy the following properties.
		\begin{itemize}
			\item[(a)] If $x_i$ and $x_j$ are connected in $\mathcal G_{\mathbf x}^{(\mu)}$, then they are also connected in $\mathcal G_{\mathbf x}^{(\mu,\gamma)}$. Furthermore, if $x_i$ and $x_j$ are connected in {$\mathcal G_{\mathbf x}^{(\mu)}$} without using free or ghost edges, then they are also connected in \smash{$\mathcal G_{\mathbf x}^{(\mu,\gamma)}$} without using free or ghost edges.
			\item[(b)] 
			If $\cal G_{\mathbf x}^{(\mu,\gamma)}$ contains $k \geq 2$ non-isolated atoms (i.e., atoms which have at least one neighbor), then there are at least $\lceil k/2\rceil$ special atoms such that each of them is connected with a different non-silent edge.
			
			\item[(c)] Every $\mathcal G_{\mathbf x}^{(\mu,\gamma)}$ satisfies $c_\gamma(W,L)\size(\mathcal G_{\mathbf x}^{(\mu,\gamma)})\le W^{-2\soe(p-t)}$ under the definition \eqref{eq:def-size}, where $t$ denotes the number of connected components in $\mathcal G_{\mathbf x}^{(\mu,\gamma)}$.
		\end{itemize}
	\end{lemma}
	\begin{remark}
		The bound $c_\gamma(W,L)\size(\mathcal G_{\mathbf x}^{(\mu,\gamma)})\le W^{-2\soe(p-t)}$ is actually a quite conservative estimate for general inputs. For example, if the initial graph $\cal G_{\bx}$ already contains a lot of edges in it, then the trivial bound $c_\gamma(W,L)\size(\mathcal G_{\mathbf x}^{(\mu,\gamma)})\le \size(\cal G_{\bx})$ may be better. However, the property (c) is sufficient for our purpose in this paper.  
	\end{remark}

	With Lemma \ref{mGep}, we can derive the following estimate. It will be used in proving the continuity estimate, Proposition \ref{lem: ini bound}, which is needed for the proof of the local law. 
	\begin{lemma}
		\label{gvalue_continuity}
		Suppose the assumptions of Lemma \ref{def nonuni-T} hold. Fix any $p\in \N$, consider a $p$-gon graph 
		\be\label{eq_pgons}
		\cal G_{\bx}(z) = \prod_{i=1}^{p}G^{s_i}_{x_i x_{i+1}}(z),
		\ee
		where $\bx:=(x_1,\cdots, x_{p})$, $x_{p+1}\equiv x_1$, and $s_i \in \{\pm\}$ with the conventions $ G_{xy}^{+}:= G_{xy}$ and $G_{xy}^{-}:= G^*_{xy}$. Let $\cal I\subset \Z_L^d$ be a subset with $|\cal I|\gtrsim W^d$ and denote $K:=|\cal I|^{1/d}$. Then, we have 
		\begin{equation}\label{eq_bound_Gx}
			\begin{split}
				&\frac{1}{|\cal I|^p}\sum_{x_i \in \cal I, i \in [p]}\mathcal G_{\mathbf x}(z) \prec \conc(K,W,\eta)^{{p-1}} \,,
			\end{split}
		\end{equation} 
		where
		\begin{equation}
			\label{eq:def-conc}
			\conc(K,W,\eta):=\sqrt{\left(\frac{1}{W^2 K^{d-2}}+ \frac{1}{N\eta}+ \sqrt{\frac{1}{N\eta}\frac{L^2}{W^2 K^d}}\right)\left(\frac{1}{W^4 K^{d-4}} + \frac{1}{N\eta}\frac{L^2}{W^2}\right)}.
		\end{equation}
	\end{lemma}
	
	\section{Proof of the main results}\label{sec_pf_main}
	
	\subsection{Proof of the local law}

Theorem \ref{thm: locallaw} follows immediately from Proposition \ref{locallaw-fix} and Theorem \ref{completeTexp}.
\begin{proposition}\label{locallaw-fix}
	Under the setting of Theorem \ref{thm: locallaw}, suppose we have an $n$-th order $T$-expansion satisfying Definition \ref{defn genuni}.
	
	\begin{itemize}
		\item[(i)] 
		Suppose for some constant $c_0>0$,
		\begin{equation}
			\label{Lcondition1}  
			{L^2}/{W^2}  \le W^{(n-1)d/2-c_0} .
		\end{equation}
		Then, the following local law holds uniformly in all $z= E+\ii\eta$ with $|E|\le  2-\kappa $ and $\eta \in [W^{2}/L^2,1]$:
		\begin{equation}\label{locallaw1}
			|G_{xy} (z) -m(z)\delta_{xy}|^2 \prec B_{xy} + \frac{1}{N\eta},\quad \forall \ x,y \in \Z_L^d. 
		\end{equation}
		
		\item[(ii)] Suppose for some constant $c_0>0$,
		\begin{equation}
			\label{Lcondition2}  
			{L^2}/{W^2}  \le W^{(n-1)\delta_0/2-c_0} .
		\end{equation}
		Then, the local law \eqref{locallaw1} holds uniformly in all $z= E+\ii\eta$ with $|E|\le  2-\kappa $ and $\eta \in [\etas,1]$. 
	\end{itemize}
\end{proposition}

The two conditions \eqref{Lcondition1} and \eqref{Lcondition2} come from \eqref{Lcondition10} due to the two values of $\soe$. Moreover, the two parts of Proposition \ref{locallaw-fix} are used at different stages of the proof. Part (i) will be used in the proof of Theorem \ref{completeTexp} (see the proof of Proposition \ref{sum0:L} in Section \ref{sec_pf_sumzero}). Once we have a sequence of $T$-expansions up to any order by Theorem \ref{completeTexp}, taking $n$ sufficiently large so that \eqref{Lcondition2} holds, we conclude Theorem \ref{thm: locallaw} by part (ii).

Similar to many previous proofs of local laws in the literature, we prove Proposition \ref{locallaw-fix} through a multi-scale argument in $\eta$, that is, we gradually transfer the local law at a larger scale of $\eta$ to a multiplicative smaller scale of $\eta$. The structure of the proof of \Cref{locallaw-fix} is depicted in \Cref{Fig pfchart2}, which will be implemented at the end of this subsection.

\begin{figure}[htb]
\color{black}
\tikzstyle{startstop} = [rectangle,rounded corners, minimum width=1cm,minimum height=0.5cm,text centered, draw=black]
\tikzstyle{startstop3} = [rectangle,rounded corners, minimum width=3cm, text width=6em, minimum height=0.5cm,text badly centered, draw=black]
\tikzstyle{block} = [rectangle, draw=black, 
    text width=18em, text centered, rounded corners, minimum height=1.5em]
\tikzstyle{line} = [draw=black, -latex']
\tikzstyle{line0} = [draw=black]
\tikzstyle{decision} = [diamond, draw=black, 
    text width=3em, text badly centered, node distance=2cm, 
    inner sep=0pt]

\tikzstyle{cloud} = [draw=black, ellipse, 
node distance=2.5cm,
    minimum height=2em]
\tikzstyle{null} = [draw=none,fill=none,right]

\begin{center}  
\begin{tikzpicture}[node distance = 1.5cm, auto]

    \node [startstop] (start) {Step 1,  \Cref{eta1case0}: Local law when $\eta=1$};
    \node [block, below of= start, node distance=1.3cm] (op1) {Induction hypothesis: Local law \eqref{locallaw1} when $\eta=\eta_{k}$};
    \node [block, below of=op1, node distance=1.5cm] (op2) {Step 2,  \Cref{lem: ini bound}: Obtaining the weak continuity estimate \eqref{eq:weak_local_aver} for $\eta=\eta_{k+1}$};
 
    \node [block, below of=op2, node distance=1.5cm] (op3) {Step 3,  \Cref{lemma ptree}: Obtaining the stronger estimate \eqref{pth T} for $\eta=\eta_{k+1}$};
    \node [block, below of=op3, node distance=1.3cm] (op4) {Concluding \eqref{locallaw1} for $\eta=\eta_{\ell}$};
     \node [block, below of=op4, node distance=1.2cm] (op5) {Step 4: Extending \eqref{locallaw1} to all $z$ uniformly};
     \node [cloud, left of=op2, node distance=6cm] (update) {Induction: $k+1 \to k$};

    \path [line] (start) -- (op1);
    \path [line] (op1) -- (op2);
    \path [line] (op2) --node [midway] {} (op3);
    \path [line] (op3) --node [midway] {} (op4);
     \path [line] (op4) --node [midway] {} (op5);
    \node [null, left of=op2, node distance=4.12cm] (null2) {};
    \node [null, right of=op2, node distance=4.12cm] (null22) {};
    \node [null, right of=op3, node distance=4.12cm] (null33) {};
    \node [null, left of=op4, node distance=4.12cm] (null4) {};
    \node [null, right of=op4, node distance=4.12cm] (null44) {};
     
    \path [line] (update) |- (op1);
    \path [line] (op3) -| node [near start] {} (update);

\end{tikzpicture}
\end{center}
\caption{We take $\eta_k=W^{-k \e_0}\vee \eta_*$, $k=0,1,\ldots, \ell$, for a small constant $\e_0>0$, where $\eta_*=W^2/L^2$ or $\etas$ and $\ell$ is the smallest integer such that $W^{-\ell \e_0}\le \eta_*$ (so there is $\eta_\ell=\eta_*$).}\label{Fig pfchart2}
\end{figure}

We first have an initial estimate at $\eta=1$, which holds in all dimensions $d\ge 1$.

\begin{lemma}[Initial estimate, Lemma 7.2 of \cite{BandII}] \label{eta1case0} 
Let $\kappa,\delta\in (0,1)$ be arbitrarily small constants. Under Assumptions \ref{assmH} and \ref{var profile}, fix any $d \geq 1$ and suppose $W\ge L^\delta$. 
 Then, for any $z=E+\ii\eta$ with $|E|\le  2-\kappa$ and $\eta=1$, we have that 
	\be\label{locallaw eta1}
	|G_{xy} (z) -m(z)\delta_{xy}|^2\prec  B_{xy}  ,\quad \forall \ x,y \in \Z_L^d.  
	\ee
\end{lemma}

Next, starting with a large $\eta$, Proposition \ref{lem: ini bound} gives a key continuity estimate, which says that if the local law holds at $\eta$, then a weaker local law will hold at a multiplicative smaller scale than $\eta$.  To state it, we need the following definition.

\begin{defn}\label{def_swnorms}
	An edge $\mathscr E_{xy}$ is said to be $\lambda$-bounded if $\|\mathscr  E \|_{w} \prec W^\lambda$, where
	\begin{align*}
		\|\mathscr  E \|_{w} & := W^{\soe}\max_{x,y \in \Z_L^d} |\mathscr E_{xy}|+  \sup_{k\in [W,L/2]}\max_{x,x_0 \in  \Z_L^d}\frac{1}{K^d\sqrt{\conc(K,W,\eta)}}\sum_{y:|y-x_0| \leq K} (|\mathscr E_{xy}|+|\mathscr E_{yx}|),
	\end{align*}
	where $\conc$ is defined in \eqref{eq:def-conc}.
	An edge $\mathcal A_{xy}$ is said to be $(\Phi,\lambda)$-bounded if it is $\lambda$-bounded and $\| \mathcal A \|_{s;\Phi} \prec 1$, where
	\begin{align*}
		\|\mathscr  E\|_{s;\Phi} := &\max_{x,y \in \Z_L^d} 
		|\mathscr E_{xy}| /\left(W^{-1} \gE{x-y}^{1-d/2} + \Phi  \right)\,.
	\end{align*}
\end{defn}
Here, the second stronger norm corresponds to the local law when $\Phi=(N\eta)^{-1/2}$, while the first weaker norm is due to the following continuity estimate. 

\begin{proposition}[Continuity estimate]\label{lem: ini bound}
	Under the setting of Proposition \ref{locallaw-fix}, suppose 
	\be\label{eq_cont_ini}
	\|G (\widetilde z) -m(\widetilde z)\|_{s;(N\eta)^{-1/2}} \prec 1,
	\ee
	with $\widetilde z =E+ \ii \widetilde\eta$ for $|E|\le  2-\kappa$ and $\widetilde\eta\in [\etas,1]$. Then, we have that  
 \be\label{eq:weak_local_aver}
 \|G (z) -m(z)\|_w \prec  \wt \eta/\eta ,
 \ee
	uniformly in $z=E+\ii \eta$ with $\max(\etas,W^{-\soe/20}\wt\eta) \le \eta \le \wt \eta$.
\end{proposition}

\begin{proof}
	Let $\mathcal I = \{y:|y-x_0| \leq K\}$. In equations (5.24) and (5.25) of \cite{BandI}, it has been proved that
 $$\sum_{y\in \mathcal I} \left(|G_{xy}(z)|^2 + |G_{yx}(z)|^2\right) \lesssim \sum_{y\in \mathcal I} \left(|G_{xy}(\wt z)|^2+|G_{yx}(\wt z)|^2\right) +  \left( \frac{\wt\eta}{\eta}\right)^2 \|\cal A\|_{\ell^2\to \ell^2}.$$
 Together with \eqref{eq_cont_ini} for $G(\wt z)$, it implies that
	\be\label{cont_lem1}
	\sum_{y \in \mathcal I} \left(|G_{yx}(z)|^2 + |G_{xy}(z)|^2\right) \prec \frac{K^2}{W^2} + \frac{K^d}{N\eta}+ \left(\frac{\wt \eta}{\eta} \right)^2 \| \mathcal A_{\mathcal I} \|_{\ell^2 \to \ell^2},
	\ee
	where $\mathcal A_{\mathcal I}$ is the submatrix of $\cal A=\frac{1}{2i}(G- G^*)$ with row and column indices in ${\mathcal I}$. 
	Using Lemma \ref{gvalue_continuity}, we get that for any $p\in 2\N$, 
	$$ \E \| \mathcal A_{\mathcal I} \|_{\ell^2 \to \ell^2} ^{p} \leq \E \mathrm{Tr}(\mathcal A_{\mathcal I}^{p})  \prec (K^d)^p \conc(K,W,\eta)^{p-1}\,.$$
	Together with Markov's inequality, it yields that $\| \mathcal A_{\mathcal I}\|_{\ell^2 \to \ell^2} \prec K^d\conc(K,W,\eta)$ since $p$ is arbitrary. Plugging this estimate into \eqref{cont_lem1} gives that
	\begin{equation}
		\max_{x,x_0}\frac{1}{K^d}\sum_{y:|y-x_0| \leq K} \left(|G_{yx}(z)|^2 + |G_{xy}(z)|^2\right) \prec \left(\frac{\wt \eta}{\eta} \right)^2 \conc(K,W,\eta)\,.
	\end{equation}
	Using the Cauchy-Schwarz inequality, we get from the above bound that  
	$$
	\max_{x,x_0}\frac{1}{K^d}\sum_{y:|y-x_0| \leq K} \left(|G_{yx}(z)| + |G_{xy}(z)|\right) \prec \frac{\wt \eta}{\eta} \sqrt{\conc(K,W,\eta)}\,.
	$$
	It remains to prove that
	\begin{equation}
		\|G(z)-m(z)\|_{\max}\prec  \frac{\wt \eta}{\eta}   W^{-\soe}\,.
	\end{equation} This can be proved using a standard $\e$-net and perturbation argument, see e.g., the proof of equation (5.8) in \cite{BandI}. We omit the details.
\end{proof}

Finally, Proposition \ref{lemma ptree} and Lemma \ref{lem G<T} will improve the weak continuity estimate in Proposition \ref{lem: ini bound} to the stronger local law.

\begin{proposition}[Entrywise bounds on $T$-variables]\label{lemma ptree} 
	Under the setting of Proposition \ref{locallaw-fix}, fix any $z=E+ \ii\eta$ with $|E|\le 2-\kappa$ and $\eta\in [\etas,1]$. Suppose 
	\begin{equation}
		\label{eq:cond-ewb}
		\|G(z) - m(z) \|_{w} \prec W^{\lambda}
	\end{equation} 
	for some constant $\lambda$ sufficiently small depending on $d$, $\delta_0$, $n$ and $c_0$ in \eqref{Lcondition1} or \eqref{Lcondition2}. Then, we have that 
	\begin{equation}\label{pth T}
		T_{xy}(z) \prec B_{xy} + \frac{1}{N\eta},\quad \forall \ x,y\in \Z_L^d.
	\end{equation}
\end{proposition}
We postpone the proof of Proposition \ref{lemma ptree} to Section \ref{sec nondeter}.  Combining this proposition with the following lemma, we can obtain the local law \eqref{locallaw1}. The bound \eqref{offG largedev} was proved in equation (3.20) of \cite{Band1D_III}, while \eqref{diagG largedev} was proved in Lemma 5.3 of \cite{delocal}.
\begin{lemma}\label{lem G<T}
	Suppose for a constant $\e>0$, a deterministic parameter $W^{-d/2}\le \Phi\le W^{-\e}$ and a subset $\mathbf D\subset \mathbb C_+$, we have that
	\begin{equation}
		\label{initialGT} 
		\|G(z)-m(z)\|_{\max}\prec W^{-\e},\quad \|T\|_{\max} \prec \Phi^2 ,\quad \text{uniformly in $z\in \mathbf D$.}
	\end{equation}
	Then, the following estimates hold:
	\begin{align}
		\mathbf 1_{x\ne y} |G_{xy}(z)|^2  \prec T_{xy}(z) \quad &\text{uniformly in $x\ne y \in \Z_L^d$ and $z\in \mathbf D$,}	\label{offG largedev}
		\\
		\label{diagG largedev} |G_{xx}(z)-m(z)| \prec \Phi\quad &\text{uniformly in $x\in \Z_L^d$ and $z\in \mathbf D$.}
	\end{align}
\end{lemma}
\begin{proof}[\bf Proof of Proposition \ref{locallaw-fix}]
	We define a sequence of $z_k=E+\ii \eta_k$ with decreasing imaginary parts $\eta_k:= \max\left(W^{-k \lambda/3}, \etas \right)$ for a sufficiently small constant $\lambda>0$. First, Lemma \ref{eta1case0} shows that the local law \eqref{locallaw1} holds for $z_0 = E +\ii $. Now, suppose \eqref{locallaw1} holds for $z_k = E + \ii\eta_k$, then Proposition \ref{lem: ini bound} yields that $\|G_{xy}(z) - m(z) \delta_{xy} \|_w \prec W^{\lambda}$ uniformly in $z=E+\ii \eta$ with $\eta_{k+1}\le \eta\le \eta_k$. Therefore, the condition \eqref{eq:cond-ewb} holds and we get \eqref{pth T} by Proposition \ref{lemma ptree}, which, together with Lemma \ref{lem G<T}, implies that the local law \eqref{locallaw1} holds at $z_{k+1}$. By induction in $k$, the above arguments show that the local law \eqref{locallaw1} holds for any fixed $z=E+\ii \eta$ with $\eta \in [\etas,1]$ (or $\eta \in [W^2/L^2,1]$ if we only have \eqref{Lcondition1}). The uniformity in $z$ follows from a standard $\e$-net and perturbation argument, see e.g., the proof of Theorem 2.16 in \cite{BandI}. We omit the details.
\end{proof}

\subsection{Proof of Proposition \ref{lemma ptree}} \label{sec nondeter}

Proposition \ref{lemma ptree} is a simple consequence of the following lemma.  
\begin{lemma}\label{lem highp1}
	Suppose the assumptions of Proposition \ref{lemma ptree} hold. Assume that
	\begin{equation}
		\label{initial_p}
		T_{xy} \prec B_{xy}+\widetilde\Phi^2,\quad \forall \ x,y \in \Z_L^d,
	\end{equation}
	for a deterministic parameter $\widetilde\Phi $ satisfying $0\le \widetilde\Phi \le W^{-\e}$ for some constant $\e>0$. Then, for any fixed $p\in \N$, we have that
	\begin{equation}
		\label{locallawptree}
		\E T_{xy} (z) ^p \prec  \Big(B_{xy} +W^{-c} \widetilde\Phi^2+\frac{1}{N\eta}\Big)^p 
	\end{equation}
	for some constant $c>0$ depending only on  $d$, $\delta_0$, $n$ and $c_0$ in \eqref{Lcondition1} or \eqref{Lcondition2}. 
\end{lemma}

\begin{proof}[\bf Proof of Proposition \ref{lemma ptree}]
	Starting with $T_{xy}\prec B_{xy} +\wt\Phi_0^2$, where $\wt\Phi_0:= W^{-d_\eta+\lambda}$ due to \eqref{eq:cond-ewb}, we combine \eqref{locallawptree} with Markov's inequality to obtain that
	$$T_{xy} (z) \prec  B_{xy} +W^{-c} \wt\Phi_0^2 +\frac1{N\eta} \ .$$
	Hence, \eqref{initial_p} holds with a smaller parameter $\wt\Phi=\wt\Phi_1:=W^{-c/2}\wt\Phi_0 + (N\eta)^{-1/2}$. Repeating this argument for $\left\lceil D/c\right\rceil$ many times, we obtain that 
	$$T_{xy} (z) \prec  B_{xy} +\frac1{N\eta}+W^{-D} .$$
	This concludes \eqref{pth T} as long as $D$ is large enough.
\end{proof}

It remains to prove Lemma \ref{lem highp1}. We first simplify the atomic graphs using their molecular structures. We will see (in Lemma \ref{GtoAG}) that it suffices to consider the following class of graphs, called {\it auxiliary} graphs, which are actually obtained as quotient graphs with molecules reduced to vertices.

\begin{defn}
An auxiliary graph consists of vertices and double-line or solid edges between them, where
\begin{itemize}
	\item every vertex represents a molecule,
	\item every double-line edge between vertices $x$ and $y$ represents a $B_{xy}$ factor and has scaling order 2,
	\item every solid edge represents a non-negative $\lambda$-bounded edge and has scaling order 1.
\end{itemize}
\end{defn} 

\begin{lemma}
\label{GtoAG}
Let $\mathcal G$ be a normal graph. Suppose $(G-m)$ is $\lambda$-bounded for a constant $\lambda\in [0,\soe)$. Then, for any constants $\tau,D>0$, there exists an auxiliary graph \smash{$\wt{\mathcal G}\equiv \wt{\mathcal G}(\tau)$}, whose vertices are representative atoms of the molecules in $\cal G$, such that if two molecules in $\cal G$ are connected by a $\dashed$ (resp. blue solid) edge, then their representative atoms are also connected by a double-line (resp. solid) edge in $\wt{\cal G}$. Furthermore, we have
\be\label{G_by_auxG} \mathcal G_{\abs} \prec W^{-[\ord(\mathcal G) - \ord(\wt{\mathcal G})](\soe-\lambda)+ \tau} \wt{\mathcal G} + W^{-D}\, ,\ee
where $ {\cal G}_{{\rm abs}}$ is obtained by replacing each component (including edges, weights, and the coefficient) in $\cal G$ with its absolute value and ignoring all the $P$ or $Q$ labels (if any). Finally, if $(G - m)$ is $\lambda$-bounded (resp. $(\Phi,\lambda)$-bounded), then so are the solid edges in $\wt{\mathcal G}$.
\end{lemma}
\begin{proof}
We first choose a representative atom for each molecule in $\cal G$. Then, the key to the proof is to show that all diffusive and $\lambda$-bounded (or $(\Phi,\lambda)$-bounded) edges between two molecules can be bounded by double-line and $\lambda$-bounded (or $(\Phi,\lambda)$-bounded) solid edges between their representative atoms. Our proof basically follows that of \cite[Lemma 6.10]{BandI}, and, for the convenience of the reader, we give more details below.

 In the following proof, we fix a small constant $\tau>0$ and a large constant $D>0$. Suppose there are $\ell$ internal molecules $\cal M_i$, $1\le i \le \ell$, and we choose one atom $x_i$ in each $\cal M_i$ as a representative. For simplicity of notations, we will use ``$\al\sim_{\cal M} \beta$" to denote that ``atoms $\al$ and $\beta$ belong to the same molecule". 
For any $y_i\sim_{\cal M}x_i$, it suffices to assume that  
\be\label{yixi}
|y_i-x_i|\le W^{1+\tau/2},
\ee
because otherwise the graph is smaller than $W^{-D}$.

We first define a matrix $\Psi$ as 
\be\label{eq defPsi}
\Psi^2_{xy}:=W^{-D}+ \max\limits_{ \substack{|x_1-x| \le   W^{1+\tau}  \\ |y_1-y|\le  W^{1+\tau}}}s_{x_1y_1} + W^{-(2+2\tau)d}\sum_{ |x_1-x| \le  W^{1+\tau}}\sum_{ |y_1-y|\le  W^{1+\tau}} |G_{x_1y_1}|^2  .\ee
It is easy to check that $\|\Psi\|_{w} \prec \|G-m\|_{w} + 1$ and  $\|\Psi\|_{s;\Phi} \prec \|G-m\|_{s;\Phi} + 1$ as long as $D$ is large enough. In equation (6.16) of \cite{BandI}, it has been proved that given $x_i,y_i\in \Z_L^d$ satisfying \eqref{yixi} for $i\in \{1,2\}$, 
\begin{align}
|G_{y_1y_2}|&\prec W^{d\tau} \Psi_{x_1x_2}.\label{Gpsi}
\end{align}
Now, under \eqref{yixi}, for $y_i\sim_{\cal M} x_i$ and $y_j\sim_{\cal M} x_j$, by \eqref{Gpsi}, \eqref{thetaxy} and \eqref{BRB}, we have that
\begin{align}
    |G_{y_i y_j}|\prec W^{d\tau} \Psi_{x_i x_j},\quad &\Theta^{\circ}_{y_i y_j}\prec B_{y_i y_j} \lesssim W^{(d-2)\tau/2} B_{x_i x_j},\label{intermole1}\\
\Theta^{(n)}_{y_i y_j}\prec W^{(d-2)\tau/2} B_{x_i x_j},\quad  &(\Theta^{(n)}_k)_{y_i y_j} \prec W^{-(k-2)d /2 + (d-2)\tau/2} B_{x_i x_j}.    \label{intermole2} 
\end{align}
These estimates show that we can bound the edges between different molecules with $\Psi$ or $B$ entries that only contain the representative atoms $x_i$ in their indices. In this way, we can obtain that 
\be\label{reduce Gaux0}
\cal G_{\abs}\prec W^{-n_1 d/2+ n_2 \tau}\sum_{x_1, \ldots, x_\ell} \wt{\mathcal G}(x_1, \cdots, x_\ell) \prod_{i=1}^\ell |\cal G_{x_i}^{(i)}| ,
\ee
where $W^{-n_1 d/2+ n_2\tau}$ is a factor coming from applying the estimates \eqref{intermole1} and \eqref{intermole2}, $\wt{\mathcal G}$ is a product of solid edges representing $\Psi$ entries and double-line edges representing $B$ entries, and every \smash{$\cal G_{x_i}^{(i)}$} is the subgraph inside the molecule $\cal M_i$ which has $x_i$ as an external atom. We next bound the local structure $\cal G_{x_i}^{(i)}$ inside $\cal M_i$ as follows: 
\begin{itemize}
\item by \eqref{subpoly}, \eqref{S+xy} and \eqref{thetaxy}, each waved or $\dashed$ edge is bounded by $\OO_\prec(W^{-d })$;  
\item by \eqref{BRB}, each labeled $\dashed$ edge is bounded by $\OO_\prec (W^{-kd/2 })$, where $k$ is its scaling order; 
\item each off-diagonal $G$ edge and light weight is bounded by $\OO_\prec(W^{-(\soe-\lambda)})$; 
\item each summation over an internal atom in $\cal M_i \setminus \{x_i\}$ provides a factor $\OO( W^{(1+\tau/2)d})$ due to \eqref{yixi}. 
\end{itemize}
Thus, with the definition of the scaling order in \eqref{eq_deforderrandom2}, we get that
\begin{align}\label{internal struc1}
|\cal G_{x_i}^{(i)}|\prec W^{-\ord(\cal G^{(i)}_{x_i}) \cdot (\soe-\lambda) + k_{i}\cdot \tau d/2 }, 
\end{align} 
where $k_i$ is the number of internal atoms in $\cal G_{x_i}^{(i)}$. 
Plugging \eqref{internal struc1} into \eqref{reduce Gaux0} concludes \eqref{G_by_auxG} since $\tau$ is arbitrary.
\end{proof}

A key ingredient for the proof of Lemma \ref{lem highp1} is the following lemma. 
\begin{lemma}\label{w_s}
Suppose $d \geq 7$ and $\eta \geq \etas$. Given any two matrices $ \mathcal A^{(1)}$ and $ \mathcal A^{(2)}$ with non-negative entries, we have that
\begin{equation}
	\label{keyobs3}
	\sum_{x_i}\mathcal A^{(1)}_{x_i \al}\cdot \prod_{j=1}^k B_{x_i y_j } \prec W^{-\soe} \Gamma(y_1,\cdots, y_k)\|\mathcal A^{(1)}\|_{w}\,,
\end{equation}
where $\Gamma(y_1,\cdots, y_k)$ is a sum of $k$ different products of $(k-1)$ double-line edges:  
\begin{equation}
	\label{defn_Gamma}\Gamma(y_1,\cdots, y_k):= \sum_{i=1}^k \prod_{j\ne i}B_{y_{i}y_j}.
\end{equation}
In addition, if $ \mathcal A^{(1)}$ and $ \mathcal A^{(2)}$ are $\lambda$-bounded (or $(\Phi,\lambda)$-bounded), then 
\begin{equation}
	\label{keyobs2}
	{\mathcal A}_{\alpha \beta}: = \frac{W^{ \soe - \lambda }}{\Gamma(y_1,\cdots, y_k)}\sum_{x_i}\mathcal A^{(1)}_{x_i \alpha }\mathcal A^{(2)} _{x_i \beta}\cdot \prod_{j=1}^k B_{x_i y_j } \ \text{is also $\lambda$ (or $(\Phi,\lambda)$)-bounded.}
\end{equation}

\end{lemma}

In a doubly connected auxiliary graph, we choose a blue spanning tree of the blue net and sum over all internal atoms from leaves to a chosen root (which is usually chosen as an external atom). Let $x_i$ be a leaf of the blue tree, \smash{$\mathcal A^{(1)}_{x_i \al}$} be the blue edge in the blue tree, and $B_{x_i y_j }$ be edges in the black net. Then, the estimate \eqref{keyobs3} shows that the summation over an internal atom $x_i$ can be bounded by a sum of doubly connected graphs obtained by removing the atom $x$ and the edges attached to it, and then adding edges between its neighboring atoms. Furthermore, if $x_i$ connects to an external atom, say $\beta$, then the estimate \eqref{keyobs2} shows that the summation over an internal atom $x_i$ can be bounded by a sum of doubly connected graphs, where the external atom $\beta$ still connects to an internal atom $\al$. 
After summing over all internal atoms, we can bound an auxiliary graph by a sum of graphs consisting of external atoms only. 
Hence, together with Lemma \ref{GtoAG}, Lemma \ref{w_s} enables us to bound the graphs in the $T$-expansion, which leads to the proof Lemma \ref{lem highp1}.


\begin{proof}[\bf Proof of Lemma \ref{lem highp1}] 
This proof is similar to that for Lemma 8.1 in \cite{BandII}. Using the $n$-th order $T$-expansion \eqref{mlevelTgdef} with $\fb_1=\fb_2=\fb$, we write that  
\begin{equation}
	\label{eq:ETp-Texp}
	\begin{split}
		\E T_{\fa \fb}^p =\E T_{\fa \fb}^{p-1} \Big\{ & m \Theta^{(n)}_{\fa \fb} \overline{G}_{\fb\fb} + \frac{|m|^2}{N\eta} \im G_{\fb\fb}\\
		&+\sum_x \Theta^{(n)}_{\fa x}\left[\PTn_{x,\fb \fb} +  \ATn_{x,\fb \fb}  + \Wn_{x,\fb \fb}  + \QTn_{x,\fb \fb}  +  (\Err_{n,D})_{x,\fb \fb}\right] \Big\}.
	\end{split}
\end{equation}
With Lemma \ref{GtoAG} and Lemma \ref{w_s} as inputs, we can show that 
\begin{equation}\label{estimates_ptree}
	\begin{split}
		& \sum_x \Theta^{(n)}_{\fa x}\PTn_{x,\fb \fb}  \prec W^{-\soe + \lambda} B_{\fa \fb}\,,\\
		& \sum_x \Theta^{(n)}_{\fa x}\ATn_{x,\fb \fb}  \prec W^{(n-1)(-\soe + \lambda)}\frac{L^2}{W^2} (B_{\fa \fb} + \wt \Phi^2)\,,  \\
		& \sum_x \Theta^{(n)}_{\fa x}\Wn_{x,\fb \fb}  \prec \frac{1}{N\eta}\frac{L^2}{W^2} (B_{\fa \fb} + \wt \Phi^2)\,,\\
		& \sum_x \Theta^{(n)}_{\fa x}\Err_{x,\fb \fb}  \prec W^{D(-\soe + \lambda)}\frac{L^2}{W^2} (B_{\fa \fb} + \wt \Phi^2)\,, \\
		&\E T^{p-1}_{\fa\fb}\QTn_{\fa,\fb\fb} \prec \sum_{k=2}^p \Big[W^{-\soe/2+\lambda/2}(B_{\fa\fb} + \wt \Phi^2)\Big]^k \E T^{p-k}_{\fa\fb}\,.
	\end{split}
\end{equation}
More precisely, the estimates in Lemma \ref{w_s} will replace the role of estimates (8.10)--(8.12) in \cite{BandII}. With these estimates, using exactly the same arguments as in \cite[Section 8.2]{BandII}, we can prove all the bounds in \eqref{estimates_ptree} except the third one involving $\Wn$.
For any graph in {$\Wn$}, 
removing the unique free edge in it still gives a doubly connected graph (see property 4 of Definition \ref{def genuni2} below), so it can be written as $(N\eta)^{-1} \cdot \mathcal G_{x\fb}$ for a doubly connected graph $\mathcal G$. We can bound $\mathcal G_{x\fb}$ using the same argument as the one for $\ATn$ in Claim 8.6 of \cite{BandII} and get  that
$$\sum_x \Theta^{(n)}_{\fa x}\mathcal G_{x\fb} \prec  \frac{L^2}{W^2} (B_{\fa \fb} + \wt \Phi^2).$$
This yields the third bound in \eqref{estimates_ptree}. Substituting \eqref{estimates_ptree} into \eqref{eq:ETp-Texp}, the desired result \eqref{locallawptree} with $c = \min(c_0 - (n-1)\lambda,\soe/2-\lambda/2)$ follows from an application of H\"older's inequality and Young's inequality.  
\end{proof}

Finally, the proof of Lemma \ref{w_s} involves a basic calculation using Definition \ref{def_swnorms}.
\begin{proof}[\bf Proof of Lemma \ref{w_s}]
Define the subsets
\begin{equation}
	\mathcal I_l = \Big\{ x \in \Z^d_L: \gE{x - y_l} \leq \min_{j\ne l}\gE{x- y_j}  \Big\}\,.
\end{equation}
So $B_{x y_l}=\max_{j}B_{xy_j}$ for $x \in \mathcal I_l$. Moreover, for $j\ne l$, we have $\gE{y_l-y_j} \le \gE{x- y_j}+\gE{x-y_l} \le 2\gE{x- y_j}$, implying that $B_{x_i y_j }\lesssim B_{y_l y_j }$. Thus, we obtain that
\begin{equation}
	\label{eq:1140223}
	\sum_{x_i \in I_l}\mathcal A^{(1)}_{x_i \beta}\cdot \prod_{j=1}^k B_{x_i y_j } \lesssim \sum_{x_i \in I_l}\mathcal A^{(1)}_{x_i \beta} B_{x_i y_l } \cdot \prod_{j \not = l} B_{y_l y_j } \le  \sum_{x_i \in \Z_L^d}\mathcal A^{(1)}_{x_i \beta} B_{x_i y_l } \cdot \prod_{j \not = l} B_{y_l y_j }\,.
\end{equation}
Without loss of generality, assume that $\mathcal A^{(1)}$ is $0$-bounded, i.e., $\| \mathcal A^{(1)} \|_{w} \prec 1$. With $K_n := 2^n W$, we have that
\begin{align*}
	& \sum_{x_i \in \Z_L^d}\mathcal |\mathcal A^{(1)}_{x_i \beta}| B_{x_i y_l }  \leq \sum_{ 1 \leq n \leq \log_2 \frac L W}\sum_{K_{n-1} \leq \gE{x_i - y_l} \leq K_n}\mathcal |\mathcal A^{(1)}_{x_i \beta}| B_{x_i y_l } \\
	&\leq  \sum_{ 1 \leq n \leq \log_2 \frac L W}\max_{K_{n-1} \leq \gE{x_i - y_l} \leq K_n} B_{x_i y_l}\cdot  \sum_{K_{n-1} \leq \gE{x_i - y_l} \leq K_n} |\mathcal A^{(1)}_{x_i \beta}|  \\
	& \prec \max_{ 1 \leq n \leq \log_2 \frac L W}W^{-2}K_n^{2-d}\cdot K_n^d\conc(K_n,W,\eta)^{\frac{1}{2}}\\
	&\leq \max_{ 1 \leq n \leq \log_2 \frac L W}  \left[ \left(\frac{1}{W^7 K_n^{d-7}}+ \frac{K_n^5}{W^5 N\eta} + \sqrt{\frac{K_n^{10-d}}{W^{10}N\eta}\frac{L^2}{W^2}}\right)\left(\frac{1}{W^7 K_n^{d-7}} + \frac{K_n^3}{W^3N\eta}\frac{L^2}{W^2}\right)\right]^{\frac{1}{4}}\,,
\end{align*}
which implies that 
\begin{equation}
	\sum_{x_i \in \Z_L^d}\mathcal |\mathcal A^{(1)}_{x_i \beta}| B_{x_i y_l }  \prec \begin{cases}
		W^{-\delta_0/2}, &\text{if}\ d \geq 7, \eta \geq \etas\\
		W^{- {d}/{2}}, &\text{if}\ d \geq 7, \eta \geq W^2/L^2
	\end{cases}\,.
\end{equation}
Combined with \eqref{eq:1140223}, it yields \eqref{keyobs3} (and also explains why we choose the $\soe$ in \eqref{eq:def-soe}). The result \eqref{keyobs2} is a simple consequence of \eqref{keyobs3}, and we omit the details of the proof.
\end{proof}

\subsection{Construction of the $T$-expansion}

In this section, we construct a sequence of $T$-expansions satisfying Definition \ref{defn genuni} order by order. First, it is not hard to see that the $n$-th order $T$-expansion can be obtained by solving the $n$-th order $T$-equation defined as follows. 

\begin{defn}[$n$-th order $\incomp$]\label{def incompgenuni}
Fix any $n\in \N$ and a large constant $D>n$. For $\fa, \fb_1,\fb_2 \in \Z_L^d$, an $n$-th order $\incomp$ of $T_{\fa,\fb_1\fb_2}$ with $D$-th order error is an expression of the following form:
\begin{equation}\label{mlevelT incomplete}
	\begin{split}
		T_{\fa,\fb_1 \fb_2}&= m  \Theta^{\circ}_{\fa \fb_1}\overline G_{\fb_1\fb_2} + \frac{|m|^2}{2\ii N\eta} (G_{\fb_2 \fb_1} - \overline G_{\fb_1 \fb_2})  +\sum_x (\Theta^{\circ} \Sigma^{(n)})_{\fa x} T_{x,\fb_1\fb_2} \\
		&+ \sum_x \Theta^{\circ}_{\fa x}\left[\PTn_{x,\fb_1 \fb_2} +  \ATn_{x,\fb_1\fb_2}  + \Wn_{x,\fb_1\fb_2}  + \QTn_{x,\fb_1\fb_2}  +  (\Err_{n,D})_{x,\fb_1\fb_2}\right]\,,
	\end{split}
\end{equation}
where $\PTn,\ATn,\Wn,\QTn,\Err_{n,D}$ are the graphs in Definition \ref{defn genuni}. 
\end{defn}

Second, given the $(n-1)$-th order $T$-expansion, following the expansion strategy described in \cite{BandII}, we can construct an $n$-th order $\incomp$ as in \Cref{Teq} below. 
Before stating it, we introduce the following notion of redundant edges.
\begin{defn}[Redundant edges]\label{def-redundant} 
In a doubly connected graph, an edge is said to be redundant if after removing it, the resulting graph is still doubly connected. Otherwise, the edge is said to be pivotal. 
\end{defn}

\begin{proposition}[Construction of the $\incomp$]  \label{Teq}
Given any $n\in \N$, suppose we have constructed an $(n-1)$-th order $T$-expansion satisfying Definition \ref{defn genuni}. 
Then, we can construct an $n$-th order $\incomp$ satisfying Definition \ref{def incompgenuni}, where $\Sigma^{(n)}$ is a deterministic matrix such that $\Sigma^{(n)}=\cal E_{n} + \Sigma^{(n-1)}$ with $\cal E_{n}$ being a sum of doubly connected deterministic graphs satisfying Definition \ref{collection elements} (i).  Moreover, every $\dashed$ edge in each graph in $\cal E_{n}$ is redundant.
\end{proposition} 

Third, we can prove that the deterministic matrix $\cal E_{n}$ constructed in Proposition \ref{Teq} indeed is a $\self$, that is, it satisfies properties \eqref{two_properties0}--\eqref{3rd_property0}.
\begin{proposition}\label{cancellation property}
The deterministic matrix $\Sele_n$ constructed in the $n$-th order $\incomp$ in Proposition \ref{Teq} satisfies the properties \eqref{two_properties0}--\eqref{3rd_property0} with $l=n/2$ (recall the convention \eqref{trivial conv}).
\end{proposition}

 The proof of Proposition \ref{Teq} is similar to that for Theorem 3.7 in \cite{BandII}, and we will describe an outline of it in \Cref{sec notation} without giving all the details. The proof of Proposition \ref{cancellation property} will be presented in Section \ref{sec_pf_sumzero}. 

Combining the above results, we can prove Theorem \ref{completeTexp} by induction.  

\begin{proof}[\bf Proof of Theorem \ref{completeTexp}]
Suppose we have constructed the $n$-th order $\incomp$.
Using property \eqref{two_properties0} for $\Sigma^{(n)}=\sum_{2l=4}^n\mathcal E_{2l}$, we get that
\begin{equation}\label{sum_TE0}
	\sum_{y} (\Theta^{\circ} \Sigma^{(n)})_{xy} = \sum_{\alpha,\beta} \Theta^{\circ}_{x\al} \Sigma^{(n)}_{\al\beta} =  \sum_{\alpha} \Theta^{\circ}_{x\al}\cdot \sum_{\beta}  \Sigma^{(n)}_{0\beta}  = 0.  
\end{equation}
Using \eqref{sum_TE0} and \eqref{eq:TTC}, we can rewrite \eqref{mlevelT incomplete} as
\begin{equation}
	\label{eq:solve-Teq-0}
	\begin{split}
		&\sum_x \left(I-\Theta^{\circ} \Sigma^{(n)}\right)_{\fa x} T^\circ_{x,\fb_1\fb_2}= m  \Theta^{\circ}_{\fa \fb_1}\overline G_{\fb_1\fb_2}   \\
		&+ \sum_x \Theta^{\circ}_{\fa x}\left[\PTn_{x,\fb_1 \fb_2} +  \ATn_{x,\fb_1\fb_2}  + \Wn_{x,\fb_1\fb_2}  + \QTn_{x,\fb_1\fb_2}  +  (\Err_{n,D})_{x,\fb_1\fb_2}\right]\,.
	\end{split}
\end{equation}
Solving \eqref{eq:solve-Teq-0} and recalling the definition \eqref{theta_renormal1}, we obtain that
$$ T^\circ_{\fa,\fb_1\fb_2} = m  \Theta^{(n)}_{\fa \fb_1}\overline G_{\fb_1\fb_2} + \sum_x \Theta^{(n)}_{\fa x}\left[\PTn_{x,\fb_1 \fb_2} +  \ATn_{x,\fb_1\fb_2}  + \Wn_{x,\fb_1\fb_2}  + \QTn_{x,\fb_1\fb_2}  +  (\Err_{n,D})_{x,\fb_1\fb_2}\right]. $$
Substituting it back to \eqref{eq:TTC} 
gives the $n$-th order $T$-expansion \eqref{mlevelTgdef}. The above argument together with Propositions \ref{Teq} and \ref{cancellation property} shows that given the $(n-1)$-th order $T$-expansion, we can construct the $n$-th order $T$-expansion. By mathematical induction, we conclude Theorem \ref{completeTexp}.   
\end{proof}
	
	\subsection{Proof of the quantum unique ergodicity}

	To prove Theorem \ref{thm:QUE}, we first observe that the quantity of interest $  {|I_N|}^{-1}\sum_{x \in I_N} (N|u_\alpha(x)|^2 -1)$ can be controlled by $\tr{\mathcal A \Pi \mathcal A \Pi}$, where $\mathcal A : =\im G$ and $\Pi$ is a diagonal matrix of zero  trace. Then, some bounds on high moments of $\tr{\mathcal A \Pi \mathcal A \Pi}$ obtained using Lemma \ref{mGep} will conclude Theorem \ref{thm:QUE}. The key ingredient for the proof is identifying a cancellation in $\tr{\mathcal A \Pi \mathcal A \Pi}$ from the graphical properties of its complete expansions.
	
	\begin{lemma}
		\label{uab-tr}
		Let $z = E + \ii\eta$ and $\Pi = \mathrm{diag}((\Pi_x)_{x \in \Z^d_L})$ be a real diagonal matrix. Then, for any $l\ge \eta$,
		\begin{align}
			& {\sum_{\alpha, \beta: |\lambda_\alpha - E| \leq l,|\lambda_\beta - E| \leq l}}|\gE{u_\alpha,\Pi u_\beta}|^2  \leq \frac{4l^4}{\eta^2} \tr{\mathcal A(z) \Pi \mathcal A(z) \Pi}  ,\label{eq:uab-tr-1}
			\\	
			\label{eq:uab-tr-2} &\sum_{\alpha, \beta: |\lambda_\alpha - E| \leq l,|\lambda_\beta - E| \leq l}|\gE{u_\alpha,\Pi \bar u_\beta}|^2  \leq \frac{4l^4}{\eta^2} \tr{\mathcal A(z) \Pi \overline {\mathcal A}(z) \Pi}\,.
		\end{align}
	\end{lemma}
	\begin{proof}
		Using the spectral decomposition of $\cal A$,
		$$	\mathcal A = \frac{1}{2i} ( G - G^*) = \sum_{\alpha}\frac{\eta}{|\lambda_\alpha -z|^2} u_\alpha u_\alpha^*\, ,
		$$
		we obtain that
		\be\label{APiAPi}
		\tr{\mathcal A \Pi \mathcal A \Pi} = \sum_{\alpha,\beta}\frac{\eta^2}{|\lambda_\alpha -z|^2|\lambda_\beta- z|^2} |\gE{u_\alpha,\Pi u_\beta}|^2 \,.
		\ee
		This yields \eqref{eq:uab-tr-1} immediately. The proof of \eqref{eq:uab-tr-2} is similar.
	\end{proof}
	
	With Lemma \ref{uab-tr}, the proof of Theorem \ref{thm:QUE} follows easily from the following proposition. We postpone the proof of \Cref{trABAB} until we complete the proof of Theorem \ref{thm:QUE}.
	\begin{proposition}
		\label{trABAB}
		Under the setting of Lemma \ref{mGep}, if $\Pi = \mathrm{diag}(\Pi_x:{x \in \Z^d_L})$ is a real diagonal matrix with $\tr{\Pi} = 0$, then for any fixed $p \in \N$, we have
		\begin{equation}
			\label{eq:trABAB}
			\E\left[|\tr{\mathcal A(z) \Pi \mathcal A(z) \Pi}|^{2p}\right] \prec \Big(\sum_y |\Pi_y|\Big)^{2p}  \Big(\max_x \sum_{y}B_{xy}|\Pi_y|\Big)^{2p} \,.
		\end{equation}
		The same bound holds for $\E\left[|\tr{\mathcal A(z) \Pi \overline {\mathcal A}(z) \Pi}|^{2p}\right]$.
	\end{proposition}
	
	\begin{proof}[\bf Proof of Theorem \ref{thm:QUE}]
		Under the assumptions of Theorem \ref{thm:QUE}, we know that the local law \eqref{locallaw0} holds by Theorem \ref{thm: locallaw}. Moreover, in the proof of Theorem \ref{thm: locallaw}, we have constructed a sequence of $T$-expansions up to arbitrarily high order $n$ by Theorem \ref{completeTexp}. We can choose $n$ large enough depending on $\delta,\ \delta_0$ and $ c_0$ so that \eqref{Lcondition10} holds. Then, both Lemma \ref{def nonuni-T} and Lemma \ref{mGep} hold, so we can use Proposition \ref{trABAB} in the following proof. 
		
		Now, define $b_x := \frac{N}{|I_N|}\11_{x \in I_N}$ and $\Pi_x: = b_x - 1$. Then, $\tr{\Pi}= 0$ and 
		\begin{equation}\label{QUE-id}
			\begin{split}
				\left|\gE{u_\alpha,\Pi u_\alpha}\right|^2 &= \Big(\sum_x \Pi_x |u_\alpha(x)|^2 \Big)^2 =  \Big(\sum_x b_x |u_\alpha(x)|^2 -1\Big)^2  \\
				&= \bigg(\frac{1}{|I_N|}\sum_{x \in I_N} (N|u_\alpha(x)|^2 -1)\bigg)^2 \,.
			\end{split}
		\end{equation}
		Next, taking $z=E+\ii \eta$ with $\eta= \etas$ and applying Markov's inequality to \eqref{eq:trABAB}, we obtain that
		\begin{align*}
			\tr{\mathcal A \Pi \mathcal A \Pi} & \prec N  \bigg( \frac{|I_N|^{2/d}}{W^2} \frac{N}{|I_N|} + \frac{L^2}{W^2}  \bigg) \le \frac{2 N^2}{ W^{2}|I_N|^{1-2/d}}\,.\end{align*}
		Therefore, it follows from \eqref{eq:uab-tr-1} that 
		\begin{align}
			\label{eq:supque} \sup_{\alpha: |\lambda_\alpha - E| \leq \eta}|\gE{u_\alpha,\Pi u_\alpha}|^2  &\prec \eta^2 \tr{\mathcal A \Pi \mathcal A \Pi} \prec \frac{(N\eta)^2}{W^{2}|I_N|^{1-2/d}} = \frac{L^{10}W^{2\delta_0}}{W^{12}|I_N|^{1-2/d}}\,,\\
			\label{eq:avque}	\frac{1}{N\eta}\sum_{\alpha: |\lambda_\alpha - E| \leq \eta}|\gE{u_\alpha,\Pi u_\alpha}|^2 & \prec \frac{\eta}{N} \tr{\mathcal A \Pi \mathcal A \Pi} \prec \frac{N\eta}{W^{2}|I_N|^{1-2/d}} = \frac{L^{5}W^{\delta_0}}{W^7|I_N|^{1-2/d}}\,.
		\end{align}
		Combining \eqref{eq:supque} and \eqref{QUE-id}, we conclude \eqref{eq:que} since $\delta_0<1$. 
		From \eqref{eq:avque}, we obtain that
		$$\frac{1}{N}\sum_{\alpha: |\lambda_\alpha | \leq 2-\kappa}|\gE{u_\alpha,\Pi u_\alpha}|^2  \prec \frac{L^{5}W^{\delta_0}}{W^{7}|I_N|^{1-2/d}}\,,$$
		which, together with Markov's inequality, implies that
		\begin{equation}
			\frac{1}{N}\,\left|\left\{ \alpha: |\lambda_\alpha| \leq 2 - \kappa,  |\gE{u_\alpha,\Pi u_\alpha}|\geq \epsilon \right\}\right| \prec  \frac{\epsilon^{-2}L^{5}W^{\delta_0}}{W^{7 }|I_N|^{1-2/d}}\,.
		\end{equation}
		Combining this bound with \eqref{QUE-id}, with a union bound over $I_N \in \mathcal I$, we conclude \eqref{eq:weakque}.
	\end{proof}
	
	Finally, the proof of Proposition \ref{trABAB} follows from Lemma \ref{mGep}.
	\begin{proof}[\bf Proof of Proposition \ref{trABAB}]
		First, note that $|\tr{\mathcal A \Pi \mathcal A \Pi}|^{2p}$ is a sum of terms of the form $$\sum_{\mathcal G} \sum_{\mathbf x,\mathbf y}\mathcal G_{\mathbf x,\mathbf y} \prod_{i}\Pi_{x_i} \Pi_{y_i},$$ where $\mathcal G_{\mathbf x,\mathbf y}$  are graphs of the form
		\begin{equation*}
			\mathcal G_{\mathbf x,\mathbf y} = c(\{s_i\})\prod_{i = 1}^{2p} G^{(s_{2i-1})}_{x_i y_i} G^{(s_{2i})}_{y_i x_i}, 
		\end{equation*}
		with ${s_i} \in \{\pm\}$ and $c(\{s_i\})$ denoting a deterministic coefficient of order $\OO(1)$. By Lemma  \ref{mGep}, we have that
		\begin{equation*}
			\E[\mathcal G_{\mathbf x,\mathbf y}] =  \sum_{\mu} \mathcal G^{(\mu)}_{\mathbf x,\mathbf y} + \OO(W^{-D}),
		\end{equation*}
		where $\mathcal G^{(\mu)}_{\mathbf x,\mathbf y}$ are deterministic graphs as in \eqref{EGx}. This implies that
		\begin{equation}\label{GGmugamma}
			\E[|\tr{\mathcal A \Pi \mathcal A \Pi}|^{2p}] = \sum_{\mu} \sum_{\mathbf x,\mathbf y}\mathcal G^{(\mu)}_{\mathbf x,\mathbf y} \prod_{i}\Pi_{x_i} \Pi_{y_i} +\OO(W^{-D})\,.
		\end{equation}
		
		We observe that it suffices to consider $\mathcal G^{(\mu)}_{\mathbf x,\mathbf y}$ in which the $4p$ external atoms belong to at most $2p$ connected components when we do not include free edges into the edge set; otherwise the graph vanishes. To see this, suppose there are at least $2p+1$ such connected components. Then, there must be a connected component that contains only one external atom. Without loss of generality, suppose this atom is $x_1$. Since our graphs are translation invariant,
		we know that $\mathcal G^{(\mu)}_{\mathbf x,\mathbf y}$ does not depend on $x_1$. Hence, 
		$\sum_{x_1}\mathcal G^{(\mu)}_{\mathbf x,\mathbf y} \Pi_{x_1} = \mathcal G^{(\mu)}_{\mathbf x,\mathbf y}\sum_{x_1} \Pi_{x_1} = 0. $

		By \eqref{EGx2}, every $\mathcal G^{(\mu)}_{\mathbf x,\mathbf y} $ is bounded by a sum of $\OO(1)$ many simple auxiliary graphs $\mathcal G^{(\mu,\gamma)}_{\mathbf x,\mathbf y}$, with a coefficient 
		\be\label{cgamma_que} c_{\mu,\gamma} \le W^{2\soe(k_{\mu,\gamma}+t_{\mu,\gamma}-4p)}, \ee
		where $k_{\mu,\gamma}$ and $t_{\mu,\gamma}$ are respectively the number of edges and the number of connected components (without removing free edges) in $\mathcal G^{(\mu,\gamma)}_{\mathbf x,\mathbf y}$, and we have used property (c) of Lemma \ref{mGep}. 
		Now, consider a graph $\wt{\mathcal G}^{(\mu,\gamma)}_{\mathbf x,\mathbf y}$ obtained from the following procedure.
		\begin{itemize}
			\item Remove all silent/non-silent free edges from $\mathcal G^{(\mu,\gamma)}_{\mathbf x,\mathbf y}$.
			\item Remove a set of silent/non-silent pseudo-diffusive edges from the graph until each connected component becomes a rooted tree (the choice of the root is arbitrary). We always remove silent pseudo-diffusive edges first whenever possible.
		\end{itemize}
		Denote by $Fr_{\mu,\gamma}$ and $DF_{\mu,\gamma}$ the number of (silent/non-silent) free edges removed and the number of (silent/non-silent) pseudo-diffusive edges removed, respectively. It follows that $${\mathcal G}^{(\mu,\gamma)}_{\mathbf x,\mathbf y} \leq W^{-d \cdot DF_{\mu,\gamma}}\left( \frac{L^2}{W^2}\frac{1}{N\eta} \right)^{Fr_{\mu,\gamma}}\wt{\mathcal G}^{(\mu,\gamma)}_{\mathbf x,\mathbf y}.$$
		
		To bound $\sum_{\mathbf x,\mathbf y}|\wt{\mathcal G}^{(\mu,\gamma)}_{\mathbf x,\mathbf y}|\prod_{i}|\Pi_{x_i}| |\Pi_{y_i}|$, we sum over atoms from the leaves to the chosen root of each tree. If an atom connects to a pseudo-diffusive edge, then we get a factor $\max_x \sum_y B_{xy}|\Pi_y| $ from the summation over it. If an atom connects to a silent diffusive edge, then we get a factor {$\max_x \sum_y \wt B_{xy}|\Pi_y| $}. If an atom is a root, then we get a factor $\sum_y |\Pi_y|$. In sum, we obtain that
		\begin{align}
			\sum_{\mathbf x,\mathbf y} |{\mathcal G}^{(\mu,\gamma)}_{\mathbf x,\mathbf y} &|\prod_{i}|\Pi_{x_i}| |\Pi_{y_i}| \prec W^{-d \cdot DF_{\mu,\gamma}}\left( \frac{L^2}{W^2}\frac{1}{N\eta} \right)^{Fr_{\mu,\gamma}}\nonumber\\
			&\quad  \times  \bigg(\max_x \sum_y B_{xy}|\Pi_y| \bigg)^{PD_{\mu,\gamma}}\bigg(\max_x \sum_y \wt B_{xy}|\Pi_y| \bigg)^{SD_{\mu,\gamma}} \bigg(\sum_y |\Pi_y| \bigg)^{CC_{\mu,\gamma}}, \label{QUE-pf-bdd}
		\end{align}
		where the exponents $PD_{\mu,\gamma}$, $SD_{\mu,\gamma}$, and $CC_{\mu,\gamma}$ denote respectively the numbers of pseudo-diffusive edges, silent pseudo-diffusive edges, and connected components in $\wt{\mathcal G}^{(\mu,\gamma)}_{\mathbf x,\mathbf y}$.
		Since $SD_{\mu,\gamma} = 4p-PD_{\mu,\gamma} - CC_{\mu,\gamma}$, abbreviating $ A:= \max_x \sum_y \wt B_{xy}|\Pi_y|$, we get that the right-hand side of \eqref{QUE-pf-bdd} is 
		\begin{align}
			\, W^{-d \cdot DF_{\mu,\gamma}} \left( \frac{L^2}{W^2}\frac{1}{N\eta} \right)^{Fr_{\mu,\gamma} }  &\left(\frac{\max_x \sum_y B_{xy}|\Pi_y|}{A} \right)^{PD_{\mu,\gamma}} A^{4p}   \left(\frac{\sum_y |\Pi_y| }{A}\right)^{CC_{\mu,\gamma}} \nonumber\\
			\leq \, W^{-d (DF_{\mu,\gamma} + 2p - CC_{\mu,\gamma})} &\left( \frac{L^4}{W^4}\frac{1}{N\eta} \right)^{Fr_{\mu,\gamma}}  \left(\frac{\max_x \sum_y B_{xy}|\Pi_y|}{A} \right)^{2p}A^{4p} \left(\frac{\sum_y |\Pi_y| }{A}\right)^{2p} \nonumber\\
			\le & W^{-2\soe ( k_{\mu,\gamma}+t_{\mu,\gamma}-4p)}\bigg(\max_x \sum_y B_{xy}|\Pi_y| \bigg)^{2p} \bigg(\sum_y |\Pi_y| \bigg)^{2p} .\label{QUE-pf-bdd2}
		\end{align}
		Here, in the first step, we used $CC_{\mu,\gamma} \leq 2p$, $Fr_{\mu,\gamma} + PD_{\mu,\gamma} \geq 2p$ by Lemma  \ref{mGep} (b), and 
		$$ \frac{W^2}{L^2} \cdot A \leq  \max_x \sum_y B_{xy}|\Pi_y|, \quad A\le W^{-d}\sum_y |\Pi_y| ,$$
		and in the second step, we used the trivial bounds
		$$W^{-d} + \frac{L^4}{W^4}\frac{1}{N\eta}\le W^{-2\soe},\quad DF_{\mu,\gamma}+Fr_{\mu,\gamma} +2p - CC_{\mu,\gamma} \ge k_{\mu,\gamma}+t_{\mu,\gamma}-4p,$$
		due to $k_{\mu,\gamma}=DF_{\mu,\gamma}+Fr_{\mu,\gamma}+ PD_{\mu,\gamma} +SD_{\mu,\gamma}$ and $ t_{\mu,\gamma} \le CC_{\mu,\gamma} \le 2p$.  Plugging \eqref{QUE-pf-bdd2} into \eqref{QUE-pf-bdd} and using \eqref{cgamma_que}, we obtain that 
		\begin{align*}
			\sum_{\mathbf x,\mathbf y} |{\mathcal G}^{(\mu)}_{\mathbf x,\mathbf y}|\prod_{i}|\Pi_{x_i}| |\Pi_{y_i}| & \prec \sum_\gamma c_{\mu,\gamma} \sum_{\mathbf x,\mathbf y} |{\mathcal G}^{(\mu,\gamma)}_{\mathbf x,\mathbf y}|\prod_{i}|\Pi_{x_i}| |\Pi_{y_i}|  \\
			&\prec \bigg(\max_x \sum_y B_{xy}|\Pi_y| \bigg)^{2p} \bigg(\sum_y |\Pi_y| \bigg)^{2p} .
		\end{align*}
		This yields \eqref{eq:trABAB} together with \eqref{GGmugamma}.
	\end{proof}
	
	\subsection{Proof of the bulk universality}
	
	Let $\delta_0 \in (0,1 )$ be a sufficiently small constant and set \begin{equation}\label{eq:eta*etao}
  \eta_* := \etas,\quad \eta_\circ := W^{-\delta_0}L^{-d}. 
	\end{equation} 
 We define the matrix Ornstein–Uhlenbeck process $H_t$ as the solution to
	\begin{equation}
		\dif H_t = -\frac{1}{2} H_t \dif t+ \frac1{\sqrt{N}}\dif B_t, \quad \text{with} \quad H_0 = H,
	\end{equation}
	where $B_t$ is a Hermitian matrix-valued Brownian motion. 
	Define its Green's function $G_t(z) = (H_t - z)^{-1}$ and the averaged trace $m_t(z) :=N^{-1}\tr{G_t(z)}$.
	\begin{proposition}
		\label{hmbdelta}
		Under the setting of Theorem \ref{main thm-uni}, fix $n\in \N$ and $|E|\le  2-\kappa$. 
		Consider a sequence $z_i=E_i+\ii \eta_i$, $i=1,\cdots, n$, with $|E_i-E|\le \eta_*$ and $\eta_{\circ}\le \eta_i \le W^{\delta_0}L^{-d}$. If $W^{95+d} \geq L^{95+c}$ for a constant $c>0$, then 
		\begin{equation}
			\sup_{0 \leq t \leq \eta_*}\bigg| \E  \prod_{i=1}^n\im m_t(z_i)    -  \E  \prod_{i=1}^n\im m_{\eta_*}(z_i)  \bigg| \leq L^{-c/6+ (2n+5)\delta_0 } \,,
		\end{equation}
		for any constant $\delta_0 \in (0, c/(12n+30))$.
	\end{proposition}
	\begin{proof}[\bf Proof of Theorem \ref{main thm-uni}]
		With Theorem \ref{thm: locallaw} as the input, applying \cite[Theorem 2.2]{LSY2019} gives the universality of the correlation functions of $H_t$ at $t = \eta_*$. Then, Theorem \ref{main thm-uni} follows from a standard correlation function comparison theorem \cite[Theorem 15.3]{erdHos2017dynamical} and Proposition \ref{hmbdelta}.  as long as we take a sufficiently small $\delta_0$.
	\end{proof}
	
	To conclude Proposition \ref{hmbdelta}, we will bound products of $\im m_t(z_i)$ using the following lemma.
	\begin{lemma}
		\label{sdedt}
		For any fixed $n\in \N$, we have that 
		\begin{equation}
			\label{eq:multipoints}
			\begin{split}
				&\bigg| \partial_t \E   \prod_{i=1}^n\im m_t(z_i)   \bigg| \\
				&\leq C\E\bigg[\sum_{u} L_{1,t}(z_u) \prod_{i\not 
				= u}\im m_t(z_i) + \sum_{u\ne v} L_{2,t}(z_u,z_v)\prod_{i\not 
				= u,v}\im m_t(z_i)\bigg] \,,
			\end{split}
		\end{equation}
		for an absolute constant $C>0$, where we denote
		\begin{align*}
			L_{1,t}(z)&:=  \sum_{\mathbf G_1,\mathbf G_2 \in \{G_t,  G_t^*\}}\Big|\frac{1}{N}\sum_{a,b} (\mathbf G_1^2(z))_{aa} s^\circ_{ab}(\mathbf G_2(z))_{bb}\Big|\,,\\
			L_{2,t}(z_1,z_2)&: = \sum_{\mathbf G_1,\mathbf G_2 \in \{G_t, G_t^*\}}\Big|\frac{1}{N^2}\sum_{a,b} (\mathbf G_1^2(z_1))_{ab} s^\circ_{ab}(\mathbf G_2^2(z_2))_{ba} \Big|\,.
		\end{align*}
	\end{lemma}
	\begin{proof}
		Using It{\^o}'s formula and Gaussian integration by parts, we get that
		\begin{align*}
			\partial_{t}\E  \prod_{i=1}^n\im m_t(z_i) & = \frac{e^{-t}}{2}\sum_{a,b}\E\left[\partial_{{ab}}\partial_{ba} \left(\prod_{i=1}^n\im m_t(z_i)\right)\right] \left( \frac{1}{N} - s_{ab}\right)\\
			& = \frac{-e^{-t}}{2}\sum_{a,b}  s_{ab}^{\circ} \E\Bigg[ \sum_u (\partial_{{ab}}\partial_{ba} \im m_t(z_u) ) \prod_{i\not = u}\im m_t(z_i) \\
			& \qquad \qquad + \sum_{u\ne v}( \partial_{ba} \im m_t(z_u) )( \partial_{{ab}} \im m_t(z_v) ) \prod_{i\not 
				= u,v}\im m_t(z_i)\Bigg] ,
		\end{align*}
		where $\partial_{ab}$ denotes $\partial/\partial_{(H_t)_{ab}}$. 
  (Note since the entries of $H_t$ are Gaussian, their cumulants of order three and higher vanish.)
  One can check that the terms coming from $\partial_{{ab}}\partial_{ba} \im m_t(z_u)$ can be bounded by $	L_{1,t}$, while those coming from $( \partial_{{ba}} \im m_t(z_u) )( \partial_{ab} \im m_t(z_v) )$ can be bounded by $	L_{2,t}$. This concludes the proof.
	\end{proof}
	With Lemma \ref{sdedt}, it suffices to prove the following estimates on $L_{1,t}$ and $L_{2,t}$. 
	\begin{lemma}
		\label{Ltbd}
		Under the setting of  Proposition \ref{hmbdelta}, for $u,v\in \{1,\ldots,n\}$,
		we have that 
		\begin{equation}\label{eq:l12bd}
			\begin{split}
			& \sup_{0 \leq t \leq \eta_*} L_{1,t}(z_u)  \prec W^{-(d+65)/6+5\delta_0}L^{65/6}  N\,,\\
			& \sup_{0 \leq t \leq \eta_*}  L_{2,t} (z_u,z_v) \prec W^{-(d+60)/6+6\delta_0}L^{10} N\, .
			\end{split}
		\end{equation}
	\end{lemma}

	\begin{proof}[\bf Proof of Proposition \ref{hmbdelta}]
		Under the condition $W^{95+d} \geq L^{95+c}$, by Lemma \ref{Ltbd}, we have
		$$ L_{1,t}(z_u)  + L_{2,t}(z_u,z_v)  \prec   {L^{-c/6+7\delta_0}}/{\eta_*},$$
		uniformly in $0 \leq t \leq \eta_*$. Plugging it into Lemma \ref{sdedt}, we obtain that 
		\begin{align*}
			\bigg|\partial_t  \E   \prod_{i=1}^n\im m_t(z_i) \bigg|  & \prec  \eta_*^{-1}L^{-c/6+7\delta_0} \cdot \bigg(\max_u \E \prod_{i\not  = u}\im m_t(z_i) + \max_{u\ne v}\E \prod_{i\not 
				= u,v}\im m_t(z_i)   \bigg) .
		\end{align*} 
		Denoting $B_k(t):= \max_{S\subset \{1,\ldots,n\},|S|=k} \E \prod_{i\in S}\im m_t(z_i)$,	the above estimate implies that  
		\be\label{iter_est}
		\begin{split}
		&\sup_{0 \leq t\leq \eta_*}  \bigg| \E   \prod_{i=1}^n\im m_t(z_i) -\E   \prod_{i=1}^n\im m_{\eta_*}(z_i)\bigg|\\
		&\prec  L^{-c/6+7\delta_0} \sup_{0 \leq t\leq \eta_*}  \left[B_{n-1}(t)+B_{n-2}(t)\right].
		\end{split}
		\ee
		Recall that an local law for $m_{\eta_*}(z)$ has been established in \cite[Theorem 3.3]{LY17} down to the scale $\wt \eta:=W^{\delta_0}L^{-d}$, which gives that $\im m_{\eta_*}(E_i + \ii \wt \eta)=\OO(1)$ with high probability. Then, using the simple fact $\eta_i \im m(E_i +\ii \eta_i)\le  \wt \eta \im m(E_i +\ii \wt\eta)$, we obtain that $B_{k}(\eta_*)\prec W^{2k\delta_0}$ for all $1\le k \le n$. With this initial bound, iterating \eqref{iter_est} for $n$ many times concludes the proof. 
	\end{proof}

	The rest of this subsection is devoted to the proof of Lemma \ref{Ltbd}. It suffices to prove \eqref{eq:l12bd} for every fixed $0 \leq t \leq  \eta_*$. After that, we can take an $L^{-C}$-net of $t\in [0, \eta_*]$ for a large enough constant $C>0$, and then use a standard union bound and perturbation argument to conclude Lemma \ref{Ltbd}. We slightly abuse the notations and still denote the eigenvalues and eigenvectors of $H_t$ by $\{\lambda_\al\}$ and $\{u_\alpha\}$.
	
	\begin{lemma}\label{L1L2-bd1}
		$	L_{1,t}$ and $	L_{2,t}$ satisfy that
		\begin{align*}
			L_{1,t}(z) \leq 4 \mathfrak m_t \sum_{\alpha} \frac{|p_\alpha(z)|^2 }{N}M_{\alpha}\,,\quad L_{2,t}(z_1,z_2) \leq 4 \sum_{\alpha, \beta} \frac{|p_\alpha(z_1)|^2|p_\beta(z_2)|^2}{N^3} M_{\alpha,\beta},
		\end{align*}
		where $p_\al(z):=(\lambda_\al-z)^{-1}$, $\mathfrak m_t:=N^{-1}\sum_\beta |p_\beta|$, and 
		\begin{align*}
			&M_\alpha: = \sup_a \left| N\gE{u_\alpha, S^\circ_a u_\alpha}\right|\,,\quad M_{\alpha, \beta}: =  \sup_a \sup_{\wt u_i \in \{u_i,\overline u_i\}: i \in\{\alpha,\beta\}} \left|N\gE{\wt u_\alpha, S^\circ_a \wt u_\beta}\right|\,,
		\end{align*}
		with $S_a^\circ := \mathrm{diag}(s^\circ_{ab}: b\in \Z^d_L)$.
	\end{lemma}
	\begin{proof}
		By the eigendecomposition of $G_t$, we have that
		\begin{align*}
			\sum_{a,b}	(G_t^2(z_1))_{ab} s^\circ_{ab}(G_t^2(z_2))_{ba} = \sum_{\alpha,\beta} p^2_\alpha(z_1) p^2_\beta(z_2) \sum_a u_\alpha(a) \bar u_\beta(a) \sum_b \bar u_\alpha(b) u_\beta(b)  s_{ab}^\circ\,.
		\end{align*}
		Hence, we get
		\begin{align*}
			\Big|\sum_{a,b}(G_t^2(z_1))_{ab} s^\circ_{ab}(G_t^2(z_2))_{ba}\Big| \leq \sum_{\alpha,\beta} |p^2_\alpha(z_1) p^2_\beta(z_2)| \sum_a |u_\alpha(a)|  |\bar u_\beta(a)|  \Big|\sum_b \bar u_\alpha(b) u_\beta(b)  s_{ab}^\circ \Big|\\
			\leq \frac{1}{N}\sum_{\alpha,\beta} |p^2_\alpha(z_1) p^2_\beta(z_2)|  M_{\alpha, \beta}\sum_a |u_\alpha(a)|  |\bar u_\beta(a)|  \leq \frac{1}{N}\sum_{\alpha,\beta} |p^2_\alpha(z_1) p^2_\beta(z_2)| M_{\alpha, \beta}\,.
		\end{align*}
		The other cases are similar. 
	\end{proof}
	
	Next, we bound the quantities appearing in the previous lemma using the local law and QUE. We have only proved the local law and QUE at $t=0$, but they can be extended to any $0 \leq t \leq \eta_*$.
	\begin{lemma}\label{lem:samefort}
		For any $0 \leq t \leq \eta_*$, Theorem \ref{thm: locallaw}, Corollary \ref{thm:supu}, and  Proposition \ref{trABAB}  hold for $H_t$.
	\end{lemma}
	\begin{proof}
		We copy the proof for the $t=0$ case verbatim, where the only difference is that $H_t$ has a different variance matrix $S_{(t)}$ with entries $e^{-t}s_{xy} + (1- e^{-t})N^{-1}$. Correspondingly, the matrices $\Theta^{\circ}$ and $S^\pm$ are replaced by (recall \eqref{Theta-S-circ} and \eqref{Thetapm2})
		\begin{align*}
			\Theta^\circ_{(t)} := \frac{e^{-t}|m|^2S^\circ}{1 - e^{-t}|m|^2 S^\circ}, \quad S^+_{(t)} :=  \frac{ m^2S_{(t)}}{1 -  m^2 S_{(t)}} , \quad S^-_{(t)} =\overline S^+_{(t)}.
		\end{align*}
	It is not hard to see that that $\Theta^\circ_{(t)}$ has the same behavior as $\Theta^\circ$, while $S^\pm_{(t)}$ has an extra zero mode compared to \eqref{S+xy}:
	$$ \big|\big(S^\pm_{(t)}\big)_{xy}- S_{xy}\big| \lesssim  t W^{-d}\mathbf 1_{|x-y|\le W^{1+\tau}} + t/{N} +  t\langle x-y\rangle^{-D} .$$
	Hence, we only need to show that if we replace some $\pm$ waved edges in the graphs appearing in the proof of the $t=0$ case by $t/{N}$ factors, the resulting new graphs will not cause any trouble. In fact, this kind of replacement only affects the definition of molecules (i.e., two atoms in the same molecule may not satisfy $|x-y|\le W^{1+\tau}$ anymore), which in turn may affect the doubly connected property. However, notice that we can split a $t/{N}$ edge between $x$ and $y$ as $\frac{1}{N} \leq B_{xy}\cdot \frac{W^2}{L^2}$ with an additional small $t$ factor in the coefficient, where $B_{xy}$ is a pseudo-diffusive edge that can be used in the black net and $ {W^2}/{L^2}$ is a ghost edge that can be used in the blue net. Hence, all the new graphs are still doubly connected if we regard a $1/N$ edge as a ``long" edge between molecules, while the additional small $t$ factor in their coefficients makes them harmless in most places. The only place where these new graphs may have an effect is that the property \eqref{4th_property0} now becomes 
	\be\label{4th_property_weak}
	\left|  (\Selek_{2l})_{0x}(z) \right| \prec  W^{-(l-2)d } B_{0x}^2  + tW^{-(l-1)d } \frac{W^2}{L^2}B_{0x} \prec W^{-(l-1)d} \frac{W^2}{\langle x\rangle^{d+2}},
	\ee
	for $x\in \Z_L^d$. The property \eqref{4th_property0} is only relevant in establishing Lemma \ref{lem redundantagain}, whose proof only uses the last weaker bound in \eqref{4th_property_weak} (see the proof of \cite[Lemma 6.2]{BandI}). Once the bounds on labeled $\dashed$ edges are proved, the rest of the proof in the $t=0$ case can be copied verbatim to the $t>0$ case, and we will not repeat them here. 
\end{proof}
\begin{remark}
To help the reader understand the above proof, we now describe some examples to illustrate the main difference from the $t = 0$ case. Consider the following two graphs that will appear in $T$-expansions (they are picked from (4.3) and (4.10) of \cite{BandI}):
\begin{center}
\includegraphics[width=6cm]{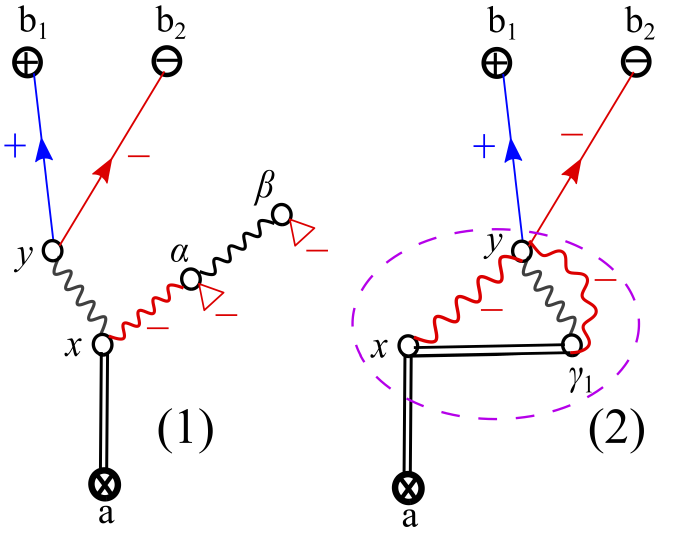} 
\end{center}
Here, graphs (1) and (2) respectively correspond to   
\begin{align*}
    \cal G^{(1)}_{\fa,\fb_1\fb_2}&=|m|^2 \sum_{x,y,\al,\beta} \Theta^\circ_{\fa x} s_{xy} S^-_{x\al} s_{\al\beta}( \overline G_{\al\al} - \overline m)( \overline G_{\beta\beta} - \overline m) G_{y \fb_1}\overline G_{y\fb_2},\\
    \cal G^{(2)}_{\fa,\fb_1\fb_2} &= |m|^2 \sum_{x,y,\gamma_1} \Theta^\circ_{\fa x} S^-_{xy} s_{y\gamma_1}S^-_{y\gamma_1} \Theta_{x\gamma_1}G_{y \fb_1}\overline G_{y\fb_2}.
\end{align*}
In particular, inside the purple circle of the graph (2) is a deterministic graph in $\Sele_6$: 
$$ \cal G_{xy} = |m|^2 S^-_{xy}\sum_{\gamma_1}s_{y\gamma_1}S^-_{y\gamma_1} \Theta_{x\gamma_1}.$$

When $t>0$, suppose we replace $S^-_{x\al}$ in $\cal G^{(1)}_{\fa,\fb_1\fb_2}$ by a $t/N$ factor and suppose we know $|G_{xy}-m|\prec B_{xy}^{1/2}+(N\eta)^{-1/2}$. Then, we can check that 
$$|\cal G^{(1)}_{\fa,\fb_1\fb_2}|\prec t W^{-\ord(\cal G^{(1)}_{\fa,\fb_1\fb_2})\cdot \soe},$$
which is negligible for all our proofs due to the $t$ factor. Adopting the viewpoint in the proof of \Cref{lem:samefort}, we can choose the two molecules as $\cal M_1=\{x,y\}$ and $\cal M_2=\{\al,\beta\}$. Then, the graph is still doubly connected if we bound a ${N}^{-1}$ factor by $B_{x\al}\cdot {W^2}/{L^2}$, where $B_{x\al}$ is a pseudo-diffusive edge and $ {W^2}/{L^2}$ is a ghost edge that can be used in the blue net. The same argument also works for the graph (2).

In $\cal G_{xy} $, if we replace $S^-_{xy}$ by a $t/N$ factor, then we have 
$$ \cal G_{xy} \prec \frac{t}{N} W^{-d}B_{xy} \le t\cdot W^{-2d} \frac{W^2}{\langle x-y\rangle^{d+2}},$$ 
which satisfies the bound in \eqref{4th_property_weak} with $2l=6$. Moreover, we have 
$$\sum_y | \cal G_{xy} | \prec t\frac{ W^{-d}}{N}\frac{L^2}{W^2} \le \eta_* \frac{W^{-d}}{N}\frac{L^2}{W^2},$$
which satisfies the sum zero property \eqref{3rd_property0}. 
\end{remark}

\begin{lemma}
	\label{pmab}
	Under the setting of Proposition \ref{hmbdelta}, for any small constant $\kappa>0$, the following estimates hold for $\ell\ge \eta_*$ and $z=E+\ii \eta$ with $|E|\le 2-\kappa$ and $\eta_\circ\le \eta\le \eta_*$ (recall \eqref{eq:eta*etao}):
	\begin{align}
		\label{eq:pmab-1}\sum_{\alpha: |\lambda_\alpha - E| \geq \ell} |p_{\alpha}(z)|^2 &\prec  \frac{N}{\ell}\frac{\eta_*}{\eta_\circ}\,,\\
		\label{eq:pmab-4}\mathfrak m_t(z) &\prec \frac{\eta_*}{\eta_\circ}\,,\\
		\label{eq:pmab-2}\sup_{ \alpha \in \llbracket \kappa N, (1-\kappa)N\rrbracket}M_\alpha  + \sup_{\alpha,\beta \in \llbracket \kappa N, (1-\kappa)N\rrbracket} M_{\alpha,\beta} &\prec N \eta_*\,,\\
		\label{eq:pmab-2.5}\sup_{\alpha}M_\alpha + \sup_{\alpha,\beta} M_{\alpha,\beta} &\prec N W^{-d}\,,\\
		\label{eq:pmab-3}\sum_{\alpha: |\lambda_\alpha - E| \leq \ell} |M_\alpha|^2  +\sum_{\alpha,\beta: |\lambda_\alpha - E| \leq \ell, |\lambda_\beta - E| \leq \ell} |M_{\alpha,\beta}|^2 &\prec \frac{N^2 \ell^4}{\eta_*^2} W^{-d}\,.
	\end{align}
\end{lemma}

\begin{proof}
	Let $z_*:=E+\ii \eta_*$. By Theorem \ref{thm: locallaw}, for any $|E|\le  2-\kappa$, we have 
	\begin{equation*}
		\begin{split}
			\#\{\alpha:|\lambda_\alpha - E| \leq \eta_\circ\} &\leq \sum_\alpha \frac{2\eta_\circ^2}{|\lambda_\alpha - (E + \ii\eta_\circ)|^2} = 2N\eta_\circ\im m_t({ z_\circ}) \\
			&\le 2N\eta_* \im m_t(z_*) \prec N\eta_* \,,
		\end{split}
	\end{equation*}
 where we denote $z_\circ:=E+\ii \eta_\circ$.
	Since $E$ and $\kappa$ are arbitrary, it gives that for any $k\in \N$ so that $k\eta_\circ \le (\log L)^{-1}$, 
	$$\#\{\alpha: |\lambda_\alpha-E| \in[k\eta_\circ,(k+1)\eta_\circ]\} \prec N\eta_*.$$
	With this estimate, we obtain  \eqref{eq:pmab-1} as follows:
	\begin{equation*}
		\begin{split}
			\sum_{\alpha: |\lambda_\alpha - E| \geq \ell} |p_{\alpha}(z)|^2  & \le \sum_{k\ge 0} \frac{4}{(\ell + k \eta_\circ)^2}\#\{\alpha: |\lambda_\alpha-E| \in[k\eta_\circ,(k+1)\eta_\circ]\}\\
			& \prec \sum_{k\in \N: k\eta_\circ \le (\log L)^{-1}} \frac{N\eta_*}{(\ell + k \eta_\circ)^2}  + N \lesssim \frac{N}{\ell} \frac{\eta_*}{\eta_\circ}\,.
		\end{split}
	\end{equation*}
	With a similar method, we get that
	\begin{equation*}
		\begin{split}
			\sum_\alpha |p_\alpha(z)| & = \sum_{\alpha: |\lambda_\alpha - E| \leq \eta_\circ} |p_\alpha(z)| + \sum_{\alpha: |\lambda_\alpha - E| > \eta_\circ} |p_\alpha(z)|  \prec  \frac{N \eta_*}{\eta_\circ} + \sum_{k\eta_\circ \le (\log L)^{-1}} \frac{N \eta_*}{k \eta_\circ}  \prec \frac{N \eta_*}{\eta_\circ},
		\end{split}
	\end{equation*}
	which yields \eqref{eq:pmab-4}. Next, \eqref{eq:pmab-2} follows from Corollary \ref{thm:supu}, and \eqref{eq:pmab-2.5} follows easily from the fact that $\|S^\circ_a\|\lesssim W^{-d}$. Finally, by Lemma \ref{uab-tr}, we have that
	\begin{equation*}
		\begin{split}
			&\sum_{\alpha: |\lambda_\alpha - E| \leq \ell} |M_\alpha|^2  +\sum_{\alpha,\beta: |\lambda_\alpha - E| \leq \ell, |\lambda_\beta - E| \leq \ell} |M_{\alpha,\beta}|^2 \\
			&\prec \sup_a \frac{N^2 \ell^4}{\eta_*^2} \left[\tr{\mathcal A(z_*) S_{a}^\circ  {\mathcal A(z_*)} S_{a}^\circ} + \tr{\mathcal A(z_*) S_{a}^\circ \overline {\mathcal A}(z_*) S_{a}^\circ}\right] \,.
		\end{split}
	\end{equation*}
	Applying Proposition \ref{trABAB} with $\Pi=S_{a}^\circ$, we get
	\begin{equation*}
		\tr{\mathcal A(z_*) S_{a}^\circ  {\mathcal A}(z_*) S_{a}^\circ} \prec W^{-d},\quad \tr{\mathcal A(z_*) S_{a}^\circ \overline{\mathcal A}(z_*) S_{a}^\circ} \prec W^{-d}\,.
	\end{equation*}
	The previous two equations together give \eqref{eq:pmab-3}.
\end{proof}
\begin{proof}[\bf Proof of Lemma \ref{Ltbd}]
	We first use Lemma \ref{L1L2-bd1} and then plug in the estimates in Lemma \ref{pmab} to get that 
	\begin{align*}
		L_{1,t}(z_u) &\leq \frac{4 \mathfrak m_t}{N} \bigg(\sum_{\alpha: |\lambda_\al|>2-\kappa/2}  |p_\alpha(z_u)|^2  M_{\alpha}+ \sum_{\alpha: |\lambda_\al-E|>\ell,|\lambda_\al|\le 2-\kappa/2}  |p_\alpha(z_u)|^2  M_{\alpha} \\
		&\qquad \quad +   \sum_{\alpha: |\lambda_\al-E|\le \ell} |p_\alpha(z_u)|^2  M_{\alpha}\bigg)\\
		& \prec \frac{\eta_*}{N\eta_\circ}\bigg[ N^2W^{-d} +\frac{N}{\ell }\frac{\eta_*}{\eta_\circ} \cdot N\eta_* +  \Big(\sum_{\alpha: |\lambda_\al-E|\le \ell} |p_\alpha(z_u)|^4\Big)^{1/2} \Big(\sum_{\alpha: |\lambda_\al-E|\le \ell} M_{\alpha}^2\Big)^{1/2}\bigg]\\
		& \prec \frac{\eta_*}{N\eta_\circ}\left( N^2W^{-d} + \frac{N}{\ell }\frac{\eta_*}{\eta_\circ} \cdot N\eta_* + \sqrt{\frac{N \im m_t(z_u)}{\eta_\circ^3}} \cdot \frac{N \ell^2}{\eta_*} W^{-d/2}\right)\\
		& \prec \frac{\eta_*}{N\eta_\circ}\left(N^2W^{-d} + \frac{N}{\ell}	\frac{\eta_*}{\eta_\circ} \cdot N\eta_* + \frac{\sqrt{N\eta_*}}{\eta_\circ^{2} } \frac{N\ell^2}{\eta_*} W^{-d/2}\right).
	\end{align*}
	Here, in the second step we used \eqref{eq:pmab-1}--\eqref{eq:pmab-2.5} and the Cauchy-Schwarz inequality, in the third step we used \eqref{eq:pmab-3} and 
	$$\sum_{\alpha: |\lambda_\al-E|\le \ell} |p_\alpha(z_u)|^4 \le \frac{1}{\eta_u^2}\sum_{\alpha} |p_\alpha(z_u)|^2=\frac{N\im m_t(z_u)}{\eta_u^3}\le \frac{N\im m_t(z_u)}{\eta_\circ^3},$$
	and in the fourth step we used $\eta_u \im m_t(z_u) \le \eta_* \im m_t(E_u+ \ii \eta_*)$ and the local law \eqref{locallaw}. Taking $\ell = W^{d/6} N^{1/6}\eta_\circ^{1/3}\eta_*^{5/6}$ gives the first bound in \eqref{eq:l12bd}. Similarly, we can bound $L_{2,t}$ as
	\begin{align*}
		&L_{2,t}(z_u,z_v) \\
		&\leq 4 \sum_{\alpha: |\lambda_\al-E_u|>\ell} \sum_\beta \frac{|p_\alpha(z_u)|^2|p_\beta(z_v)|^2}{N^3} M_{\alpha,\beta} + 4 \sum_{\alpha: |\lambda_\al-E|\le 2\ell} \sum_{\beta: |\lambda_\beta-E_v|> \ell} \frac{|p_\alpha(z_u)|^2|p_\beta(z_v)|^2}{N^3} M_{\alpha,\beta}\\
		&\quad + 4 \sum_{\alpha: |\lambda_\al-E|\le 2\ell} \sum_{\beta: |\lambda_\beta-E|\le 2\ell} \frac{|p_\alpha(z_u)|^2|p_\beta(z_v)|^2}{N^3} M_{\alpha,\beta}\\
		&\prec \frac{1}{N^3} \left(\frac{N^3}{W^{d}}+\frac{\sum_{a\in \{u,v\}} N \im m_t(z_a)}{\eta_\circ}\cdot \frac{N}{\ell } \frac{\eta_*}{\eta_\circ} N\eta_* +  \frac{N\sqrt{ \im m_t(z_u)\cdot \im m_t(z_v)}}{\eta_\circ^3}   \cdot \frac{N \ell^2}{ W^{d/2}\eta_*}\right)\\
		& \prec W^{-d}+\Big(\frac{\eta_*}{\eta_\circ}\Big)^3\frac{1}{\ell} +   \frac{\ell^2}{N\eta_\circ^4 W^{d/2}} = \frac{2\eta_*^{2}}{N^{1/3} W^{d/6}\eta_\circ^{10/3}}\,,
	\end{align*}
	where we have chosen $\ell=W^{d/6}N^{1/3}\eta_\circ^{1/3}\eta_* \ge \eta_* \ge \max_{i\in \{u,v\}} |E_i-E|$ and used $\{\alpha:|\lambda_\alpha - E_u| \le \ell\} \subset \{\alpha:|\lambda_\alpha - E| \le 2\ell\}$. This concludes the second bound in \eqref{eq:l12bd}.
\end{proof}

\section{Proof of Proposition \ref{cancellation property}}\label{sec_pf_sumzero}

The property \eqref{two_properties0} for $\cal E_n$ is a simple consequence of the fact that $S(x-y):=s_{xy}$ is symmetric and translation invariant (see Lemma A.1 of \cite{BandI}). The property \eqref{4th_property0} is a simple consequence of the following estimate on deterministic doubly connected graphs.

\begin{lemma}\label{dG-bd}
Suppose $d \geq 6$. Let $\mathcal G$ be a deterministic doubly connected normal graph without external atoms. Pick any two atoms of $\mathcal G$ and fix their values as $x,y \in Z^d_L$. The resulting graph $\mathcal G_{xy}$ satisfies that
\begin{equation}\label{det_double}
	\begin{split}
		(\mathcal G_{abs})_{xy} \prec &\, \Big( \frac{L^2}{W^2}\Big)^{k_\gh}\Big( \frac{1}{N\eta}\frac{L^2}{W^2}\Big)^{k_\fr}W^{-(n_{xy} - 4 - 2 k_{\fr})d/2}\\
		&\times \left(B_{xy}^2 + \mathbf 1_{k_\gh+k_\fr\ge 1} B_{xy} \frac{W^{2-d}}{L^2}\right),
	\end{split}
\end{equation}
where $n_{xy}: = \ord(\mathcal G_{xy})$, $k_\fr$ is the number of free edges and $k_\gh$ is the number of ghost edges. 
\end{lemma}
\begin{proof}
The $k_\gh=k_\fr=0$ case of \eqref{det_double} was proved as Corollary 6.12 of \cite{BandI}, while the case with nonzero $k_\gh$ or $k_\fr$ follows from Lemma 9.5 of \cite{BandII}.
\end{proof}

Applying \eqref{det_double} with $k_\gh=k_\fr=0$ gives \eqref{4th_property0}, because $\cal E_n$ constructed in Proposition \ref{Teq} consists of doubly connected graphs without free or ghost edges. The estimate \eqref{det_double} with nonzero $k_\gh$ or $k_\fr$ will be used in the proof of Lemma \ref{def nonuni-T} in Appendix \ref{sec_pf_completeT}.

To conclude Proposition \ref{cancellation property}, we still need to prove the key ``sum zero property" \eqref{3rd_property0}.  
\begin{proposition}
\label{sum0:L}
Let $\Sele_n$ be the deterministic matrix constructed in Proposition \ref{Teq}. Fix any $d \geq 7$, 
\begin{equation}
	\Big|\sum_x (\mathcal E_n)_{0x}(m,S,S^\pm,\Theta^{[n-1]})\Big| \prec \left[\left(\eta+ \frac{W^{2d-6}}{L^{2d-6}}\right) W^{-d}  + \frac{1}{W^{2} L^{d-2}} \right] W^{-(n-4)d/2}\,,
\end{equation}
for $z= E+ \ii \eta$ with $|E|\le 2- \kappa$ and $\eta \in [\etas, 1]$, where we abbreviate $ \Theta^{[n-1]} := (\Theta^{(l)}_r)_{2 \leq r \leq l \leq n-1}$.

\end{proposition}

To this end, we will compare $\mathcal E_n$ with its infinite space limit defined below, which is known to satisfy the sum zero property (see Proposition \ref{sum0-inf}).

\begin{defn}[Infinite space limits]\label{def infspace}
We first define the infinite space limits of $m, S,S^\pm$, and $\Theta^\circ$ by keeping $W$ fixed and taking $L \to \infty$ and $\eta \to 0$:
\begin{equation}
	\begin{split}
		m(E)&: = \lim_{\eta \to 0_+}m(E + i\eta),\quad(S_\infty)_{\al\beta}:= \lim_{L\to \infty} f_{W,L}(\al-\beta)\,,\\
		S_\infty^{+}(E)&:= \frac{m^2(E)S_\infty}{1-m^2(E)S_\infty},\quad  S_\infty^{-}(E):=\overline S_\infty^{+}(E),\quad (\Theta_\infty)_{\al\beta}:=\sum_{k\geq 1}S_\infty^{k}\,.
	\end{split}
\end{equation}
Given a deterministic graph $\cal G\equiv \cal G\left(m(z),S, S^{\pm}(z), \Theta^{[n-1]}(z)\right)$ with $ z=E+\ii \eta$, we define 
\be\label{defVnl} 
\cal G^\infty \equiv  \cal G^\infty \left( m(E), S_\infty ,  S_\infty^{\pm}(E),  \Theta^{[n-1]}_\infty(E)\right)
\ee
in the following way.
\begin{enumerate}
	\item Replace the $s_{\al\beta}$ edges in $\cal G $ with $(S_{\infty})_{\al\beta}$.
	\item Replace the $\cal S^{\pm}_{\al\beta}(z)$ edges in $\cal G $ with $(S_\infty^{\pm})_{\al\beta}(E)$.
	\item Replace the $ \Theta^\circ_{\al\beta}$ edges in $\cal G $ with $(\Theta_\infty)_{\al\beta}$. Replace the labeled diffusive edges \smash{$( \Theta^{(l)}_r)_{\al\beta}$} in $\cal G $ with \smash{$( \Theta^{(l)}_{r,\infty})_{\al\beta}$} (obtained by replacing the $S$, $S^\pm$ and $ \Theta^\circ$ edges in \smash{$\Theta^{(l)}_{r}$} by their infinite space limits).
	
	\item For all $m(z)$ in the coefficient (that is, $m(z)$'s that do not appear in $ S^{\pm}(z)$ and $\Theta^\circ(z)$), we replace them with $m(E)$.
	\item Finally, we let all internal atoms take values over $\Z^d$ and denote the resulting graph by $\mathcal G^\infty$.
\end{enumerate}
Note that $\cal G^\infty$ (if exists) depends only on $E$, $W$ and $\psi$ in Assumption \ref{var profile}, but does not depend on $L$ and $\eta$.
\end{defn}

\begin{proposition}
For $d \geq 8$, we have 
\label{sum0-inf}
\begin{equation}
	\label{eq:sum0-inf}
	\sum_x (\mathcal E^{\infty}_n)_{0x}(m(E),S_\infty(E),S^\pm_\infty(E),\Theta^{[n-1])}_\infty) = \OO(W^{-D})\,,
\end{equation}
for any large constant $D>0$. In addition, \eqref{eq:sum0-inf} also holds for $d = 7$ assuming that for some $L$ satisfying
\be\label{cond_d=7_only}W^{(n-3)d/2+c_0} \leq L^2/W^2 \leq W^{(n-2)d/2-c_0},  \ee
we have the $(n-1)$-th order $T$-expansion and the local law \eqref{locallaw1}  holds when $z = E + \ii (W^{2+\e}/L^2)$ for a sufficiently small constant $\e>0$. 
\end{proposition}
\begin{proof}
For $d \geq 8$, \eqref{eq:sum0-inf} was proved in Section 5.4 of \cite{BandI}.
The same proof works for $d = 7$ provided with the assumptions in the statement. 
We now give more details. Under the given conditions, using the argument in the proof of \cite[Lemma 5.11]{BandI}, we can show that 
$\big| \sum_x (\mathcal E_n)_{0x}\big| \leq W^{-(n-2)d/2}\cdot W^{-c}$ 
for a small constant $c>0$. Together with Lemmas \ref{LtoInf_x} and \ref{outsideL} below (which bounds the difference between $\mathcal E_n$ and $\mathcal E_n^\infty$), it gives that  
$$\Big|\sum_x (\mathcal E_n^\infty)_{0x}\Big| \lesssim W^{-(n-2)d/2}\cdot W^{-c}.$$ 
Furthermore, as proved in \cite[Lemmas 5.10 and 8.4]{BandI}, we have 
\be\label{eq:sumscale}\sum_x (\mathcal E_{n}^\infty)_{0x} = W^{-(n-2)d/2} \mathfrak G_n + \OO(W^{-D})
\ee
for some constant $\mathfrak G_n$ independent of $W$. The above two equations imply that $\mathfrak G_n=\OO(W^{-c})$, showing that there must be $\mathfrak G_n=0$. Together with \eqref{eq:sumscale}, it concludes \eqref{eq:sum0-inf}.
\end{proof}
\begin{remark}
 In the proof of \eqref{eq:sum0-inf} for the $d =7$ case, the only missing piece in \cite{BandI} is the continuity estimate, \Cref{lem: ini bound}, for $d=7$. With \Cref{lem: ini bound} in this paper, we can prove the local law \eqref{locallaw1} under the condition \eqref{cond_d=7_only}. Then, the proof in Section 5.4 of \cite{BandI} applies to the $d=7$ case almost verbatim as discussed above.
\end{remark}



With Proposition \ref{sum0-inf}, Proposition \ref{sum0:L} is a simple consequence of the next two lemmas. 

\begin{lemma}
\label{LtoInf_x}
For $d \geq 6$, let $\mathcal G$ be a deterministic graph in $\cal E_n$. For any $x \in \Z_L^d$ and $z=E+\ii \eta$, we have
\begin{equation}\label{eq:G-Gind}
\begin{split}
	& \left|\mathcal G_{0x}(m(z),S,S^\pm(z),\Theta^{[n-1]}(z)) -  \mathcal G_{0x}(m(E),S_\infty,S^\pm_\infty(E),\Theta^{[n-1]}_\infty(E)) \right|  \\
	&\prec (\eta W^{-d}  + W^{-2} L^{2-d}) B_{0x}^2 W^{-\frac{n-6}{2}d}\,.
\end{split}
\end{equation}
\end{lemma}

\begin{lemma}
\label{outsideL}
For $d \geq 6$, let $\mathcal G$ be an arbitrary deterministic graph in $\cal E_n$. We have
\begin{equation*}
\begin{split}
	&\Big|\sum_{x \in \Z^d} \mathcal G^\infty_{0x}(m(E),S_\infty,S^\pm_\infty(E),\Theta_\infty^{[n-1]}(E))- \sum_{x \in \Z^d_L} \mathcal G_{0x}(m(E),S_\infty,S^\pm_\infty(E),\Theta_\infty^{[n-1]}(E))\Big| \\
	& \prec  (W/L)^{2d-6} W^{-\frac{n - 2}{2}d}\,.
\end{split}
\end{equation*}
\end{lemma}

\begin{proof}[\bf Proof of Proposition \ref{sum0:L}]
Note that in the setting of Proposition \ref{cancellation property}, we have constructed an $(n-1)$-th order $T$-expansions, and part (i) of Proposition \ref{locallaw-fix} then shows that the local law \eqref{locallaw1}  holds for $z = E + \ii (W^{2+\e}/L^2)$ under the condition \eqref{cond_d=7_only}. Hence, \eqref{eq:sum0-inf} holds for all $d\ge 7$ by Proposition \ref{sum0-inf}. In addition, from Lemma \ref{LtoInf_x}, we get that 
\begin{equation*}
\begin{split}
	&\Big|\sum_{x\in \Z_L^d}\mathcal G_{0x}(m,S,S^\pm,\Theta^{[n-1]}) - \sum_{x\in \Z_L^d} \mathcal G_{0x}(m(E),S_\infty,S^\pm_\infty(E),\Theta^{[n-1]}_\infty(E)) \Big| \\
	&  \prec (\eta W^{-d}  + W^{-2} L^{2-d})  W^{-(n-4)d/2}.
\end{split}
\end{equation*}
Together with Lemma \ref{outsideL}, it implies that 
\begin{equation*}
\begin{split}
	&\Big|\sum_{x\in \Z_L^d} (\mathcal E_n)_{0x}(m,S,S^\pm,\Theta^{[n-1]})-\sum_{x\in \Z^d} (\mathcal E^{\infty}_n)_{0x}(m(E),S_\infty,S^\pm_\infty(E),\Theta^{[n-1])}_\infty(E))\Big|\\
	&\prec \left[\left(\eta+ (W/L)^{2d-6}\right) W^{-d}  + W^{-2} L^{2-d} \right] W^{-(n-4)d/2}.
\end{split}
\end{equation*}
Combining it with \eqref{eq:sum0-inf}, we conclude the proof. 
\end{proof}

The rest of this section is devoted to the proofs of Lemmas \ref{LtoInf_x} and \ref{outsideL}.

\subsection{Proof of Lemma \ref{outsideL}}


Note that the graphs $\cal G^\infty$ and $\cal G$ are different in the sense that the atoms in $\cal G^\infty$ take values over $\Z^d$ while the atoms in $\cal G$ take values over $\Z_L^d$.

Suppose $0$ connects to an atom $y$ through a diffusive edge $e=(0,y)$ in $\cal G$. Then, removing this edge still gives a doubly connected graph due to the last statement in \Cref{Teq}.
Now, let $\cal G^{\infty}_{0x,y}$ be the graph obtained by fixing $y$ in $\cal G^{\infty}_{0x}$ as an external atom. Note that
{$	\sum_x \mathcal  G^{\infty}_{0x,y} = \wt{\mathcal G}^{\infty}_{0y}$}, 
where \smash{$\wt{\mathcal G}^{\infty}_{0y}$} is a graph obtained by changing $x$ in $\mathcal  G^{\infty}_{0x,y}$ to an internal atom. Then, applying Lemma \ref{dG-bd} to the doubly connected graph with $e$ removed from \smash{$\wt{\mathcal G}^{\infty}_{0y}$} gives that 
\begin{equation*}
\sum_{|y| \geq L/\log L}\left|\mathcal G^{\infty}_{0y}\right| \prec W^{-(n - 6)d/2} \sum_{|y| \geq L/\log L}B_{0y}^3 \prec (W/L)^{2d-6} W^{-(n - 2)d/2}\,.
\end{equation*}
This means that if we do not sum over $|y| \geq L/\log L$ in $\sum_x \mathcal  G^{\infty}_{0x}$, it gives an error of order 
\be\label{err-det}
\OO_{\prec}\left((W/L)^{2d-6} W^{-(n - 2)d/2}\right).
\ee
If $0$ connects to an atom $y$ through a waved edge, then by \eqref{app compact f} and \eqref{S+xy}, we get that summing over $|y| \geq L/\log L$ gives a negligible error $\OO(W^{-D})$.

The above argument can be extended to any pair of neighbors. For example, suppose $y$ connects to an atom $w$. Then, for $|y| \leq L/\log L$, if we do not sum over $|w| \geq 2L/\log L$ in $\sum_x \mathcal  G^{\infty}_{0x}$, it gives an error of order \eqref{err-det}. Continuing this argument, since the graph $\cal G$ is connected and there are $\OO(1)$ many atoms, we obtain the desired result.

\subsection{Proof of Lemma \ref{LtoInf_x}}

Note that when $\eta$ is of order 1, Lemma \ref{LtoInf_x} follows directly from Lemma \ref{dG-bd}. Hence, without loss of generality, we assume $\eta\ll 1$ in the following proof. 

\begin{lemma}
\label{mSinf}
For $x\in \Z_L^d$ and any large constant $D>0$, we have that
\begin{align*} &\Big| \mathcal G_{0x}(m(E),S_\infty ,S^\pm_\infty(E),\Theta_\infty^{[n-1]}(E))  - \mathcal G_{0x}(m(z),S,S^\pm(z),\Theta_\infty^{[n-1]}(E)) \Big| \\
& \prec \eta B_{0x}^2W^{-(n - 4)d/2} + L^{-D}\,.
\end{align*}
\end{lemma}
\begin{proof}
It is an immediate consequence of \eqref{app compact f}, \eqref{S+xy}, the equation (7.8) of \cite{BandI} which shows that 
$$ |(S^+_\infty)_{0x}(E)-  S^{+}_{0x}(z)| \prec \eta W^{-d}\mathbf 1_{|x|\le W^{1+\tau}} + L^{-D}, $$
and the simple estimate $|m(z) - m(E)| = \OO(\eta)$.
\end{proof}

\begin{lemma}
Under the setting of Proposition \ref{sum0:L}, let $\mathcal G$ be an arbitrary deterministic graph in $\cal E_{l+1}$ for a fixed $2\le l \le n-1$. If \eqref{eq:G-Gind} holds for every graph in $\cal E_{k}$ with $k \leq l$, then for $x\in \Z_L^d$,
\label{sum0:changeTheta}
\begin{align*} & |  \mathcal G_{0x}(m(z),S,S^\pm(z),\Theta_\infty^{[l]}(E)) -  \mathcal G_{0x}(m(z),S,S^\pm(z),\Theta^{[l]}(z))| \\
&\prec (\eta  W^{-d} + W^{-2} L^{2-d})B_{0x}^2W^{-((l+1)-6)d/2}\,.
\end{align*}
\end{lemma}

\begin{proof}[\bf Proof of Lemma \ref{LtoInf_x}]
First, we trivially have $\cal E_2=0$. Now, suppose we have proved that Lemma \ref{LtoInf_x} holds for all graphs in the $\selfs$ of order $\leq l$. Then, Lemmas \ref{mSinf} and \ref{sum0:changeTheta} imply Lemma \ref{LtoInf_x} for the graphs in $\cal E_{l+1}$. By induction in $l$, we conclude the proof. 
\end{proof}

A key ingredient in proving Lemma \ref{sum0:changeTheta} is a comparison between (labeled) diffusive edges in $\Theta^{[n]}$ and those in $\Theta^{[n]}_\infty$.
\begin{lemma}
If \eqref{eq:G-Gind} holds for every graph in $\cal E_{k}$ with $k \leq l$, then for any $l'\le l$,
\label{theta-itself}
\begin{equation}
\label{eq:theta-itself}
\Big| (\Theta^{(l)}_{l'} )_{xy} -  (\Theta^{(l)}_{l',\infty} )_{xy}\Big| \prec W^{-(l'-2)d/2}\left( W^{-2}L^{2-d}+ \frac{\eta }{W^{4}\gE{x-y}^{d-4}}\right).
\end{equation}
\end{lemma}
With this lemma, Lemma \ref{sum0:changeTheta} is a direct consequence of the following result, which controls the error in replacing every (labeled) diffusive edge with its infinite space limit.
\begin{lemma}
\label{change_Theta}
Under the setting of Lemma \ref{sum0:changeTheta}, let $e = (a,b)$ be an arbitrary diffusive $\Theta^{(k)}_{k'}$ edge in $\mathcal G$ with $k' \leq k \leq l$. Let $\mathcal G^e$ be a graph obtained by replacing the edge $e$ in $\mathcal G$ by $(\Theta^{(k)}_{k'} - \Theta^{(k)}_{k',\infty})_{ab}$. Then, we have that
\begin{equation}
\mathcal (\mathcal G^e_{abs})_{0x} \prec (\eta  W^{-d} + W^{-2}L^{2-d})B_{0x}^2W^{-((l+1)-6)d/2}\,.
\end{equation}
\end{lemma}
\begin{proof}
Let $\wh {\mathcal G}$ be the subgraph of $\mathcal G$ with $e$ removed and $\wh {\mathcal G}(a,b)$ be the graph obtained by fixing $a,b$ in $\wh {\mathcal G}$ as external atoms. 
By the last statement in \Cref{Teq}, $\wh {\mathcal G} $ is still doubly connected. Therefore,
\begin{align*}
&(\mathcal G^e_{abs})_{0x} \le \sum_{a,b} \left|\left(\Theta^{(k)}_{k'} - \Theta^{(k)}_{k',\infty}\right)_{ab}\right|(\wh {\mathcal G}_{abs}(a,b))_{0x } \\
&\prec (W^{-2}L^{2-d } + \eta W^{-d}) W^{-(k'-2)d/2}  (\wh {\mathcal G}_{abs})_{0x } \prec  (W^{-2}L^{2-d } + \eta W^{-d}) B_{0x}^2W^{-(l-5)d/2}\, ,
\end{align*}
where we used Lemma \ref{theta-itself} in the second step and Lemma \ref{dG-bd} in the third step.
\end{proof}

It remains to prove Lemma \ref{theta-itself}. Let $\Theta(\Z_L^d, S, |m|)$ be the Green's function for the random walk on $\Z^d_L$ that has transition matrix $S$ and is killed at each step with probability $1 - |m|$, i.e., 
$$ \Theta(\Z_L^d, S, |m|)_{xy}:= \sum_{k=1}^\infty |m|^{2k}(S^k)_{xy} = \left(\frac{|m|^2 S}{1-|m|^2 S}\right)_{xy},\quad x,y \in \Z_L^d.$$
Similarly, we can define $\Theta(\Z_L, S, |m|)$.
Then, we will prove Lemma \ref{theta-itself} using the following bound.

\begin{lemma}
\label{Green-gr}
For $x,y \in \Z_L^d$ and any  constants $\tau, D>0$, we have that
\begin{align*}
&\Big| \Big( \Theta(\Z_L^d, S, |m|)_{xy} - \Theta(\Z^d, S, |m|)_{xy} \Big) - \Big( \Theta(\Z_L^d, S, |m|)_{00} - \Theta(\Z^d, S, |m|)_{00} \Big) \Big| \\
&\prec \frac{\gE{x-y}^2}{W^2 L^{d} }  \mathbf 1_{\eta \le \frac{W^{2 + \tau}}{L^2}} + L^{-D}\,.
\end{align*}
\end{lemma}
\begin{proof}
Noticing that
\begin{equation}
\label{eq:poisson}
\Theta(\Z_L^d, S, |m|)_{x,y} = \sum_{k \in \Z^d}\Theta(\Z^d, S, |m|)_{x,y+kL}\,,
\end{equation}
we have
\begin{equation}
\label{eq:knot0}
\Theta(\Z_L^d, S, |m|)_{xy} - \Theta(\Z^d, S, |m|)_{xy}= \sum_{k \in \Z^d \setminus\{0\}}\Theta(\Z^d, S, |m|)_{x,y+kL}\,.
\end{equation}
Hence, it suffices to prove that for any $x \in \Z^d_L$,
\begin{equation*}
\begin{split}
	&\sum_{ k \in \Z^d \setminus \{0\}} \Big| \Theta(\Z^d, S, |m|)_{0,x + kL} + \Theta(\Z^d, S, |m|)_{0,-x+kL}-2\Theta(\Z^d, S, |m|)_{0,kL}\Big|  \\
	&\prec  \frac{\gE{x}^2}{W^2L^{d}} \mathbf 1_{\eta \le \frac{W^{2 + \tau}}{L^2}} + L^{-D}.
\end{split}
\end{equation*}
This is a direct consequence of the transition probability estimate for the random walk on $\Z^d$ with transition matrix $S$. For example, it is proved on \cite[page 76]{BandI} that 
\begin{equation*}
\begin{split}
	&	\Big| \Theta(\Z^d, S, |m|)_{0,x + kL} + \Theta(\Z^d, S, |m|)_{0,-x+kL}-2\Theta(\Z^d, S, |m|)_{0,kL}\Big|\\
	& \prec \Big( W^{-2}(kL)^{-d}\gE{x}^2  + L^{-D} \Big) \11_{|kL| \leq \eta^{-1/2} W^{1 + \tau}} + (|k|L)^{-D}\,.
\end{split}
\end{equation*}
Summing this over $k \in \Z^d \setminus \{0\}$ yields the desired result since $\tau$ is arbitrarily small.
\end{proof}

\begin{proof}[\bf Proof of Lemma \ref{theta-itself}]
We first prove \eqref{eq:theta-itself} for $\Theta^{(l)}_2 \equiv  \Theta^\circ$. By definition, we have
\begin{align*}
&\,|\Theta^{\circ}_{xy} - (\Theta^{\circ}_{\infty})_{xy}|  \\
\leq &\, \Big| \Big( \Theta(\Z_L^d, S, |m|)_{xy} - \Theta(\Z^d, S, |m|)_{xy} \Big) - \Big( \Theta(\Z_L^d, S, |m|)_{00} - \Theta(\Z^d, S, |m|)_{00} \Big) \Big| \\
+&\, \Big| 	\Theta(\Z_L^d, S, |m|)_{00} - \Theta(\Z^d, S, |m|)_{00}  - \frac{1}{L^d}\sum_{x\in \Z_L^d}\Theta(\Z_L^d, S, |m|)_{0x} \Big|\\
+&\, |\Theta(\Z^d, S, |m|)_{xy} - \Theta(\Z^d, S_\infty, |m|)_{xy}| + |\Theta(\Z^d, S_\infty, |m|)_{xy} - \Theta(\Z^d, S_\infty, 1)_{xy}| \,.
\end{align*}
We bound the first term using Lemma \ref{Green-gr}. The second term is bounded by
\begin{align*}
&\frac{1}{L^d}\sum_{x \in \Z_L^d}\Big|	\sum_{k \not = 0}\Theta(\Z^d, S, |m|)_{0,kL}  - \sum_{k \not = 0}\Theta(\Z^d, S, |m|)_{0,x + kL} \Big| \\
&\quad  + \frac{1}{L^d}\sum_{x \in \Z_L^d}\Theta(\Z^d, S, |m|)_{0x } \prec W^{-2}L^{2-d},
\end{align*}
where we used the facts that the summands in the first term are given by
$$  \Big| ( \Theta(\Z_L^d, S, |m|)_{0x} - \Theta(\Z^d, S, |m|)_{0x} ) -  ( \Theta(\Z_L^d, S, |m|)_{00} - \Theta(\Z^d, S, |m|)_{00} ) \Big| \prec W^{-2}L^{2-d} $$
due to \eqref{eq:knot0} and Lemma \ref{Green-gr}, while the summands in the second term are bounded by $\Theta(\Z^d, S, |m|)_{0x }\prec W^{-2}\langle x\rangle^{2-d}$, a standard result for the random walk in $\Z^d$. 
Finally, we have
\begin{align*}
|\Theta(\Z^d, S, |m|)_{xy} - \Theta(\Z^d, S_\infty, |m|)_{xy}| &\leq L^{-D}\,,\\
|\Theta(\Z^d, S_\infty, |m|)_{xy} - \Theta(\Z^d, S_\infty, 1)_{xy}| &\prec \frac{\eta }{W^{4}\gE{x-y}^{d-4}} +\gE{x-y}^{-D}\,,
\end{align*}
where the first bound is due to \eqref{app compact f} and the second bound was proved as equation (7.9) in \cite{BandI}. Combining the above estimates yields \eqref{eq:theta-itself} for $l'=2$. 

For the general case, we note that by \eqref{chain S2k}, $\Theta^{(l)}_{l'}$ is a sum of graphs of the form $\Theta^{\circ}\Sele_{k_1}\Theta^{\circ}  \Sele_{k_2}\Theta^{\circ} \cdots \Theta^{\circ}  \Sele_{k_t}\Theta^{\circ}$ 	with $k_i \leq l' \leq l$. Hence, using the bound \eqref{eq:theta-itself} for $\Theta^\circ - \Theta^\circ_\infty$ and the fact that \eqref{eq:G-Gind} holds for the graphs in $\cal E_{k_i}$, $k_i\le l$, we can replace $\Theta^\circ$ and $\Sele_{k_i}$ by $\Theta^\circ_\infty$ and $\Sele_{k_i}^{\infty}$ one by one with good control of the resulting errors. The relevant proof is similar to that below equation (7.22) of \cite{BandI}, so we omit the details.
\end{proof}

\section{Proof of Lemma \ref{mGep} and Lemma \ref{gvalue_continuity}}\label{sec_pf_complete}

We define the \emph{generalized doubly connected property} of graphs with external molecules, which generalizes the doubly connected property in Definition \ref{def 2net}. 

\begin{defn}[Generalized doubly connected property]\label{def 2netex}
A graph $\cal G$ with external molecules is said to be generalized doubly connected if its molecular graph satisfies the following property. There exists a collection, say $\cal B_{black}$, of $\dashed$ (or pseudo-$\dashed$) edges, and another collection, say $\cal B_{blue}$, of blue solid, $\dashed$ (or pseudo-$\dashed$), free or ghost edges such that: (a)  $\cal B_{black}\cap \cal B_{blue}=\emptyset$, (b) every internal molecule connects to external molecules through two disjoint paths: a path of edges in $\cal B_{black}$ and a path of edges in $\cal B_{blue}$. Simply speaking, a graph is generalized doubly connected if merging all its external molecules into a single internal molecule gives a doubly connected graph in the sense of Definition \ref{def 2net}. 
\end{defn}

We will show that expanding a graph without internal molecules in a proper way always generates generalized doubly connected graphs. Moreover, in these graphs, at most half of the external molecules are used in the generalized doubly connected property in the sense of property (b) of the following lemma. 

\begin{lemma}\label{mG-fe}
Under the setting of Lemma \ref{mGep}, the deterministic graphs $\mathcal G_{\mathbf x}^{(\mu)}$ in \eqref{EGx} satisfy the following properties: 
\begin{itemize}
\item[(a)] $\mathcal G_{\mathbf x}^{(\mu)}$ are generalized doubly connected normal graphs without silent edges. 

\item[(b)] 
Consider all external molecules that are neighbors of internal molecules on the molecular graph, say $\cal M_1,\ldots, \cal M_n$. We can find at least $r \ge \lceil n/2\rceil$ of them, denoted by $\{\cal M_{k_i}\}_{i=1}^r$, that are simultaneously connected with redundant $\dashed$ or free edges in the following sense: (i) every $\cal M_{k_i}$, $1\le i \le r$, connects to an internal molecule through a $\dashed$ or free edge $e_{k_i}$, and (ii) after removing all these $r$ edges $e_{k_i}$, $1\le i \le r$, the resulting graph is still generalized doubly connected. (We will call these molecules $\cal M_{k_i}$ as special external molecules and these redundant $\dashed$ or free edges as special redundant edges.) 

\item[(c)] If $x_i$ and $x_j$ are connected in $\mathcal G_{\mathbf x}$, then they are also connected in $\mathcal G_{\mathbf x}^{(\mu)}$ through non-ghost edges.   

\item[(d)] Every $\mathcal G_{\mathbf x}^{(\mu)}$ satisfies $\size(\mathcal G_{\mathbf x}^{(\mu)})\le W^{-2\soe (p-t)}$ under the definition \eqref{eq:def-size}, where $t$  denotes the number of connected components in $\mathcal G_{\mathbf x}^{(\mu)}$.


\end{itemize}
\end{lemma}


The proof of Lemma \ref{mG-fe} will be given in Section \ref{sec_pf_mG}. To help the reader understand the statement of Lemma \ref{mG-fe}, we provide the following example, which describes a typical class of deterministic graphs that will naturally appear in the expansion of $\cal G_{\bx}$. The reader can also use this example to assist with the understanding of the proofs of Lemmas \ref{mGep} and \ref{gvalue_continuity} below.
\begin{center}
\includegraphics[width=7cm]{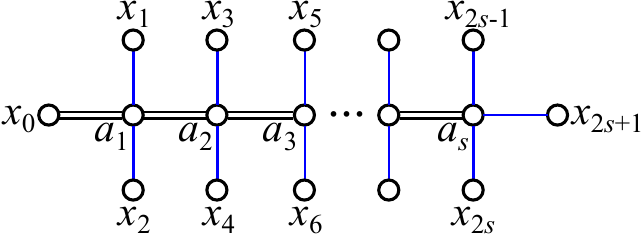}
\end{center}
This figure shows a molecular graph, where vertices $x_i$ denote external molecules and vertices $a_i$ denote internal molecules. The black double-line edges represent diffusive edges in the black net, and the blue solid lines represent diffusive or free edges in the blue net. Notice that among the two blue edges connected with an internal molecule $a_i$, $1\le i \le s-1$, one of them is redundant for the generalized doubly connected property, and $a_s$ is connected with two redundant blue edges. Hence, we can choose $x_1, x_3 ,\ldots, x_{2s+1} $ as special external molecules among the $2s+2$ external molecules in the above graph.

\subsection{Proof of Lemma \ref{mGep}}

In this subsection, we complete the proof of Lemma \ref{mGep} using Lemma \ref{mG-fe}. It remains to prove \eqref{EGx2} and the properties (a)--(c). Given a graph \smash{$\mathcal G_{\mathbf x}^{(\mu)}$} satisfying Lemma \ref{mG-fe}, similarly to Lemma \ref{GtoAG}, we first remove all local structures inside its molecules and bound it by an auxiliary graph consisting of external atoms $\bx_1,\cdots, \bx_p$, internal atoms that are representatives of internal molecules, and edges between them. More precisely, we reduce \smash{$\mathcal G_{\mathbf x}^{(\mu)}$} to an auxiliary graph together with some $W$ factors as follows. First, we pick a representative atom for each internal and external molecule (note that two external atoms can be in the same molecule if they are connected through a path of waved edges). Second, in light of \eqref{app compact f}, \eqref{thetaxy}, \eqref{S+xy} and \eqref{BRB}, we bound every waved edge inside a molecule by a pseudo-waved edge (up to a negligible error $W^{-D}$) connected with the representative atom of the molecule and bound every $\dashed$ edge by a pseudo-diffusive edge (with some $W$ factors if the $\dashed$ edge is labeled and of scaling order $>2$). Third, we bound every pseudo-$\dashed$ edge between two molecules with representatives, say $x$ and $y$, by $W^{(d-2)\tau}B_{xy}$. Finally, we bound the subgraph inside a molecule as follows: we keep one pseudo-waved edge between two neighboring external atoms; we bound every other pseudo-waved or pseudo-$\dashed$ edge inside the molecule by $W^{-d}$; we bound every free edge inside the molecule by $(N\eta)^{-1}$; we bound every summation over a non-representative internal atom by $W^{d}$. 


Denote the auxiliary graph thus obtained by $\wt{\cal G}_{\bx}$. It also satisfies the properties (a)--(c) of Lemma \ref{mG-fe} and that 
\be\label{eq_bdd_aux0} |\cal G_{\bx}^{(\mu)} | \le W^{C\tau} c(W,L) \cdot \wt{\cal G}_{\bx} + W^{-D},\ee
where $C>0$ is a constant that does not depend on $\tau$ and $c(W,L)>0$ is a $(W,L)$-dependent coefficient so that 
\be\label{eq_bdd_aux}
c(W,L) \size \left(\wt{\cal G}_{\bx}\right) \le \size\left(\cal G_{\bx}^{(\mu)}\right).
\ee
We still need to bound the summations over the internal atoms in $\wt{\cal G}_{\bx}$. With the generalized doubly connected property, we can choose a collection of blue trees such that every internal atom connects to an external atom through a unique path on a blue tree. These blue trees are disjoint, and each of them contains an external atom as the root. Then, we will sum over all internal atoms from the leaves of the blue trees to the roots. In bounding every summation, we will use the following two estimates with $k\ge 1$ and $r\ge 0$: 
\be\label{keyobs4.3}
\begin{split}
& \sum_{x } \prod_{j=1}^k B_{ x y_j}\cdot \prod_{s=1}^r \wt B_{ x z_s} \cdot B_{x a} \\ 
&\prec \sum_{l=1}^k \prod_{j: j\ne l}B_{y_j y_{l}} \cdot \Big( \wt B_{y_l a} \prod_{s=1}^r \wt B_{z_s a}+ \wt B_{y_l a} \prod_{s=1}^r \wt B_{z_s y_l} + \sum_{t=1}^r  \wt B_{y_l z_t}  \wt B_{z_t a} \prod_{s:s\ne t} \wt B_{z_s z_t} \Big) ,
\end{split}
\ee
\be\label{keyobs4.3_add}
\sum_{x} \prod_{j=1}^k B_{ x y_j}\cdot \prod_{s=1}^r \wt B_{ x z_s} \cdot \frac{W^2}{L^2}\prec \sum_{l=1}^k \prod_{j: j\ne l}B_{y_j y_{l}} \cdot \Big(  \prod_{s=1}^r \wt B_{z_s y_l} + \sum_{t=1}^r  \wt B_{y_l z_t}  \prod_{s:s\ne t} \wt B_{z_s z_t} \Big) .
\ee
The left-hand side of \eqref{keyobs4.3} is a star graph consisting of $k$ pseudo-diffusive edges (where at least one of them is in the black tree), a pseudo-diffusive edge in the blue tree, and $r$ silent pseudo-diffusive edges connected with $x_i$, while every graph on the right-hand side is a connected graph consisting of $(k-1)$ pseudo-diffusive edges and $(r+1)$ silent pseudo-diffusive edges. The estimate \eqref{keyobs4.3_add} can be applied to the case where the blue edge is a ghost or free edge. Of course, there may be some free edges connected with $x$, in which case we can still use \eqref{keyobs4.3} and \eqref{keyobs4.3_add} by multiplying some $(N\eta)^{-1}$ factors with them. 

Using \eqref{keyobs4.3} or \eqref{keyobs4.3_add}, we can bound a summation over an internal atom by a sum of new graphs. Some edges in the new graphs may be self-loops on atoms, which are picked out as $W$-dependent factors. Moreover, when we lose a ghost edge in the summation, we pick out the corresponding $L^2/W^2$ factor in its coefficient. It is easy to see that every new graph still satisfies the properties (a) and (c) of Lemma \ref{mG-fe}, and its size multiplied with the $W$-dependent factor associated with it is at most the size of \smash{$\wt{\cal G}_{\bx}$}. We still need to show that every special redundant edge in property (b) will be ``preserved" in a certain sense during summations over internal atoms and will finally contribute to a non-silent edge.

Suppose we sum over an internal atom $x$ as follows:  
\begin{center}
\includegraphics[width=3.2cm]{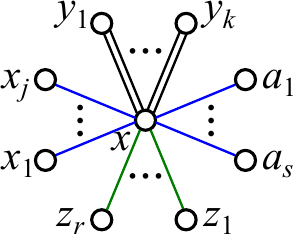}
\end{center}
Here, $(a_1,x),\ldots, (a_s,x)$ indicate pseudo-$\dashed$ or free edges, with one of them being the leaf of a blue tree, $(y_1,x),\ldots, (y_k,x)$ are pseudo-$\dashed$ edges in the black net, green edges $(z_1,x),\ldots, (z_r,x)$ indicate silent pseudo-$\dashed$ or silent free edges, and edges $(x_1,x),\cdots, (x_j,a)$ indicate special redundant edges connected with external atoms. When $(x_i,x)$ is a redundant free edge, we can treat it as a free edge between $x_i$ and another external atom. 
Hence, without loss of generality, we assume that $(x_1,x),\cdots, (x_j,a)$ are all pseudo-diffusive edges. Then, we apply \eqref{keyobs4.3} or \eqref{keyobs4.3_add} to bound the summation over $x$. In every resulting graph on the right-hand side, one of the following cases happens for a special redundant edge $(x_i,x)$.
\begin{itemize}
\item If $(x_i,x)$ becomes an edge connecting $x_i$ to another internal atom, it is still special redundant.

\item If $(x_i,x)$ becomes an edge connecting $x_i$ to another external atom, it gives a non-silent pseudo-$\dashed$ edge between external atoms and associated with $x_i$. 

\item Suppose $(x_i,x)$ is lost after the summation and there is no black pseudo-diffusive edge between $x_i$ and $x$ before the summation. Then, we apply \eqref{keyobs4.3} or \eqref{keyobs4.3_add} with the edge in the blue tree playing the role of $B_{xa}$ in \eqref{keyobs4.3} or the factor $W^2/L^2$ in \eqref{keyobs4.3_add}. In the new graph where $(x_i,x)$ is lost, $y_1, \cdots, y_k$ are still connected to each other through pseudo-diffusive edges and they are connected to $x_i$ through a black pseudo-diffusive edge, which is redundant in the black net.

\item Suppose $(x_i,x)$ is lost after the summation and there is at least one black pseudo-diffusive edge between $x_i$ and $x$ before the summation. Then, we apply \eqref{keyobs4.3} with an $(x_i,x)$ edge playing the role of $B_{xa}$ in \eqref{keyobs4.3}. If the edge, say $e$, in the blue tree before the summation is a pseudo-diffusive or free edge, then in the new graph where $(x_i,x)$ is lost, the edge $e$ becomes a redundant non-silent blue edge connected with $x_i$. If the edge $e$ is a ghost edge, then in the new graph where $(x_i,x)$ is lost, we gain an extra $W^2/L^2$ factor in the coefficient. In the final simple auxiliary graphs, this $W^2/L^2$ factor will turn a silent edge connected with $x_i$ into a non-silent edge. 
\end{itemize}

To sum up, after each summation over an internal atom, every special redundant edge connected with an external atom $x_i$ either becomes a non-silent edge between $x_i$ and another external atom or becomes another special redundant edge connected with $x_i$. Hence, after summing over all internal atoms, the special redundant edges all lead to non-silent edges between external atoms. Furthermore, we observe the simple fact that the edges between external atoms are not affected throughout the summations, so every non-isolated external atom that is not a neighbor of internal atoms in $\wt{\cal G}_{\bx}$ is connected with a non-silent edge. The above arguments show that after summing over all internal atoms, we can bound $\wt{\cal G}_{\bx}$ by a sum of simple auxiliary graphs satisfying the properties (a) and (b) in Lemma \ref{mGep} and of size $\le \size(\wt{\cal G}_{\bx})$. Combining these facts with \eqref{eq_bdd_aux0} and \eqref{eq_bdd_aux} yields \eqref{EGx2} and the property (c), since the constant $\tau$ can be arbitrarily small. This concludes the proof. 

\subsection{Proof of Lemma \ref{gvalue_continuity}}

In this subsection, we complete the proof of Lemma \ref{gvalue_continuity} with Lemma \ref{mGep}. We first decompose the sum over $\bx\in \cal I^p$ into different cases according to whether every pair of atoms $x_i$ and $x_j$ take the same value or not. We fix one case and identify atoms that take the same value. This gives a connected graph with $1\le q\le p$ external atoms. Without loss of generality, we denote these external atoms by $\wt\bx= \{x_i: 1\le i \le q\}$ and the graph by $\wt {\cal G}_{\wt\bx}$. Applying Lemma \ref{mGep}, we obtain that \be\label{EGx2.5}
\E \wt {\cal G}_{\wt\bx} \prec  \sum_{\gamma}c_\gamma(W,L) {\cal G}_{\wt\bx}^{(\gamma)}+ \OO(W^{-D})\, ,
\ee
where every $\mathcal G_{\wt{\mathbf x}}^{(\gamma)}$ is a connected simple auxiliary graph satisfying the properties (b) and (c) in Lemma \ref{mGep} with $t=1$. By the property (c), we have that
\be\label{cgamma}
c_\gamma(W,L) \le  W^{2\soe (k_\gamma-p+1)},\quad  k_\gamma:=\#\{\text{edges in } {\cal G}_{\wt\bx}^{(\gamma)}\}.
\ee
It remains to bound 
\be\label{eq_reduce_to_aux}
\frac{1}{|\cal I|^q}\sum_{x_i \in \cal I, i \in [q]} {\cal G}_{\wt\bx}^{(\gamma)},\quad [q]:=\{1,\cdots, q\}. 
\ee
By property (b), there are $n\ge \lceil q/2\rceil$ special external atoms, each of which is associated with a unique \emph{special non-silent edge}. Since every atom is connected with at least one silent/non-silent pseudo-diffusive or free edge, its average over $\cal I$ provides at least a factor
\be\label{bad factor}
\frac{1}{W^4 K^{d-4}}+ \frac{1}{N\eta}\frac{L^2}{W^2}.
\ee
In addition, we will show that each average over a special external atom connected with a special non-silent edge contributes a better factor
\be\label{good factor}
\frac{1}{W^2 K^{d-2}}+ \frac{1}{N\eta} + \frac{1}{K^{d/2}} \sqrt{\frac{1}{N\eta}\frac{L^2}{W^2}}.
\ee

In bounding \eqref{eq_reduce_to_aux}, we first estimate averages over the special external atoms. We use the following two estimates to bound the average over such an external atom $x$ connected with a special diffusive edge: for $k\ge 1$,  
\be\label{keyobs4}
\begin{split}
& \frac{1}{|\cal I|}\sum_{x\in \cal I} B_{xy_1}^2 \prod_{j=2}^k B_{ x y_j}\cdot \prod_{s=1}^r \wt B_{ x z_s} \prec  \frac{W^{-d}}{ K^d} \sum_{l=1}^k  \sum_{t=1}^r \prod_{j\in[k]: j\ne l}B_{y_j y_{l}} \cdot \wt B_{ y_1 z_t} \cdot   \prod_{s:s\ne t} \wt B_{z_s z_t}  ,
\end{split}
\ee
\be\label{keyobs4_add}
\frac{1}{|\cal I|}\sum_{x\in \cal I} \prod_{j=1}^k B_{ x y_j}\cdot \prod_{s=1}^r \wt B_{ x z_s}  \prec \frac{1}{W^2 K^{d-2}} \sum_{l=1}^k  \sum_{t=1}^r \prod_{j\in [k]: j\ne l}B_{y_j y_{l}} \cdot \wt B_{y_l z_t} \cdot   \prod_{s:s\ne t} \wt B_{z_s z_t} .
\ee
Of course, if there are free or silent free edges connected with $x$, we can still use \eqref{keyobs4} and \eqref{keyobs4_add} by multiplying them with some $(N\eta)^{-1}$ and $ {L^2}/{W^2}$ factors. 
We apply the above two estimates in two different cases. In the first case, suppose the special diffusive edge of $x$ is paired with the special diffusive edge of $y_1$. Then, we use \eqref{keyobs4} to bound the average over $x$ by a factor $W^{-d}K^{-d}$ times a sum of new graphs, each of which is still connected and has \emph{two fewer} special external atoms. If the first case does not happen, then we use \eqref{keyobs4_add} to bound the average over $x$ by a factor $W^{-2}K^{-(d-2)}$ times a sum of new graphs, each of which is still connected and has \emph{one fewer} special external atom.   

Second, we sum over special external atoms connected with non-silent free edges. Given such an external atom $x$, we use the trivial identity $|\cal I|^{-1}\sum_{x\in \cal I}1=1$ if $x$ is only connected with silent/non-silent free edges, the estimate \eqref{keyobs4_add} if $x$ is connected with at least one pseudo-diffusive edge, or the following estimate if $x$ is connected with silent pseudo-diffusive edges and silent/non-silent free edges:
\be\label{keyobs5_add}
\frac{1}{|\cal I|}\sum_{x\in \cal I} \prod_{s=1}^r \wt B_{ x z_s}  \prec \frac{1}{W^4 K^{d-4}}   \sum_{t=1}^r \prod_{s:s\ne t} \wt B_{z_s z_t} .
\ee
In this way, we can bound the average over $x$ by a factor $(N\eta)^{-1}$ times a sum of new graphs, each of which is still connected and has one fewer special external atom. Next, summing over a non-special external atom using $|\cal I|^{-1}\sum_{x\in \cal I}1=1$ or \eqref{keyobs5_add} yields a factor \eqref{bad factor} times a sum of new connected graphs. Finally, the average over the last atom is equal to $1$.  

In sum, we can bound \eqref{eq_reduce_to_aux} as
\begin{align}
 \frac{1}{|\cal I|^q}\sum_{x_i \in \cal I, i \in [q]} {\cal G}_{\wt\bx}^{(\gamma)} \label{eq_reduce_to_aux2}
&\prec W^{- 2\soe(k_{\gamma}-q+1)}\left( \frac{1}{K^d}\right)^{n_1} \left( \frac{1}{W^2K^{d-2}} + \frac{1}{N\eta}\right)^{n-2n_1}  \\
&\quad \times \left( \frac{1}{W^4K^{d-4}} + \frac{1}{N\eta}\frac{L^2}{W^2}\right)^{q-1-(n-n_1)} \nonumber\\
&\prec W^{-2\soe(k_{\gamma}-q+1)} \left( \frac{1}{W^2K^{d-2}} + \frac{1}{N\eta}+ \frac{1}{K^{d/2}} \sqrt{\frac{1}{N\eta}\frac{L^2}{W^2}}\right)^{n } \nonumber\\
&\quad \times \left( \frac{1}{W^4K^{d-4}} + \frac{1}{N\eta}\frac{L^2}{W^2}\right)^{q-1-n},\nonumber
\end{align}
where $q/2 \le n \le q-1$ is the number of special external atoms and $0\le n_1\le n/2$ is the number of times that \eqref{keyobs4} has been applied. Here, the factor $W^{-2\soe(k_{\gamma}-q+1)}$ comes from $k_\gamma - q +1$ silent/non-silent pseudo-diffusive and free edges that become self-loops during the summations (for example, the $W^{-d}$ factor in \eqref{keyobs4} comes from a pseudo-diffusive edge that becomes a self-loop in the summation), and we have applied trivial bounds $W^{-d}\le W^{-2\soe}$ and $\frac{1}{N\eta}\frac{L^2}{W^2}\le W^{-2\soe}$ to them.  Now, combining \eqref{eq_reduce_to_aux2}, \eqref{cgamma} and \eqref{EGx2.5}, we get that 
\begin{align*}
&\quad \frac{1}{|\cal I|^p}\sum_{x_i \in \cal I, i \in [p]}\mathcal G_{\mathbf x}(z)\\
&\prec \sum_{q \in [p]} \sum_{ q/2 \le n\le q-1} \frac{1}{|\cal I|^{p-q}} \left( \frac{1}{W^2K^{d-2}} + \frac{1}{N\eta}+ \frac{1}{K^{d/2}} \sqrt{\frac{1}{N\eta}\frac{L^2}{W^2}}\right)^{n }  \\ 
&\quad \times \left( \frac{1}{W^4K^{d-4}} + \frac{1}{N\eta}\frac{L^2}{W^2}\right)^{q-1-n}  \\
&\prec \left[\left( \frac{1}{W^2K^{d-2}} + \frac{1}{N\eta}+ \frac{1}{K^{d/2}} \sqrt{\frac{1}{N\eta}\frac{L^2}{W^2}}\right) \left( \frac{1}{W^4K^{d-4}} + \frac{1}{N\eta}\frac{L^2}{W^2}\right)\right]^{\frac{p-1}{2}}.
\end{align*}	
This concludes the proof of Lemma \ref{gvalue_continuity}.

\begin{appendix}

\section{Construction of the $T$-equation} \label{sec notation}

In this section, we present the proof of Proposition \ref{Teq}. Since it is very similar to that for Theorem 3.7 in \cite{BandII}, we only describe an outline of the proof and point out the main difference from the argument in \cite{BandII}.  

\subsection{More graphical properties}

The graphs in the $T$-expansion and $\incomp$ satisfy some stronger structural properties than the doubly connected property, which we define now one by one. All these properties have been defined in \cite{BandII}, but some minor modifications are needed to incorporate a new type of edges, i.e., free edges, in our graphs.

\begin{defn}[Isolated subgraphs]\label{defn iso}
Let $\cal G$ be a doubly connected graph and $\cal G_{\cal M}$ be its molecular graph with all red solid edges removed. A subset of internal molecules in $\cal G$, say $\Pol$, is called \emph{isolated} if and only if $\Pol$ is connected to its complement $\Pol^c$ exactly by two edges in $\cal G_{\cal M}$---a $\dashed$ edge in the black net and a blue solid, free or $\dashed$ edge in the blue net. An isolated subgraph of $\cal G$ is a subgraph induced on an isolated subset of molecules.
 \end{defn}
 
An isolated subgraph of $\cal G$ is said to be \emph{proper} if it is induced on a proper subset of internal molecules of $\cal G$. An isolated subgraph is said to be \emph{minimal} if it has no proper isolated subgraph. As a convention, if a graph $\cal G$ does not contain any proper isolated subgraph, then the minimal isolated subgraph (MIS) refers to the subgraph induced on all internal molecules. On the other hand, given a doubly connected graph $\cal G$, an isolated subgraph is said to be \emph{maximal} if it is not a proper isolated subgraph of $\cal G$.

\begin{defn}[Pre-deterministic property]\label{def PDG}
A doubly connected graph $\mathcal G$ is said to be pre-deterministic if there exists an order of all internal blue solid edges say $b_1 \preceq b_2 \preceq ... \preceq b_k$, such that 
\begin{itemize}
\item[(i)] $b_1$ is redundant (recall \Cref{def-redundant});
\item[(ii)] for $1 \leq i \leq {k-1}$, if we replace each of $b_1,...,b_i$ by a $\dashed$ or free edge, then $b_{i+1}$ becomes redundant.
\end{itemize}
\end{defn}

\begin{defn}[Sequentially pre-deterministic property]\label{def seqPDG}
A doubly connected graph $\mathcal G$ 
is said to be sequentially pre-deterministic (SPD) if it satisfies the following properties.
\begin{itemize}

\item[(i)] All isolated subgraphs of $\mathcal G$ that have non-deterministic closure (recall Definition \ref{def_sub}) forms a sequence $(\Iso_{j})_{j=0}^k$ such that
\begin{equation}
\Iso_0 \supset \Iso_{1} \supset \cdots \supset \Iso_{k},
\end{equation}
where $\Iso_0$ is the maximal proper isolated subgraph and $\Iso_k$ is the MIS.

\item[(ii)] 
The MIS $\Iso_k$ is pre-deterministic. Let $\mathcal G_{\mathcal M}$ be the molecular graph without red solid edges.  For any $0 \leq j \leq k-1$, if we replace $\Iso_{j+1}$ and its two external edges in $\cal G_{\mathcal M}$ by a single $\dashed$ or free edge, then $\Iso_j$ becomes pre-deterministic. 
\end{itemize}

\end{defn}

By definition, a subgraph has a non-deterministic closure if it contains $G$ edges and weights inside it, or if it is connected with at least one external $G$ edge. Moreover, by the definition of isolated subgraphs, the two external edges in property (ii) are exactly the black and blue external edges in the black and blue nets.

\begin{defn}[Globally standard graphs]\label{defn gs}
A doubly connected graph $\cal G$ is said to be \emph{globally standard} if it is SPD and its proper isolated subgraphs are all weakly isolated. Here, an isolated subgraph is said to be {weakly isolated} if it has at least two external red solid edges; otherwise, it is said to be strongly isolated. 
\end{defn}

We refer the reader to Sections 5 and 6 of \cite{BandII} for the motivations and intuitions behind these definitions and some detailed explanations about their meanings. We are now ready to complete the definitions of the $T$-expansion and $\incomp$.

\begin{defn} [$T$-expansion and $\incomp$: additional properties]\label{def genuni2}
The graphs in Definitions \ref{defn genuni} and \ref{def incompgenuni} satisfy the following properties:
\begin{enumerate}
		\item Graphs of scaling order $k$ in $\PTn_{x,\fb_1 \fb_2}$, $\Wn_{x,\fb_1\fb_2}$ and $\QTn_{x,\fb_1\fb_2}$ do not depend on $n$.
		
		\item $\PTn_{x,\fb_1\fb_2}$ is a sum of globally standard recollision graphs without free edges. 
		
		\item  $\ATn_{x,\fb_1\fb_2}$ is a sum of SPD graphs without free edges. 
		
		\item  $\Wn_{x,\fb_1\fb_2}$ is a sum of SPD graphs, each of which has exactly one redundant free edge in its MIS. 
		
		\item  $\QTn_{x,\fb_1\fb_2} $ is a sum of SPD $Q$-graphs without free edges. Moreover, the atom in the $Q$-label of a $Q$-graph belongs to the MIS, i.e., all solid edges and weights have the same $Q$-label $Q_x$ for an atom $x$ inside the MIS.

		\item Every $\Sele_l$ is a sum of globally standard deterministic graphs of scaling order $l$ and without free edges. Moreover, the labelled diffusive edges in $\Sele_l$ are all $\Theta^{(i)}_k$ edges (recall \eqref{chain S2k}) with $2\le k\le i\le l-1$.
	
	\end{enumerate}
\end{defn}

The following lemma shows that every $\dashed$ edge in $\Sele_l$ is redundant as stated in \Cref{Teq}.
\begin{lemma}\label{-edc}
	Let $\mathcal G$ be a globally standard deterministic graph without ghost or free edges. Then, any diffusive edge in $\mathcal G$, say $e = (a,b)$, is redundant in $\cal G$.
\end{lemma}
\begin{proof}
The statement is trivial if $e$ is inside a molecule. If $e$ is between different molecules, then the conclusion follows from the fact that $\cal G$ has no isolated subgraphs. 
\end{proof}

\subsection{Local expansions}\label{sec_operations}

In this subsection, we define several basic graph operations that will be used in the construction of the $\incomp$.
We first define two operations related to dotted edges. 
\begin{defn}[Dotted edge operations]\label{dot_operation}
(i) For any pair of atoms $\al$ and $\beta$ in a graph $\cal G$, if there is at least one $G$ edge but no $\times$-dotted edge between them, then we write $1=\mathbf 1_{\al=\beta} + \mathbf 1_{\al \ne \beta} ;$ if there is a $\times$-dotted edge $ \mathbf 1_{\al\ne \beta}$ but no $G$ edge between them, then we write $\mathbf 1_{\al\ne \beta} =1 - \mathbf 1_{\al = \beta}.$  Expanding the product of all these sums, we can expand $\cal G$ as
\be\label{odot}
\cal G := \sum {\Dot} \cdot \cal G,
\ee
where each ${\Dot}$ is a product of dotted and $\times$-dotted edges together with a $\pm$ sign. In ${\Dot} \cdot \cal G$, every $\times$-dotted edge is associated with at least one off-diagonal $G$ edge, and all diagonal $G$ edges become weights on atoms.

 \vspace{5pt}
 
 \noindent (ii) We merge internal atoms that are connected by a path of dotted edges (but we sometimes do not merge an external atom with an internal atom due to their different roles in graphs).
\end{defn}
The only reason for introducing these two almost trivial operations is to write a graph into a sum of 
normal graphs (recall Definition \ref{defnlvl0}), whose scaling orders are well-defined. 

Using \eqref{GmH} and Gaussian integration by parts, we obtain the following three types of graph expansions. 

\begin{lemma}[Weight expansions, Lemma 3.5 of \cite{BandI}] \label{ssl} 
Suppose $f$ is a differentiable function of $G$. Then,
\begin{align} 
  (G_{xx}  -  m) f(G)&=  m \sum_{  \al} s_{x\al}  (G_{xx}-m) (G_{\al\al}-m)f (G) \nonumber\\
  & +m \sum_{  \al,\beta}S^{+}_{x\al} s_{\al \beta}  (G_{\al\al}-m) (G_{\beta\beta}-m)f (G) \nonumber\\
 &-  m  \sum_{ \al} s_{x \al} G_{\al x}\partial_{ h_{ \al x}} f (G) -  m \sum_{ \al,\beta} S^{+}_{x\al}s_{\al \beta} G_{\beta \al}\partial_{ h_{ \beta\al}} f(G) + \cal Q_w \, ,\label{Owx}\end{align}
 where 
 $\cal Q_w$ is a sum of $Q$-graphs,
 \begin{align} 
  \cal Q_w &:=  Q_x \left[(G_{xx}  -  m) f(G)\right] + \sum_{\al} Q_\al \left[  S^{+}_{x\al}(G_{\al\al}  -  m) f(G)\right] \nonumber\\
 & - m  Q_x\Big[ \sum_{\al} s_{x\al}  (G_{\al\al}-m)G_{xx} f(G) \Big]  - m \sum_\al Q_\al\Big[ \sum_{ \beta}S_{x\al}^{+}s_{\al\beta}(G_{\beta\beta}-m)G_{\al\al} f(G) \Big] \nonumber\\
  &+  m  Q_x  \Big[  \sum_\al s_{x \al}  G_{\al x}\partial_{ h_{\al x}} f(G)\Big]+ m \sum_\al Q_\al \Big[  \sum_{ \beta} S^{+}_{x\al}s_{\al \beta} G_{\beta \al}\partial_{ h_{\beta \al}} f(G)\Big]. \nonumber
\end{align} 
\end{lemma} 

\begin{lemma}[Edge expansions, Lemma 3.10 of \cite{BandI}] \label{Oe14}
Suppose $f$ is a differentiable function of $G$. Consider a graph
\be\label{multi setting}
\cal G := \prod_{i=1}^{k_1}G_{x y_i}  \cdot  \prod_{i=1}^{k_2}\overline G_{x y'_i} \cdot \prod_{i=1}^{k_3} G_{ w_i x} \cdot \prod_{i=1}^{k_4}\overline G_{ w'_i x} \cdot f(G),
\ee
where the atoms $y_i,$ $y'_i$, $w_i$ and $w'_i$ are all not equal to $x$. If $k_1\ge 1$, then we have that
\begin{align} 
  \cal G  : = \sum_{i=1}^{k_2} &|m|^2  \left( \sum_\al s_{x\al }G_{\al y_1} \overline G_{\al y'_i}\right)\frac{\cal G}{G_{x y_1} \overline G_{xy_i'}}+ \sum_{i=1}^{k_3} m^2 \left(\sum_\al s_{x\al }G_{\al y_1} G_{w_i \al} \right)\frac{\cal G}{G_{xy_1}G_{w_i x}}    \nonumber \\
& + m   \sum_\al s_{x\al }\left(G_{\al \al}-m\right) \cal G   + \sum_{i=1}^{k_2}m  (\overline G_{xx} -\overline m) \left( \sum_\al s_{x\al }G_{\al y_1} \overline G_{\al y'_i}\right)\frac{\cal G}{G_{x y_1} \overline G_{xy_i'}}   \nonumber \\
&  + \sum_{i=1}^{k_3} m  (G_{xx}-m) \left(\sum_\al s_{x\al }G_{\al y_1} G_{w_i \al} \right)\frac{\cal G}{G_{xy_1}G_{w_i x}}   \nonumber\\
 &+(k_1-1) m  \sum_\al s_{x\al } G_{x \al} G_{\al y_1}\frac{ \cal G}{G_{xy_1}}   + k_4 m   \sum_\al s_{x\al }\overline G_{\al x} G_{\al y_1}  \frac{\cal G}{G_{xy_1}} \nonumber\\
 &- m    \sum_\al s_{x\al } \frac{\cal G}{G_{x y_1}f(G)}G_{\al y_1}\partial_{ h_{\al x}}f (G)  +\cal Q_{e} \, ,\label{Oe1x}
\end{align}
where $\cal Q_{e} $ is a sum of $Q$-graphs, 
\begin{align*}
  \cal Q_{e} &:= Q_x \left( \cal G\right) -  \sum_{i=1}^{k_2}m Q_x \left[ \overline G_{xx}  \left( \sum_\al s_{x\al }G_{\al y_1} \overline G_{\al y'_i}\right)\frac{\cal G}{G_{x y_1} \overline G_{xy_i'}} \right] \\
& - \sum_{i=1}^{k_3} m Q_x \left[ G_{xx} \left(\sum_\al s_{x\al }G_{\al y_1} G_{w_i \al} \right)\frac{\cal G}{G_{xy_1}G_{w_i x}}\right] -  m Q_x \left[  \sum_\al s_{x\al }\left(G_{\al \al}-m\right) \cal G\right]\\
& -(k_1-1) m Q_x \left[  \sum_\al s_{x\al }G_{x \al} G_{\al y_1}  \frac{ \cal G}{G_{xy_1}}\right]  - k_4 m Q_x \left[  \sum_\al s_{x\al }\overline G_{\al x} G_{\al y_1}  \frac{\cal G}{G_{xy_1}}\right] \\
& +m Q_x \left[  \sum_\al s_{x\al } \frac{\cal G}{G_{x y_1}f(G)}G_{\al y_1}\partial_{h_{\al x}}f(G)  \right].
 \end{align*} 
 \end{lemma}
 
\begin{lemma} [$GG$ expansion, Lemma 3.14 of \cite{BandI}]\label{T eq0}
Consider a graph $\cal G= G_{xy}   G_{y' x }  f (G)$ where $f$ is a differentiable function of $G$ and $y,y'\ne x$. We have that 
\begin{align}
 \cal G = m &S^{+}_{xy} G_{y' y} f(G)  + m \sum_\al  s_{x\al} (G_{\al \al}-m) \cal G + m \sum_{\al,\beta}  S^{+}_{x\al}  s_{\al\beta} (G_{\beta \beta}-m) G_{\al y}   G_{y'\al} f(G) \nonumber \\
    &  + m(G_{xx }-m)   \sum_\al s_{x\al}G_{\al y}   G_{y'\al} f(G) + m \sum_{\al,\beta}  S^+_{x\al} s_{\al\beta}  (G_{\al\al }-m) G_{\beta y}   G_{y'\beta} f(G)  \nonumber\\
    & - m \sum_{ \al}  s_{x\al}G_{\al y} G_{y' x} \partial_{ h_{\al x}}f(G) - m \sum_{\al,\beta} S^{+}_{x\al} s_{\al\beta}G_{\beta y} G_{y' \al} \partial_{ h_{\beta \al}}f(G) + \cal Q_{GG} ,\label{Oe2x}
    \end{align}
where $\cal Q_{GG} $ is a sum of $Q$-graphs,
    \begin{align}
  \cal Q_{GG}&:=  Q_x \left(\cal G\right)+  \sum_\al Q_\al\Big[ S^+_{x\al}  G_{\al y}   G_{y'\al} f(G) \Big]  - m Q_y\Big[S^{+}_{xy}  G_{y' y} f(G)\Big] - m Q_x\Big[\sum_\al  s_{x\al} (G_{\al\al}-m) \cal G\Big]  \nonumber\\
   &- m \sum_\al Q_\al\Big[\sum_{ \beta}  S^{+}_{x\al}  s_{\al\beta} (G_{\beta \beta}-m)G_{\al y}   G_{y'\al} f(G)\Big] - mQ_x\Big[ G_{xx }  \sum_\al s_{x\al} G_{\al y}   G_{y'\al} f(G)\Big] \nonumber\\
&- m\sum_\al Q_\al\Big[ \sum_{ \beta}  S^{+}_{x\al}  s_{\al\beta} G_{\al\al }  G_{\beta y}   G_{y'\beta} f(G)\Big]  + m Q_x\Big[\sum_{ \al}  s_{x\al} G_{\al y} G_{y' x} \partial_{ h_{ \al x}}f(G)\Big] \nonumber\\
& + m\sum_\al Q_\al\Big[ \sum_{\beta} S^{+}_{x\al} s_{\al\beta}G_{\beta y} G_{y' \al} \partial_{h_{\beta \al}}f (G)\Big] .\nonumber 
 \end{align}
\end{lemma} 

Using the above three lemmas, we can define the graph operations that represent weight, edge, and $GG$ expansions. We refer the reader to Section 3 of \cite{BandI} for their precise definitions. All these operations are called \emph{local expansions on the atom $x$}, in the sense that they do not create new molecules (since all new atoms created in these expansions connect to $x$ through paths of waved edges) in contrast to the global expansion that will be defined in Section \ref{sec_goperations} below. 

The goal of local expansions is to expand a graph into a sum of \emph{locally standard graphs}. To explain this concept, we first define \emph{standard neutral atoms}.

\begin{defn}[{Standard neutral atoms}]\label{def SNA}
An atom in a normal graph is said to be \emph{standard neutral} if it satisfies the following two properties:
\begin{itemize}
    \item it has a neutral charge, where the charge of an atom is defined by counting its incoming and outgoing plus $G$ (i.e., blue solid) edges and minus $G$ (i.e., red solid) edges: 
$$\#\{\text{incoming $+$ and outgoing $-$ $G$ edges}\}- \#\{\text{outgoing $+$ and incoming $-$ $G$ edges}\} ;$$

\item it is only connected with three edges besides the $\times$-dotted edges: a $G$ edge, a $\overline G$ edge, and a neutral waved edge (i.e., an $S$ edge). 
\end{itemize}
 \end{defn}
 
By definition, the edges connected with a standard neutral atom, say $\al$, take the form 
\begin{equation}\label{12in4T}
t_{x,y_1 y_2}=\sum_\al s_{x\al}G_{\al y_1}\overline G_{\al y_2}\mathbf 1_{\al\ne y_1}\mathbf 1_{\al\ne y_2}\ \ \text{or}\ \ t_{y_1y_2, x}=\sum_\al G_{ y_1\al }\overline G_{y_2 \al}s_{\al x}\mathbf 1_{\al\ne y_1}\mathbf 1_{\al\ne y_2} ,
\end{equation}
which are essentially the $T$-variables (except for the two $\times$-dotted edges and the missing coefficient $|m|^2$). We define \emph{locally standard graphs} as graphs that only contain standard neutral atoms or atoms that are not connected with any $G$ edge.

 \begin{defn} [Locally standard graphs] \label{deflvl1}
A graph is  \emph{locally standard}  if 
\begin{itemize}
\item[(i)] it is a normal graph without $P/Q$ labels; 

\item[(ii)] it has no weights or light weights;

\item[(iii)] any internal atom is either standard neutral or connected with no solid edge. 
\end{itemize}
 \end{defn}

As discussed in Section 3.4 of \cite{BandI}, applying local expansions repeatedly, we can expand any normal graph into a linear combination of locally standard, recollision, higher order, or $Q$ graphs.

\begin{lemma} \label{lvl1 lemma}
Let $\mathcal G$ be a normal graph. Then, for any fixed $n\in \N$, we can expand $\cal G$ into a sum of $\OO(1)$ many graphs:
\begin{align}\label{expand lvl1}
\mathcal G =\mathcal G_{local} + \ATn  + \QGn ,
\end{align}
where $\mathcal G_{local} $ is a sum of locally standard graphs,  
$\ATn$ is a sum of graphs of scaling order $> n$, and $\QGn$ is a sum of $Q$-graphs. Some of the graphs on the right-hand side may be recollision graphs, i.e., there is at least one dotted edge between a pair of internal and external atoms (recall Definition \ref{Def_recoll}). 
In addition, every molecular graph on the right-hand side can be obtained by a composition of the operations (L1)--(L3) in Definition \ref{Goper} below acting on the molecular graph of $\mathcal G$. As a consequence, if $\cal G$ is doubly connected, then all graphs on the right-hand side of \eqref{expand lvl1} are also doubly connected.
\end{lemma}
\begin{proof}
This lemma is a generalization of Lemma 3.22 in \cite{BandI}, which only considered graphs with three external atoms $\fa,\fb_1,\fb_2$ (i.e., graphs coming from expansions of $T_{\fa,\fb_1\fb_2}$), while the graph $\cal G$ in our lemma is more general. But the proof for \cite[Lemma 3.22]{BandI} can be applied to our case almost verbatim, so we omit the details. 
\end{proof}

\begin{defn}[Local molecular operations]\label{Goper}
We define the following operations on molecular graphs related to local expansions on $x$. Let $\cal M_x$ be the molecule containing $x$.
\begin{itemize}
\item[(L1)] Merge $\cal M_x$ with another molecule.
\item[(L2)] Add two new solid edges of the same color between $\cal M_x$ and another molecule.
\item[(L3)] For a pair of molecules, say $\mathcal M_1$ and $\mathcal M_2$, remove a solid edge between $\mathcal M_1$ and $\mathcal M_2$, and then add two solid edges of the same color: one between $\mathcal M_1$ and $\mathcal M_x$ and the other between $\mathcal M_2$ and $\mathcal M_x$. For simplicity of presentation, we will call this operation ``$\mathcal M_x$ pulls a solid edge between $\cal M_1$ and $\cal M_2$".
\end{itemize}
\end{defn}

\subsection{Global expansions}\label{sec_goperations}

In this subsection, we define global expansions, which may create new molecules in contrast to local expansions. Suppose we have the $(n-1)$-th order $T$-expansion by induction. 

Given a locally standard graph, say $\cal G$, a global expansion involves replacing the $T$-variable containing a standard neutral atom by the $(n-1)$-th order $T$-expansion. More precisely, picking a standard neutral atom $\al$ in a locally standard graph, so that the edges connected to it take one of the forms in \eqref{12in4T}. Then, we apply the $(n-1)$-th order $T$-expansion in \eqref{mlevelTgdef} to these variables in the following way: 
\begin{align}\label{eq_tT}
	  t_{x,y_1 y_2} & = \frac { T_{x,y_1y_2} }{|m|^{2}}  - \sum_\al s_{x\al}G_{\al y_1}\overline G_{\al y_2}\left(\mathbf 1_{\al\ne y_1}\mathbf 1_{\al = y_2}+\mathbf 1_{\al=y_1}\mathbf 1_{\al \ne y_2} +\mathbf 1_{\al = y_1}\mathbf 1_{\al = y_2}\right) .
\end{align}
The last term on the right-hand side gives one (if $y_1=y_2$) or two (if $y_1\ne y_2$) recollision graphs, so we combine it with $ |m|^{-2}\sum_x \Theta^{(n-1)}_{x\al}\PT^{(n-1)}_{\al,y_1 y_2} $ and denote the resulting expression by $\mathfrak R^{(n-1)}_{x,y_1 y_2}$. This gives that  
\be\label{replaceT}
 \begin{split}
	t_{x,y_1 y_2} &=   \overline m^{-1}  \Theta^{(n-1)}_{xy_1}\overline G_{y_1y_2} + \frac{G_{y_2 y_1} - \overline G_{y_1 y_2}}{2\ii N\eta} +  \mathfrak R^{(n-1)}_{x,y_1 y_2}  \\
	& +|m|^{-2} \sum_\al \Theta^{(n-1)}_{x\al}\left[ \AT^{(>n-1)}_{\al,y_1y_2}+ \cal W^{(n-1)}_{\al,y_1y_2}  + \QT^{(n-1)}_{\al,y_1y_2} +  (\Err_{n-1,D})_{\al,y_1y_2} \right].
\end{split}
\ee
The expansion of $t_{y_1y_2, x}$ can be obtained by exchanging the order of matrix indices in the above equation. 


In a global expansion, if we replace $t_{x,y_1y_2}$ with a graph on the right-hand side of \eqref{replaceT} that is not in \smash{$ |m|^{-2}\sum_x \Theta^{(n-1)}_{x\al}\QT^{(n-1)}_{\al,y_1 y_2} $}, we then need to perform dotted edge operations in Definition \ref{dot_operation} to write it into a sum of normal graphs. For each resulting graph, we either stop expanding it or continue performing local and global expansions. On the other hand, if we replace $t_{x,y_1y_2}$ with a graph in {$ |m|^{-2}\sum_x \Theta^{(n-1)}_{x\al}\QT^{(n-1)}_{\al,y_1 y_2} $}, we will get a graph of the form  
\be\label{QG}\cal G=\sum_y \Gamma Q_y (\cal G_1) ,\ee
where both $\Gamma$ and $\cal G_1$ are graphs without $P/Q$ labels. 
Then, we need to perform the so-called \emph{$Q$-expansion} to expand \eqref{QG} into a sum of $Q$-graphs and graphs without $P/Q$ labels. The $Q$-expansion is defined in Section 4.4 of \cite{BandII}. Instead of stating the full definition, we only describe its key feature here.

Recall that $H^{(y)}$ is the $(N-1)\times(N-1)$ minor of $H$ obtained by removing the $y$-th row and column. We define the resolvent minor $G^{(y)}(z):=(H^{(y)}-z)^{-1}$. Using the Schur complement formula, we obtain that
\be\label{resolvent_GGG}
G_{x_1x_2}=G_{x_1x_2}^{(y)}+\frac{G_{x_1 y}G_{yx_2}}{G_{yy}},\quad x_1,x_2\in \Z_L^d.
\ee
Applying this identity to expand the resolvent entries in $\Gamma$ one by one, we get that
\be\label{decompose_gamma} \Gamma=\Gamma^{(y)}+\sum_{\omega} \Gamma_\omega,\ee 
where $\Gamma^{(y)}$ is a graph whose weights and solid edges are all $G^{(y)}$ entries, so that it is independent of the $y$-th row and column of $H$, and 
every $\Gamma_\omega$ has a strictly higher scaling order than $\Gamma$, at least two new solid edges connected with $y$, and a factor of the form $(G_{yy})^{-k}(\overline G_{yy})^{-l}$ for some $k,l\in \N$. 
Using \eqref{decompose_gamma}, we can expand \eqref{QG} as
\be\label{QG2}
\cal G=\sum_{\omega}\sum_y \Gamma_\omega Q_y(\cal G_1) + \sum_y Q_y\left( \Gamma \cal G_1\right)- \sum_{\omega}\sum_y Q_y\left(\Gamma_\omega \cal G_1\right),
\ee
where the second and third terms are sums of $Q$-graphs, 
and the graphs in the first term have the following key features: (i) the scaling order of $\sum_y \Gamma_\omega Q_y(\cal G_1)$ is strictly higher than $\ord(\cal G)$; (ii) at least one weight or solid edge in $\Gamma$ is pulled to atom $y$ (i.e., replaced by two solid edges connected with $y$) in $\Gamma_\omega$. 

Next, we apply \eqref{resolvent_GGG} in the reverse way to remove all $G^{(y)}$ entries and apply Taylor expansion to $(G_{yy})^{-1}=[m+(G_{yy}-m)]^{-1}$ to remove all $(G_{yy})^{-1}$ and $(\overline G_{yy})^{-1}$ entries. In this way, we can write the right-hand side of \eqref{QG2} into a sum of graphs containing only regular $G$ edges and $G$ weights. Finally, we still need to remove the $Q_y$ label in $\sum_y \Gamma_\omega Q_y(\cal G_1)$, which can be achieved by applying the local expansions and \eqref{resolvent_GGG} repeatedly. Since the full definition of the $Q$-expansion is tedious, we will not repeat it here and refer the reader to Section 4.4 of \cite{BandII} for more details. We only record the following key lemma regarding the $Q$-expansion. 


\begin{lemma}[Lemma 4.15 of \cite{BandII}]\label{Q_lemma}
Let $\cal G_0 := \sum_x \Gamma_0 Q_x(\wt{\Gamma}_0)$, where $\cal G_0$, $\Gamma_0$ and $\wt\Gamma_0$ are all normal graphs without $P/Q$ labels and $x$ is an internal atom. Then, for any large constant $D>0$, $\cal G_0$ can be expanded into a sum of $O(1)$ many graphs:
\be\label{G0Q} \cal G_0 = \sum_\omega \cal G_\omega + \sum_\zeta Q_x(\wt{\cal G}_\zeta)  + \cal G_{err},\ee
where $\cal G_{\omega}$ and $\wt{\cal G}_\zeta$ are normal graphs without $P/Q$ labels, and $\cal G_{err}$ is a sum of normal graphs of scaling order $>D$. Moreover, the following properties hold.
\begin{itemize}
	\item[(i)] Every graph on the right-hand side of \eqref{G0Q} has scaling order $\ge \ord(\cal G_0)$.
	
	\item[(ii)] If there is a new atom in a graph on the right-hand side of \eqref{G0Q}, then it is connected to $x$ through a path of waved edges (i.e., the expansion \eqref{G0Q} is a local expansion on $x$). 
	
	\item[(iii)] Every molecular graph on the right-hand side can be obtained by a composition of the operations (L1)--(L3) acting on the molecular graph of $\mathcal G_0$. 
	As a consequence, if $\cal G_0$ is doubly connected, then all graphs on the right-hand side of \eqref{G0Q} are also doubly connected.
	
	\item[(iv)] If $\Gamma_0$ does not contain any weight or solid edge attached to $x$, then the scaling order of  $\cal G_\omega$ is  strictly higher than $\ord(\cal G_0)$ for every $\omega$.
 Furthermore,  $\cal G_\omega$ contains at least one atom that belongs to the original graph $\Gamma_0$ and is connected to $x$ through a solid, waved, or dotted edge (before it is merged with $x$ by a dotted edge operation).

\end{itemize}
\end{lemma}

To summarize, global expansions correspond to the following operations. 

\begin{defn}[Global operations]\label{Goper2}
Given $t_{x,y_1 y_2}$, we define the following operations. 
\begin{itemize}
\item[(G1)] Replace $t_{x,y_1 y_2}$ by a $\dashed$ edge between $x$ and $y_1$ and a red solid edge between $y_1$ and $y_2$,  and then apply dotted edge operations if necessary. 
\item[(G2)] Replace $t_{x,y_1 y_2}$ by a free edge between $x$ and $y_1$ and a solid edge between $y_1$ and $y_2$,  and then apply dotted edge operations if necessary.
	\item [(G3)]Replace $t_{x,y_1y_2}$ by a graph in $ \mathfrak R^{(n-1)}_{x,y_1 y_2} $, and then apply dotted edge operations if necessary. 
	\item [(G4)] Replace $t_{x,y_1y_2}$ by a graph in $ \sum_\al \Theta^{(n-1)}_{x\al} \AT^{(>n-1)}_{\al,y_1y_2} $, and then apply dotted edge operations if necessary.
	\item [(G5)] Replace $t_{x,y_1y_2}$ by a graph in $ \sum_\al \Theta^{(n-1)}_{x\al} \cal W^{(n-1)}_{\al,y_1y_2} $, and then apply dotted edge operations if necessary.
	\item [(G6)] Replace $t_{x,y_1y_2}$ by a graph in $ \sum_\al \Theta^{(n-1)}_{x\al} \QT^{(n-1)}_{\al,y_1y_2} $, and then apply the $Q$-expansion if necessary. 
	\item [(G7)] Replace $t_{x,y_1y_2}$ by a graph in $ \sum_\al \Theta^{(n-1)}_{x\al} (\Err_{n-1,D})_{\al,y_1y_2}$, and then apply dotted edge operations if necessary.
\end{itemize}
\end{defn}

Using the doubly connected property of the graphs in the $(n-1)$-th order $T$-expansion, it is not hard to see that if we apply the operations (G1)--(G7) to a redundant blue solid edge (or more precisely, a $t$-variable containing this edge), the resulting graphs are still doubly connected. On the other hand, expanding a pivotal blue solid edge may break the doubly connected property. Hence, we need to make sure that whenever performing a global expansion, there is at least one redundant blue solid edge so that the expansion process can be continued without breaking the doubly connected property. This is guaranteed by the SPD property in Definition \ref{def seqPDG}. We now show that this property is preserved not only in local expansions but also in global expansions as long as we expand \emph{the first blue solid edge in a pre-deterministic order of the MIS}.  

\begin{lemma}\label{lem_SPD_preserve}
Let $\cal G$ be an SPD graph. 
\begin{enumerate}
    \item[(i)] Applying operations (L1)--(L3) to the molecular graph of $\cal G$ still gives a SPD molecular graph.
    
    \item[(ii)] Applying operations (G1)--(G6) to a $t_{x,y_1y_2}$ variable containing the first blue solid edge in a pre-deterministic order of the non-deterministic MIS still gives a sum of SPD graphs.
\end{enumerate}

\end{lemma}
\begin{proof}
The first statement was proved in Lemma 5.10 of \cite{BandII}, but under a weaker definition of the SPD property: the property (ii) in Definition \ref{def PDG} was ``if we replace each of $b_1,...,b_i$ by a $\dashed$ edge, then $b_{i+1}$ becomes redundant" and the property (ii) in Definition \ref{def seqPDG} was ``if we replace $\Iso_{j+1}$ and its two external edges in $\cal G_{\mathcal M}$ by a single $\dashed$ edge, then $\Iso_j$ becomes pre-deterministic". (This is because the concept of free edges was not introduced in \cite{BandII}.) In checking the SPD property, we change the blue solid edges or closures of isolated subgraphs into $\dashed$ or free edges one by one and check whether the next blue solid edge is redundant. Based on Lemma 5.10 of \cite{BandII}, we only need to consider cases where some blue solid edge or the closure of some isolated subgraph is replaced by a free edge at a certain step. Notice that the free edge obtained at this step is redundant in the graph, say $\cal G'$, and hence we can move the free edge to other places so that every remaining blue solid edge becomes redundant. For example, for a blue solid edge $b$ between atoms $x$ and $y$, we can treat the free edge as a blue free edge between $x$ and $y$ so that the edge $b$ becomes redundant. Then, the SPD property of $\cal G'$ is trivial to check.

The second statement for operations (G1) and (G2) follows from the definition of the SPD property of $\cal G$, because they replace the first blue solid edge of the MIS with a $\dashed$ edge and a free edge, respectively. For the operations (G3)--(G6), if we replaced $t_{x,y_1y_2}$ by a graph in 
$$\mathfrak R^{(n-1)}_{x,y_1 y_2} + \sum_\al \Theta^{(n-1)}_{x\al}\left[ \AT^{(>n-1)}_{\al,y_1y_2}+ \cal W^{(n-1)}_{\al,y_1y_2}  + \QT^{(n-1)}_{\al,y_1y_2}  \right],$$
the resulting graph, say $\cal G_{new}$, is SPD due to the SPD property of the graphs in the $T$-expansion (recall Definition \ref{def genuni2}) and the fact that $t_{x,y_1y_2}$ contains the first blue solid edge of the MIS. Then, applying dotted edge operations or the $Q$-expansion to $\cal G_{new}$ gives a sum of SPD graphs by the first statement. 
\end{proof}


With Lemma \ref{lem_SPD_preserve}, we can easily show that the globally standard property (recall Definition \ref{defn gs}) is also preserved in local and global expansions. 

\begin{lemma}\label{lem globalgood}
Let $\cal G$ be a globally standard graph without $P/Q$ labels and let $\Iso_k$ be its MIS with non-deterministic closure. For local expansions, we have that:
\begin{enumerate}
\item[(i)] Applying operations (L1)--(L3) to the molecular graph of $\cal G$ still gives a globally standard molecular graph. Furthermore, if we apply the local expansions in Lemmas \ref{ssl}--\ref{T eq0} on an atom in $\Iso_k$, then in every new $Q$-graph, the atom in the $Q$-label also belongs to the MIS with non-deterministic closure. 
\end{enumerate}
Let $t_{x,y_1y_2}$ be a $t$-variable that contains the first blue solid edge in a pre-deterministic order of $\Iso_k$. Then, we have that: 
\begin{itemize}
\item[(ii)] If we apply the operation (G1), then every new graph has no $P/Q$ label or free edge, is globally standard, and 
has one fewer blue solid edge.

\item[(iii)] If we apply the operation (G2), then every new graph has no $P/Q$ label, is globally standard, and has exactly one redundant free edge in its MIS.


\item[(iv)] If we apply the operation (G3), then every new graph has no $P/Q$ label or free edge, is globally standard, and has a scaling order $\ge \ord(\cal G) +1 $.


\item[(v)] If we apply the operation (G4), then every new graph has no $P/Q$ label or free edge, is SPD, and has a scaling order $\ge \ord(\cal G)+n-2$.


\item[(vi)] If we apply the operation (G5), then every new graph has no $P/Q$ label, is SPD, and has exactly one redundant free edge in its MIS.

\item[(vii)] If we apply the operation (G6), then we get a sum of $\OO(1)$ many new graphs: 
\be\label{mlevelTgdef3_Q}
\sum_\omega \cal G_\omega  + \cal Q   + \cal G_{err} ,
\ee
where every $\cal G_\omega$ has no $P/Q$ label or free edge, is globally standard and has a scaling order $\ge \ord(\cal G) +1 $; $\cal Q$ is a sum of $Q$-graphs, each of which has no free edge, is SPD,  and has a MIS containing the atom in the $Q$-label; $\cal G_{err}$ is a sum of doubly connected graphs of scaling orders $> D$.

\item[(viii)] If we apply the operation (G7), then every new graph is doubly connected and has a scaling order $\ge \ord(\cal G)+D-1$. 

\end{itemize}
\end{lemma}

\begin{proof}
Based on Lemma \ref{lem_SPD_preserve}, the statement (i) was proved in \cite[Lemma 6.4]{BandII}, and the statements (ii), (iv), (v), (vii), and (viii) were proved in \cite[Lemma 6.5]{BandII}. The statement (iii) can be proved in the same way as statement (ii), while the statement (vi) can be proved in the same way as statement (v). We omit the details.
\end{proof}

\subsection{Proof of Proposition \ref{Teq}}

Based on Lemma \ref{lem globalgood}, we can define the global expansion strategy for the proof of Proposition \ref{Teq}. Starting with the second order $T$-expansion \eqref{eq:2nd}, we continue to expand the term \smash{$\sum_x \Theta^{\circ}_{\fa x} \mathcal A^{(>2)}_{x,\fb_1\fb_2}$} with local and global expansions. We will stop expanding a graph if it is a normal graph and satisfies at least one of the following properties: 
\begin{itemize}
\item[(S1)] it is a $\fb_1$/$\fb_2$-recollision graph; 

\item[(S2)] its scaling order is at least $n+1$;

\item[(S3)] it contains a redundant free edge;

\item[(S4)] it is a $Q$-graph; 

\item[(S5)] it is \emph{non-expandable}, that is, its MIS with non-deterministic closure is locally standard and has no redundant blue solid edge. 
\end{itemize}
Note that if a graph $\cal G$ satisfies the property that
\be\label{det_eq}
\text{its subgraph induced on all internal atoms is deterministic,}
\ee
then $\cal G$ is non-expandable. On the other hand, in a non-expandable graph that does not satisfy \eqref{det_eq}, we cannot expand a plus $G$ edge in the MIS without breaking the doubly connected property. 


\begin{strategy}[Global expansion strategy]\label{strat_global}
Fix any $n\in \N$ and large constant $D>n$. Given the above stopping rules (S1)--(S5), we expand $T_{\fa,\fb_1\fb_2}$ according to the following strategy.
 \vspace{5pt}

\noindent{\bf Step 0}: We start with the second order $T$-expansion \eqref{eq:2nd}, where we only need to expand $\sum_x \Theta^{\circ}_{\fa x} \mathcal A^{(>2)}_{x,\fb_1\fb_2}$ since all other terms already satisfy the stopping rules. We apply local expansions to it and obtain a linear combination of new graphs, each of which either satisfies the stopping rules or is locally standard. At this step, there is only one internal molecule in every graph. Hence, the graphs are trivially globally standard.  

\vspace{5pt}
\noindent{\bf Step 1}: Given a globally standard input graph, we perform local expansions on atoms in the MIS with non-deterministic closure. We send the resulting graphs that already satisfy the stopping rules (S1)--(S5) to the outputs. Every remaining graph is globally standard by Lemma \ref{lem globalgood}, and its MIS is locally standard (i.e. the MIS contains no weight and every atom in it is either standard neutral or connected with no solid edge).

\vspace{5pt}
\noindent{\bf Step 2}: Given a globally standard input graph $\cal G$ with a locally standard MIS, say $\Iso_k$, we find a $t_{x,y_1y_2}$ or $t_{y_1y_2,x}$ variable that contains the first blue solid edge in a pre-deterministic order of $\Iso_k$. If we cannot find such a $t$-variable, then we stop expanding $\cal G$.

\vspace{5pt}
\noindent{\bf Step 3}: We apply the global expansions (G1)--(G7) to the $t_{x,y_1y_2}$ or $t_{y_1y_2,x}$ variable chosen in Step 2. We send the resulting graphs that already satisfy the stopping rules (S1)--(S5) to the outputs. The remaining graphs are all globally standard by Lemma \ref{lem globalgood}, and we sent them back to Step 1. 
\end{strategy}

Applying the above expansion strategy, we can expand $T_{\fa,\fb_1\fb_2}$ into 
$$m  \Theta^{\circ}_{\fa \fb_1}\overline G_{\fb_1\fb_2} +  \frac{|m|^2}{2\ii N\eta} (G_{\fb_2 \fb_1} - \overline G_{\fb_1 \fb_2}) $$ plus a sum of $\OO(1)$ many graphs satisfying the stopping rules (S1)--(S5). The graphs satisfying stopping rules (S1)--(S4) can be included into 
\be\label{RAWQ}
\sum_x \Theta^{\circ}_{\fa x}\left[\PTn_{x,\fb_1 \fb_2} + {\cal A}^{(>n)}_{x,\fb_1\fb_2}  + \Wn_{x,\fb_1\fb_2}  + \QTn_{x,\fb_1\fb_2}  +  (\Err_{n,D})_{x,\fb_1\fb_2}\right].
\ee
Now, suppose a graph, say $\cal G_{\fa,\fb_1\fb_2}$, is an output graph of Strategy \ref{strat_global} and does not satisfy (S1)--(S4). 
Then, $\cal G_{\fa,\fb_1\fb_2}$ either is non-expandable or does not contain a $t$-variable required by Step 2 of Strategy \ref{strat_global}. In either case, $\cal G_{\fa,\fb_1\fb_2}$ contains a locally standard MIS  $\Iso_k$, which, by the pre-deterministic property of $\Iso_k$, does not contain any internal blue solid edge. Suppose $\Iso_k$ is a proper isolated subgraph. Due to the weakly isolated property, $\Iso_k$ has at least two external red solid edges, at most one external blue solid edge, and no internal blue solid edge. Hence, $\Iso_k$ cannot be locally standard, which gives a contradiction. This shows that $\Iso_k$ is indeed the subgraph of $\cal G_{\fa,\fb_1\fb_2}$ induced on all internal atoms. Then, $\Iso_k$ is locally standard and does not contain any internal blue solid edges, so it must satisfy \eqref{det_eq}. Furthermore, by the locally standard property, $\Iso_k$ contains a standard neutral atom connected with $\fb_1$ and $\fb_2$. Thus, the output graphs of Strategy \ref{strat_global} that do not satisfy (S1)--(S4) can be written as
$$|m|^2\sum_x (\Theta^{\circ} \wt \Sigma^{(n)})_{\fa x} t_{x,\fb_1\fb_2}=  \sum_x (\Theta^{\circ} \wt \Sigma^{(n)})_{\fa x} T_{x,\fb_1\fb_2} + \sum_x (\Theta^{\circ} \wt \Sigma^{(n)})_{\fa x} \left(|m|^2 t_{x,\fb_1\fb_2} - T_{x,\fb_1\fb_2}\right),$$
where $\wt \Sigma^{(n)}$ is a sum of deterministic graphs and the second term can be included into $\sum_x \Theta^{\circ}_{\fa x} \PTn_{x,\fb_1 \fb_2}$ (recall \eqref{eq_tT}). Finally, we expand all $\Theta^{(n-1)}$ edges in \smash{$ \wt \Sigma^{(n)}$} using \eqref{chain S2k}:
$$\Theta^{(n-1)}  = \sum_{k=2}^{n-1} \Theta^{(n-1)}_k  + \Delta ^{(n-1)},$$
where $\Delta_{xy}^{(n-1)}$ is regarded as a labelled diffusive edge of scaling order $>n-1$ between $x$ and $y$. Then, we collect all graphs of scaling order $\le n$ into $ \Sigma^{(n)}$ and all remaining graphs into $\sum_x \Theta^{\circ}_{\fa x} \ATn_{x,\fb_1 \fb_2}$. 

To conclude Proposition \ref{Teq}, we still need to show that: (i) every $\dashed$ edge in $\cal E_{n}$ is redundant; (ii) the property 1 in Definition \ref{def genuni2} holds and that the sum of graphs of scaling order $\le n-1$ in $\Sigma^{(n)}$ is equal to $\Sigma^{(n-1)}$. 
Here, conclusion (i) follows from \Cref{-edc}; conclusion (ii) follows from the argument in Section 6.4 of \cite{BandII}, and we omit the details.

\section{Construction of the complete $T$-expansion}\label{sec_pf_completeT}

The $Q$-graphs in \eqref{mlevelTgdef weak} satisfy some additional properties given in the following lemma.

\begin{lemma}[$\Nonuni$: additional properties]\label{def nonuni-T extra}
Under the assumptions of Lemma \ref{def nonuni-T}, for any large constant $D>0$, $T_{\fa,\fb_1 \fb_2}$ can be expanded as \eqref{mlevelTgdef weak}, where every graph $\mathcal Q^{(\omega)}_{x,\fb_1\fb_2}$ satisfies the following properties.

\begin{enumerate}

\item It is an SPD graph (with ghost edges included in the blue net). 

\item The atom in the $Q$-label of $\mathcal Q^{(\omega)}_{x,\fb_1\fb_2} $ belongs to the MIS with non-deterministic closure. 

\item The size of $\mathcal Q^{(\omega)}_{x,\fb_1\fb_2}$ satisfies that $\size ( \mathcal Q^{(\omega)}_{x,\fb_1\fb_2} ) \le W^{-2\soe}$.

\item There is an edge, blue solid or $\dashed$ or dotted, connected to $\fb_1$; there is an edge, red solid or $\dashed$ or dotted, connected to $\fb_2$. 
\end{enumerate} 

 \end{lemma}

 \begin{proof}[\bf Proof of Lemmas \ref{def nonuni-T} and \ref{def nonuni-T extra}]
 Starting with the $n$-th order $T$-expansion \eqref{mlevelTgdef}, to conclude Lemma \ref{def nonuni-T}, we need to further expand the graphs in 
 \be\label{remain_graphs}
 \sum_x \Theta^{(n)}_{\fa x} \left(\PTn_{x,\fb_1 \fb_2} +  \ATn_{x,\fb_1\fb_2}  + \Wn_{x,\fb_1\fb_2} \right).
 \ee
 To describe the expansion strategy, we first define the stopping rules. Given the large constant $D$ in Lemma \ref{def nonuni-T}, we stop expanding a graph $\cal G$ if it is normal and satisfies at least one of the following properties:
\begin{itemize}
\item[(T1)] the subgraph of $\cal G$ induced on internal atoms is deterministic and locally standard (i.e., if there is an internal atom connected with external edges, it must be standard neutral);

\item[(T2)] $\size(\cal G)\le W^{-D}$;

\item[(T3)] $\cal G$ is a $Q$-graph. 
\end{itemize}

The following lemma shows that a graph from the expansions can be bounded by its size. It will also be used in the proof of Lemma \ref{mG-fe} in Section \ref{sec_pf_mG}.
 \begin{lemma}\label{bdd_gene_double}
 Under the setting of Lemma \ref{def nonuni-T}, let $\cal G$ be a generalized doubly connected graph (recall Definition \ref{def 2netex}). Then, we have
 $$\cal G\prec \size(\cal G).$$
\end{lemma}
\begin{proof}
With Lemma \ref{GtoAG} and Lemma \ref{w_s}, its proof is exactly the same as that for Lemma 9.11 of \cite{BandII}. So we omit the details. 
\end{proof}

The doubly connected property of the graphs from our expansions trivially implies that they are generalized doubly connected. Hence, those graphs satisfying the stopping role (T2) will be errors in $\Err_{\fa,\fb_1 \fb_2}$. 

Similar to Strategy \ref{strat_global}, the core of the expansion strategy for the proof of Lemma \ref{def nonuni-T} is still to expand the plus $G$ edges according to a pre-deterministic order in the MIS with non-deterministic closure. However, at a certain step, we have to expand a pivotal blue solid edge, say $b$, connected with an isolated subgraph. To deal with this issue, we add a ghost edge between the ending atoms of the pivotal edge $b$ and multiply the coefficient by $L^2/W^2$. Then, the blue solid edge $b$ becomes redundant, so we can expand it as in Step 3 of Strategy \ref{strat_global}. However, we need to make sure that every graph from the expansion does not have a diverging size.

\begin{strategy}
\label{strat_global_weak}
Given a large constant $D>0$ and an SPD graph $\cal G$ without $P/Q$ labels, we perform one step of expansion as follows. Let $\Iso_{k}$ be the MIS with non-deterministic closure. 
 
\vspace{5pt}

\noindent{\bf Case 1}: Suppose $\Iso_{k}$ is not locally standard. We then perform local expansions on atoms in $\Iso_{k}$ and send the resulting graphs that already satisfy the stopping rules (T1)--(T3) to the outputs. Every remaining graph has a locally standard MIS with non-deterministic closure and satisfies the SPD property by Lemma \ref{lem_SPD_preserve}.

\vspace{5pt}
\noindent{\bf Case 2}: Suppose $\Iso_{k}$ is locally standard. We find a $t_{x,y_1y_2}$ or $t_{y_1y_2,x}$ variable that contains the first blue solid edge in a pre-deterministic order of $\Iso_{k}$, and then apply the global expansions (G1)--(G7) in Definition \ref{Goper2} to it (where we need to apply \eqref{replaceT} with $n-1$ replaced by $n$). We send the resulting graphs that already satisfy the stopping rules (T1)--(T3) to the outputs, while the remaining graphs are all SPD by Lemma \ref{lem_SPD_preserve}.

\vspace{5pt}
\noindent{\bf Case 3}: Suppose $\Iso_{k}$ is deterministic, strongly isolated in $\Iso_{k-1}$, and locally standard. In other words, $\Iso_{k}$ contains a standard neutral atom, say $\al$, connected with an external blue solid edge and an external red solid edge. Suppose the edge $G_{\al y_1}$ or $G_{y_1\al}$ in a $t_{x,y_1y_2}$ or $t_{y_1y_2,x}$ variable is the pivotal external blue solid edge of $\Iso_{k}$. If there is a redundant free edge, then we move it to $(\al,y_1)$; otherwise, we add a ghost edge between $(\al,y_1)$ and multiply the graph by $L^2/W^2$. In the resulting graph, the edge $G_{\al y_1}$ or $G_{y_1\al}$ becomes redundant if we include the free edge or the added ghost edge into the blue net. Then, we apply the global expansions (G1)--(G7) in Definition \ref{Goper2} by using \eqref{replaceT} with $n-1$ replaced by $n$. We send the resulting graphs that already satisfy the stopping rules (T1)--(T3) to the outputs, while the remaining graphs are all SPD by Lemma \ref{lem_SPD_preserve}.
\end{strategy}

Applying the Strategy \ref{strat_global_weak} repeatedly to expand the graphs in \eqref{remain_graphs}, we get $\OO(1)$ many output graphs that satisfy the stopping rules (T1)--(T3). Combining them together, we get an expansion of $T^{\circ}_{\fa,\fb_1\fb_2}$ as 
	\be\label{mlevelTgdef weak2} 
	\begin{split}
		T^{\circ}_{\fa,\fb_1 \fb_2}& =m  \Theta^{\circ}_{\fa \fb_1}\overline G_{\fb_1\fb_2}+ \sum_x \Theta^{\circ}_{\fa x}\Err'_{x,\fb_1 \fb_2} +\sum_\mu \sum_{x,y} \Theta^{\circ}_{\fa x} \wt{\cal D}^{(\mu)}_{x y} T_{y, \fb_1 \fb_2}  \\
		& +  \sum_{\mu} \sum_x  \Theta^{\circ}_{\fa x}\mathcal D^{(\mu)}_{x \fb_1} \overline G_{\fb_1\fb_2} f_\mu (G_{\fb_1\fb_1})+   \sum_\nu \sum_{x}  \Theta^{\circ}_{\fa x}\mathcal D^{(\nu)}_{x \fb_2} G_{\fb_2\fb_1} \wt f_\nu(G_{\fb_2\fb_2}) \\
		&+\sum_{\gamma} \sum_x \Theta^{\circ}_{\fa x}\cal D^{(\gamma)}_{x, \fb_1 \fb_2}g_\mu(G_{\fb_1\fb_1},G_{\fb_2\fb_2},\overline G_{\fb_1\fb_2},  G_{\fb_2\fb_1})  + \sum_{\omega}\sum_x \Theta^{\circ}_{\fa x} \cal Q^{(\omega)}_{x,\fb_1\fb_2}  ,  
	\end{split}
	\ee
	where the graphs satisfying (T2) are included into $\sum_x \Theta^{\circ}_{\fa x}\Err'_{x,\fb_1 \fb_2}$, the graphs satisfying (T3) are included into \smash{$\sum_{\omega}\sum_x \Theta^{\circ}_{\fa x} \cal Q^{(\omega)}_{x,\fb_1\fb_2}$}, and the graphs satisfying (T1) give the third to sixth terms on the right-hand side (where we have included the term $m   (\Theta^{(n)}-\Theta^{\circ})_{\fa \fb_1}\overline G_{\fb_1\fb_2}$ in \eqref{mlevelTgdef} into the fourth term on the right-hand side). Following the proof of \cite[Lemma 9.7]{BandII}, we can show that the graphs \smash{$\mathcal D^{(\mu)}_{x \fb_1}$, $\mathcal D^{(\nu)}_{x \fb_2}$, $\mathcal D^{(\gamma)}_{x , \fb_1 \fb_2}$  and $\mathcal Q^{(\omega)}_{x,\fb_1\fb_2} $} satisfy the properties in Lemmas \ref{def nonuni-T} and \ref{def nonuni-T extra}, and that \smash{$\wt{\cal D}^{(\mu)}_{x y}$} are doubly connected graphs satisfying 
	\be\label{size_wtD}
	\size(\wt{\cal D}^{(\mu)}_{x y})\le \frac{W^2}{L^2} W^{-2\soe-c_0},
	\ee
under \eqref{Lcondition10}.	In fact, the doubly connected property of the graphs in \eqref{mlevelTgdef weak2} follows from the fact that the SPD property is maintained in Strategy \ref{strat_global_weak}. The main technical part is to ensure that the sizes of these graphs will not diverge due to the $L^2/W^2$ factors introduced alongside the ghost edges. Since the relevant argument is very similar to that in the proof of \cite[Lemma 9.7]{BandII},  we omit the details.
	
Now, solving the equation \eqref{mlevelTgdef weak2} as in \eqref{eq:solve-Teq-0}, we can get an expansion of $T^\circ_{\fa,\fb_1\fb_2}$, which, together with \eqref{eq:TTC}, gives the expansion \eqref{mlevelTgdef weak} with 
$$ \wt\Theta:=\Big(1- \sum_\mu \Theta^\circ \wt{\cal D}^{(\mu)}\Big)^{-1}\Theta^\circ.$$
Using Lemma \ref{dG-bd} and recalling the definition \eqref{eq:def-size}, we get that 
\begin{align*} 
\sum_{\al} \Theta^{\circ}_{x\al} \wt{\cal D}^{(\mu)}_{\al y}\prec \sum_\al B_{x\al} \left(B_{\al y}^2 +  B_{\al y} \frac{W^{2-d}}{L^2}\right) W^{4\soe}\size\left(\wt{\cal D}^{(\mu)}_{\al y}\right) \le \frac{W^{-c_0}}{ \langle x- y\rangle^{d}},
\end{align*} 
where we used \eqref{size_wtD} and $W^2/L^2 \cdot B_{xy}\le \langle x- y\rangle^d$ in the second step. Using this estimate and the Taylor expansion of $\wt\Theta$, we can get that $|\wt \Theta_{xy}| \prec B_{xy}$. This concludes the proof. 
\end{proof}

\section{Proof of Lemma \ref{mG-fe}}\label{sec_pf_mG}

Corresponding to Definition \ref{def 2netex}, we also define the \emph{generalized SPD and globally standard properties} for graphs with external molecules (recall Definitions \ref{def seqPDG} and \ref{defn gs}).

\begin{defn}\label{GSS_external}
\noindent {(i)} A graph $\cal G$ is said to satisfy the \emph{generalized SPD property} with external molecules if merging all external molecules of $\cal G$ into one single internal molecule yields an SPD graph. 

\vspace{5pt}
\noindent {(ii)} A graph $\cal G$ is said to be \emph{generalized globally standard} with external molecules if it is generalized SPD in the above sense and every proper isolated subgraph with non-deterministic closure is weakly isolated.
\end{defn}

We will expand $\cal G_{\bx}$ by applying Strategy \ref{strat_global} with the following modifications: 
\begin{itemize}
\item we will use the generalized globally standard property in Definition \ref{GSS_external};
\item for a global expansion, we will use the $\nonuni$ \eqref{mlevelTgdef weak} instead of the $n$-th order $T$-expansion; 
\item we will stop expanding a graph if it is deterministic, its size is less than $W^{-D}$, or it is a $Q$-graph.
\end{itemize} 
More precisely, we expand $\cal G_{\bx}$ according to the following strategy. 

\begin{strategy}
\label{strat_global_weak2}
Given a large constant $D>0$ and a generalized globally standard graph without $P/Q$ labels, we perform one step of expansion as follows. 

\vspace{5pt}
\noindent{\bf Case 1:} Suppose we have a graph where all solid edges are between external molecules. Corresponding to Step 0 of Strategy \ref{strat_global}, we perform local expansions to get a sum of locally standard graphs 
plus some graphs that already satisfy the stopping rules. Then, we apply \eqref{mlevelTgdef weak} to 
an arbitrary $t$-variable and get a sum of generalized globally standard graphs plus some graphs that already satisfy the stopping rules. 

\vspace{5pt}
\noindent{\bf Case 2:} As in Step 1 of Strategy \ref{strat_global}, we perform local expansions on atoms in the MIS with non-deterministic closure and get a sum of generalized globally standard graphs with locally standard MIS plus some graphs that already satisfy the stopping rules. 

\vspace{5pt}
\noindent{\bf Case 3:} Given a generalized globally standard input graph with a locally standard MIS, we find a $t$-variable that contains the first blue solid edge in a pre-deterministic order of the MIS as in Step 2 of Strategy \ref{strat_global}, and then apply \eqref{mlevelTgdef weak} to it as in Step 3 of Strategy \ref{strat_global} to get a sum of generalized globally standard graphs plus some graphs that already satisfy the stopping rules. 
\end{strategy}

In the proof of Lemma 9.14 in \cite{BandII}, it has been shown that applying the Strategy \ref{strat_global_weak2} repeatedly to expand $\cal G_{\bx}$ will finally give a sum of graphs satisfying the stopping rules, where the deterministic graphs are denoted as \smash{$\cal G_{\bx}^{(\mu)}$}, the $Q$-graphs have zero expectation, and the remaining graphs have small enough sizes and give an error $\OO(W^{-D})$ by Lemma \ref{bdd_gene_double}. Similar to the proof of \cite[Lemma 9.14]{BandII}, we can show that the property (a) of Lemma \ref{mG-fe} follows from the generalized globally standard property. The property (d) can be shown easily by keeping track of the sizes of the graphs from expansions using property (2) of Lemma \ref{def nonuni-T} and property (3) of Lemma \ref{def nonuni-T extra}. The main goal of the rest of the proof is to show that properties (b) and (c) of Lemma \ref{mG-fe} hold. For this purpose, we define the following property, which coincides with property (b) for deterministic graphs. We will show that this property holds throughout the expansion process.

\begin{defn}
(b') After replacing the maximal (weakly) isolated subgraph by a $\dashed$ or free edge in the molecular graph without red solid edges, we can find at least $r \ge \lceil n/2\rceil$ external molecules that are simultaneously connected with special redundant ($\dashed$ or free) edges in the following sense: {after replacing every blue solid edge in the graph with a $\dashed$ or free edge}, the resulting graph satisfies property (b) of Lemma \ref{mG-fe}.
\end{defn}

In the above definition, by ``replacing every blue solid edge in the graph with a $\dashed$ or free edge", we mean all possible assignments of the types ($\dashed$ or free) for the blue solid edges instead of only one particular assignment.

Given a generalized doubly connected graph, say $\cal G$, satisfying properties (b') and (c), we will call the $r$ molecules in property (b') as ``\emph{special redundant molecules}", the corresponding $r$ redundant edges as ``\emph{special redundant edges}", 
other external molecules that are neighbors of internal molecules as ``\emph{special pivotal molecules}", and other edges between external and internal molecules as ``\emph{special pivotal edges}". 
We need to show that all new graphs from an expansion of $\cal G$ also satisfy properties (b') and (c). In the following proof, whenever we say the ``original graph", we actually mean the original graph together with a certain assignment of the types ($\dashed$ or free) for all blue solid edges. It will be clear from the context which assignment we are referring to. In the proof, we will use the following trivial but very convenient fact. 
\begin{claim}\label{trivial_graph1}
Given a graph $\cal G$, let $\cal S$ be a subset of molecules such that the subgraph induced on $\cal S$ is doubly connected. Then, $\cal G$ is doubly connected if and only if the following quotient graph $\cal G/\cal S$ is doubly connected: $\cal G/\cal S$ is obtained by treating the subset of molecules $\cal S$ as one single vertex, and the edges connected with $\cal S$ in $\cal G$ are now connected with this vertex $\cal S$ in $\cal G/\cal S$. 
\end{claim}

We divide the proof into three cases corresponding to the three cases in Strategy \ref{strat_global_weak2}, which are respectively given in Sections \ref{sec_case1}--\ref{sec_case3}. These proofs show that the properties (b') and (c) are preserved throughout the expansions, and hence complete the proof of Lemma \ref{mG-fe}.

\subsection{Global expansions of blue solid edges between external molecules}\label{sec_case1} 

It is trivial to see that the local expansions in Case 1 of Strategy \ref{strat_global_weak2} do not break properties (b') and (c) of the input graph $\cal G$ since no internal molecule is generated in this process. 
We only need to consider the global expansion of a blue solid edge between external molecules. First, replacing the relevant $t$-variable with $\Err_{\fa,\fb_1 \fb_2}$ in \eqref{mlevelTgdef weak} gives an error term $\OO(W^{-D})$. Second, by using Claim \ref{trivial_graph1}, it is easy to check that replacing the relevant $t$-variable with other non-$Q$ graphs in \eqref{mlevelTgdef weak} does not break properties (b') and (c). 
Hence, we only need to consider the case where the blue solid edge is replaced by a $Q$-graph in  \eqref{mlevelTgdef weak}. Suppose we have expanded a blue solid edge between molecules $x_i$ and $x_j$. After $Q$-expansions, we obtain the following three possible cases:
\begin{center}
    \includegraphics[width=10cm]{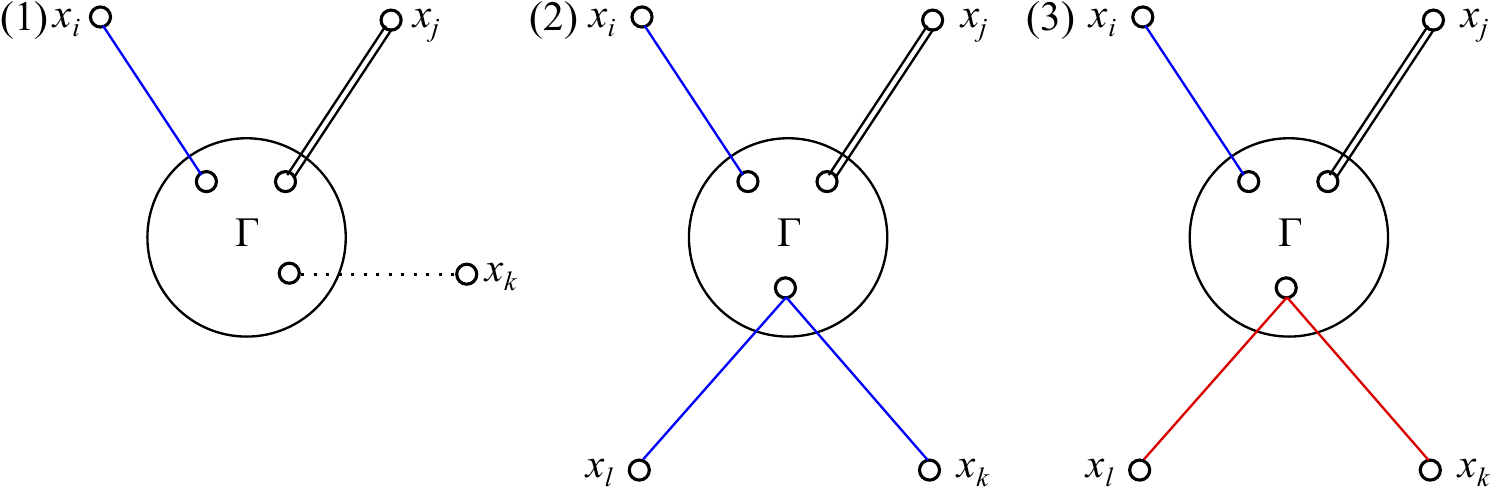}
\end{center}
All these graphs are molecular graphs, $\Gamma$ indicates the subgraph component induced on newly generated internal molecules, and we have only drawn the relevant molecules and edges without showing all other details. In graph (1), a molecule in $\Gamma$ is merged with an external molecule $x_k$ due to the dotted edge between them; in graph (2), a blue solid edge between external molecules $x_k$ and $x_l$ is pulled to a molecule in $\Gamma$; in graph (3), a red solid edge between external molecules $x_k$ and $x_l$ is pulled to a molecule in $\Gamma$ and the cases (1) and (2) do not happen. Note that in case (3), $\Gamma$ is a weakly isolated subgraph. After replacing its closure with a $\dashed$ or free edge, the resulting graph trivially satisfies properties (b') and (c). It remains to show that the graphs in cases (1) and (2) satisfy properties (b') and (c). 

In case (1), the internal molecules in $\Gamma$ are connected with $x_k$ through paths of blue edges and paths of black edges, so both $x_i$ and $x_j$ are special redundant in the new graph. Hence, if either $x_i$ or $x_j$ is \emph{not} special redundant in the original graph, then we have at least one more special redundant molecule and at most one more special pivotal molecule $x_k$ in the new graph. On the other hand, suppose $x_i$ or $x_j$ are both special redundant molecules in the original graph. Then, they are also special redundant in the new graph, and the internal molecules in $\Gamma$ are connected with $x_j$ through paths of black edges and connected with $x_i$ through paths of blue edges. Hence, $x_k$ is a special redundant molecule in the new graph. Together with the simple fact that $x_i$ still connects to $x_j$, it shows that the graph in case (1) satisfies properties (b') and (c). 

In case (2), by Claim \ref{trivial_graph1}, it is equivalent to consider the following graph (2.1) with $\Gamma$ reduced to a vertex $a$:
\begin{center}
    \includegraphics[width=9cm]{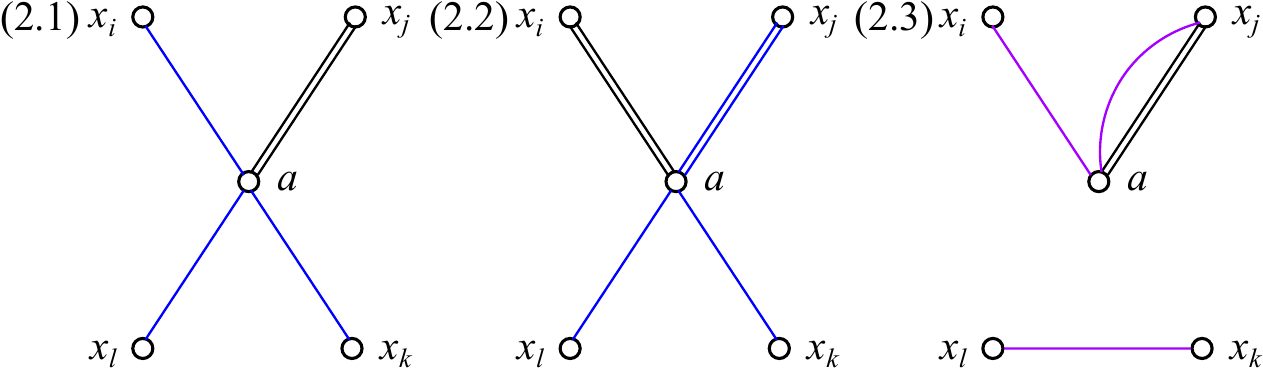}
\end{center}
Note that two of $x_i$, $x_k$ and $x_l$ can be chosen as special redundant. Hence, if at least two of them are \emph{not} special redundant in the original graph, then the new graph will have two more special redundant molecules and at most two more special pivotal molecules. If only one of them is \emph{not} special redundant in the original graph, the resulting graph will have one more special redundant molecule and at most one more special pivotal molecule $x_j$. If all of them are special redundant in the original graph, then we need to consider two different cases depending on how we assign the types for the blue solid edges. First, suppose we replace one external blue solid edge, say $(x_i,a)$, with a $\dashed$ edge as in graph (2.2). Then, putting $(x_i,a)$ into the black net and $(x_j,a)$ into the blue net gives a graph where $x_j$ becomes special redundant. Second, suppose we replace all external blue solid edges with free edges, then we rearrange them as in graph (2.3), where the molecule $x_j$ becomes special redundant. Moreover, notice that in graphs (2.2) and (2.3), $x_i$ still connects to $x_j$ and $x_k$ still connects to $x_l$. Hence, the new graph in case (2) satisfies properties (b') and (c).  


\subsection{Local expansions} \label{sec_case2}
In this subsection, we consider the local expansions in Case 2 of Strategy \ref{strat_global_weak2}. We divide the following proof into three different cases.

\vspace{5pt}

\noindent{\bf Case I}: Suppose that there are no isolated subgraphs in $\cal G$. It is easy to check that special redundant molecules after a local expansion are still special redundant. We only need to consider molecules that are not connected with internal molecules directly in the original graph, but later become neighbors of some internal molecules after local expansions. This may happen if: (i) an external molecule is merged with an internal molecule due to newly added waved or dotted edges; (ii) a blue solid edge between external molecules is pulled to an internal molecule. In both cases, the property (c) for new graphs is simple to check, and we only need to examine the property (b').

In case (i), suppose an internal molecule $a$ is merged with an external molecule $x_k$. We claim that this molecule is connected with at least one redundant blue edge and one redundant black edge. Take the black net as an example. In the original graph, we pick a disjoint union of black spanning trees that connect all internal molecules to special pivotal external molecules. Suppose $a$ connects to a special pivotal molecule $x_i$ on a black tree with $x_i$ being its root. Then, after merging $a$ and $x_k$, the black edge between $a$ (i.e., $x_k$) and its parent on the tree is special redundant, because all children of $a$ now connect to $x_k$ and all ancestors of $a$ still connect to $x_i$ in the new graph. 
In case (ii), suppose a blue solid edge between external molecules is pulled to an internal molecule $a$. Then, both the edges $(x_i,a)$ and $(x_i,a)$ are special redundant, and hence $x_i$ and $x_j$ are special redundant molecules.

\vspace{5pt}

\noindent{\bf Case II}: In this case, suppose we are performing local expansions on atoms in a maximal weakly isolated subgraph, say $\Gamma$, between internal molecules. Then, we show that the operations (L1)--(L3) in Definition \ref{Goper} do not break the properties (b') and (c). If these operations do not involve molecules outside $\Gamma$, then there is nothing to prove. Also, notice that operation (L2) is a special case of (L3) by taking $\cal M_1=\cal M_2$. Hence, we only need to show that after the following five kinds of operations, the new graphs still satisfy properties (b') and (c): 
\begin{center}
	\includegraphics[width=11cm]{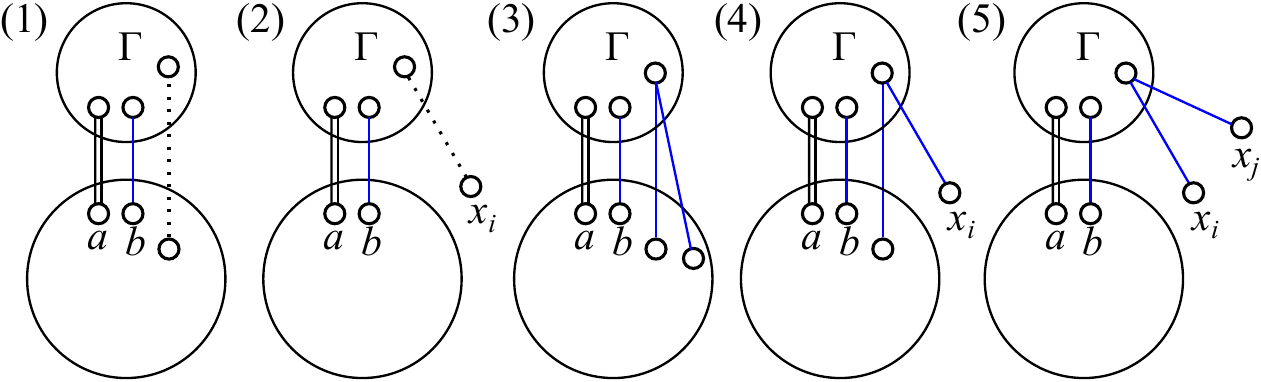}
\end{center}
All these graphs are molecular graphs without red solid edges, $\Gamma$ indicates the maximal isolated subgraph in the original graph, and the lower circle indicates the subgraph induced on molecules that are not in $\Gamma$. Again, we have only drawn the relevant molecules and edges without showing all other details. In graphs (1) and (2), a molecule in $\Gamma$ is merged with an internal molecule and an external molecule, respectively, where we have used a dotted edge to indicate this operation. In graph (3), a blue solid edge between internal molecules is pulled to a molecule in $\Gamma$. In graph (4), a blue solid edge between an internal molecule and an external molecule $x_i$ is pulled to a molecule in $\Gamma$. In graph (5), a blue solid edge between external molecules $x_i$ and $x_j$ is pulled to a molecule in $\Gamma$. For simplicity of presentation, in the following proof, whenever we refer to the ``original graph", we actually mean the original graph with the closure of $\Gamma$ replaced by a $\dashed$ or free edge $(a,b)$. 


Using Claim \ref{trivial_graph1}, it is trivial to check that the graph in case (1) satisfies properties (b') and (c). In case (2), it is equivalent to consider the graph with $\Gamma$ reduced to a vertex:
\begin{center}
    \includegraphics[width=3.5cm]{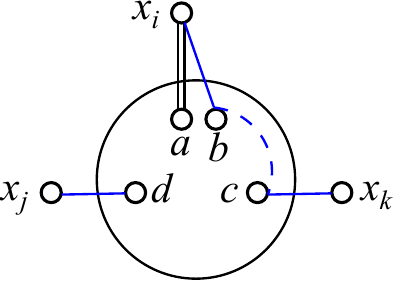}
\end{center}
Here, $x_i$, $x_j$, and $x_k$ are external molecules and we assume that $(x_j,d)$ is a special redundant edge and $(x_k,c)$ is a special pivotal edge in the original graph. 
We claim that the special redundant edges such as $(x_j,d)$ in the original graph are still redundant in the new graph, and one of the edges $(a,x_i)$ and $(b,x_i)$ is special redundant. First, if $(b,x_i)$ is replaced by a free edge, then it can play the role of a free edge $(a,b)$ in the original graph, so that $(x_j,d)$ and $(x_i,a)$ are special redundant in the new graph. 
Second, suppose $(b,x_i)$ is replaced by a $\dashed$ edge. Assume that $(a,b)$ is used as a blue $\dashed$ edge in the original graph for the property (b') to hold.  
Then, there exists a blue path between $a$ (or $b$) and a special pivotal edge (this is because by the generalized SPD property of the original graph, removing the closure of $\Gamma$ still gives a doubly connected graph). 
WLOG, suppose this path is between $b$ and $c$, denoted by the blue dashed edge in the above graph. Then, the edge $(x_i, b)$ is special redundant in the new graph, because the molecules $a$ and $b$ are connected through the edge $(a,x_i)$ from $a$ to (the equivalence class of) external molecules, the edge $(x_k, c)$, and the blue dashed path, where the edge $(b,x_i)$ is not used. The case where $(a,b)$ is used as a black $\dashed$ edge in the original graph can be proved in the same way by using a black-blue symmetry (i.e., switching the colors of the $\dashed$ edges in the above argument leads to a proof). Together with the fact that $a$ and $b$ are still connected in the new graph, the above arguments show that the properties (b') and (c) hold for case (2).

In case (3), reducing $\Gamma$ into a single vertex $y$ gives the following graph:
\begin{center}
    \includegraphics[width=2.3cm]{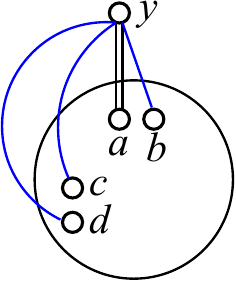}
\end{center}
If the edge $(b,y)$ is replaced by a free edge, then it can play the role of a free edge $(a,b)$ in the original graph. If both $(c,y)$ and $(d,y)$ are replaced by free edges, then they can play the role of free edges $(a,b)$ and $(c,d)$ in the original graph. If $(b,y)$ is replaced by a $\dashed$ edge while $(c,y)$ and $(d,y)$ are replaced by a free edge and a $\dashed$ edge, respectively, then the free edge can play the role of a free edge $(c,d)$ and the two $\dashed$ edges $(a,y)$ and $(b,y)$ replace the role of a $\dashed$ edge $(a,b)$ in the original graph. With all these assignments, we can easily check that the special redundant molecules in the original graph are still special redundant in the new graph in all the above cases. It remains to consider the hardest case where the edges $(b,y)$, $(c,y)$ and $(d,y)$ are all replaced by diffusive edges. 

First, we assume that in the original graph, both $(a,b)$ and $(c,d)$ are used as blue $\dashed$ edges in order for the property (b') to hold. Without loss of generality, we assume that the molecule $b$ connects to a special pivotal molecule, say $x_j$.
\begin{center}
    \includegraphics[width=7cm]{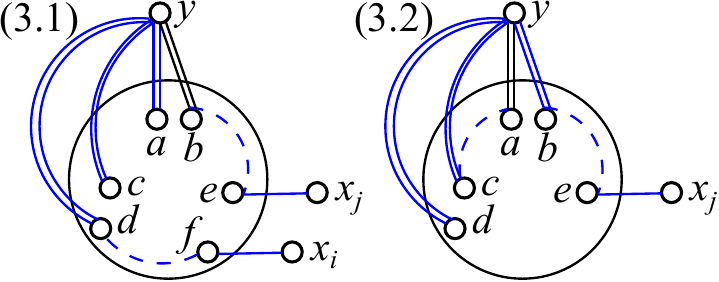}
\end{center}
We first consider a case such that in the original graph, $c$ or $d$ connects to a special pivotal molecule, say $x_i$, through a blue path $(d,f)$ which does not pass the edge $(a,b)$; see graph (3.1) above. Then, if we put $(b,y)$ into the black net and $(a,y)$, $(c,y)$ and $(d,y)$ into the blue net, $a$ still connects to $b$ through a blue path consisting of edges $(a,y)$, $(d,y)$, $(x_i,f)$ and $(x_j,e)$, the blue path from $d$ to $f$, and the blue path from $e$ to $b$. Thus, the graph (3.1) also satisfies property (b') as the original graph. On the other hand, suppose in the original graph, every blue path from $c$ or $d$ to a special pivotal molecule, say $x_j$, has to pass the edge $(a,b)$; see graph (3.2) above. Then, if we put $(a,y)$ into the black net and $(b,y)$, $(c,y)$ and $(d,y)$ into the blue net, $a$ still connects to $b$ through a blue path consisting of edges $(b,y)$, $(c,y)$ and the blue path from $a$ to $c$. Thus, the graph (3.2) also satisfies property (b') as the original graph. 
Second, suppose that in the original graph, $(a,b)$ and $(c,d)$ are used as blue $\dashed$ edge and black $\dashed$ edge, respectively, in order for the property (b') to hold. Then, putting $(a,y)$ and $(b,y)$ into the blue net and $(c,y)$ and $(d,y)$ into the black net, the resulting graph still satisfies property (b'). Finally, the case where the original graph uses $(a,b)$ as a black $\dashed$ edge can be handled in the same way by using a black-blue symmetry. Together with the simple fact that $a$, $b$ are connected and $c$, $d$ are connected in the new graphs, the above arguments show that the properties (b') and (c) hold for case (3).

In case (4), without loss of generality, we assume that the pulled edge, say $(x_i,c)$, is used as a \emph{special redundant edge} or a \emph{special pivotal blue edge} in the original graph. The case where $(x_i,c)$ is used as a \emph{special pivotal black edge} in the original graph can be handled in the same way by using a black-blue symmetry. Reducing $\Gamma$ into a single molecule $y$, we get the following graph (4.1): 
\begin{center}
    \includegraphics[width=14cm]{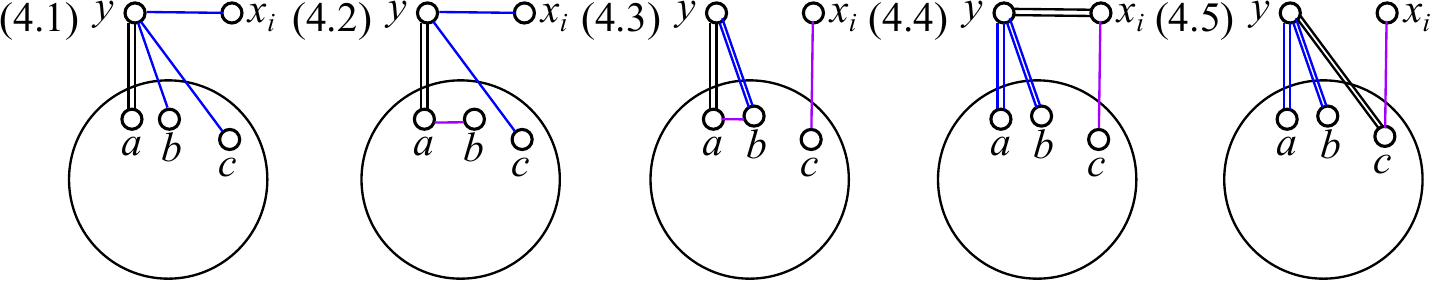}
\end{center}
We need to show that the special redundant molecules in the original graph are still special redundant in new graphs. First, if $(b,y)$ is replaced by a free edge, then it can play the role of a free edge $(a,b)$  in the new graph (4.2), and we see that the original special redundant molecules (which may include $x_i$) are still special redundant. Second, if $(c,y)$ and $(x_i,y)$ are replaced by free edges and $(b,y)$ is replaced by a $\dashed$ edge, then one free edge plays the role of a free edge $(a,b)$ and the other one plays the role of a free edge $(x_i,c)$ in the new graph (4.3), and we see that the original special redundant molecules are still special redundant. Third, if $(c,y)$ is replaced by a free edge and $(b,y)$, $(x_i,y)$ are replaced by $\dashed$ edges, then the free edge can play the role of a free edge $(x_i,c)$. If $(a,b)$ is used as a blue $\dashed$ edge in the original graph, then we put $(a,y)$ and $(b,y)$ into the blue net and $(x_i,y)$ into the black net as in graph (4.4); otherwise, we put $(a,y)$ and $(b,y)$ into the black net and $(x_i,y)$ into the blue net. In this way, the original special redundant molecules are still special redundant in the new graph. Fourth, if $(x_i,y)$ is replaced by a free edge and $(b,y)$, $(c,y)$ are replaced by $\dashed$ edges, then the free edge can play the role of a free edge $(x_i,c)$. If $(a,b)$ is used as a blue $\dashed$ edge in the original graph, we put $(a,y)$, $(b,y)$ into the blue net and $(c,y)$ into the black net as in graph (4.5); otherwise, we put $(a,y)$, $(b,y)$ into the black net and $(c,y)$ into the blue net. In this way, the original special redundant molecules are still special redundant in the new graph. Fifth, suppose $(x_i,y)$, $(b,y)$ and $(c,y)$ are all replaced by $\dashed$ edges, and $(a,b)$ is used as a black $\dashed$ edge in the original graph for the property (b') to hold. Then, we put $(a,y)$, $(b,y)$ into the black net and $(c,y)$, $(x_i,y)$ into the blue net as in graph (4.6). We see that the original special redundant molecules are still special redundant in the new graph, because $(a,y)$ and $(b,y)$ play the same role as $(a,b)$ in the original graph, while $(x_i,y)$ and $(c,y)$ play the same role as $(x_i,c)$ in the original graph.
\begin{center}
    \includegraphics[width=6cm]{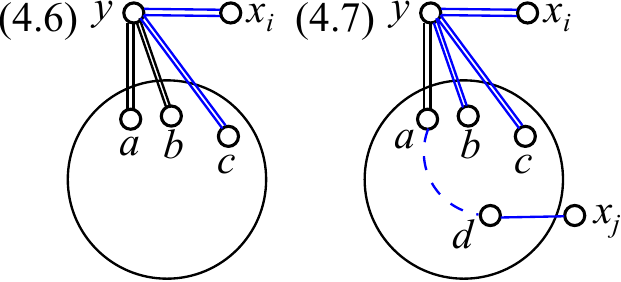}
\end{center}
Sixth, suppose $(x_i,y)$, $(b,y)$ and $(c,y)$ are all replaced by $\dashed$ edges, $(a,b)$ is used as a blue $\dashed$ edge, and $(x_i,c)$ is used as a special redundant edge in the original graph for the property (b') to hold. Then, we put $(a,y)$, $(b,y)$ into the blue net and $(c,y)$ into the black net, and we see that the original special redundant molecules are still special redundant in the new graph.
Seventh, suppose $(x_i,y)$, $(b,y)$ and $(c,y)$ are all replaced by $\dashed$ edges, $(a,b)$ is used as a blue $\dashed$ edge, and $(x_i,c)$ is used as a special pivotal blue $\dashed$ edge in the original graph for the property (b') to hold. Without loss of generality, suppose $a$ connects to a special pivotal edge $(x_j,d)$ through a blue path as in graph (4.7). (The molecule $d$ can be $c$, in which case the following proof still works.) If we put $(a,y)$ into the black net and $(b,y)$, $(c,y)$ and $(x_i,y)$ into the blue net, then  $a$ still connects to $b$ through a blue path consisting of edges $(b,y)$, $(x_i,y)$ and $(x_j,d)$, and the blue path from $a$ to $d$. Moreover, $(x_i,y)$ and $(c,y)$ can play the same role as $(x_i,c)$ in the original graph. Hence, the original special redundant molecules are still special redundant in the new graph (4.7). 

Combining all the above arguments with the simple fact that $a$, $b$ are connected and $x_i$, $c$ are connected in the new graphs, we obtain that the properties (b') and (c) hold for case (4).

Finally, in case (5), we consider the following graph (5.1) with $\Gamma$ reduced into a single molecule $y$. 
\begin{center}
    \includegraphics[width=10cm]{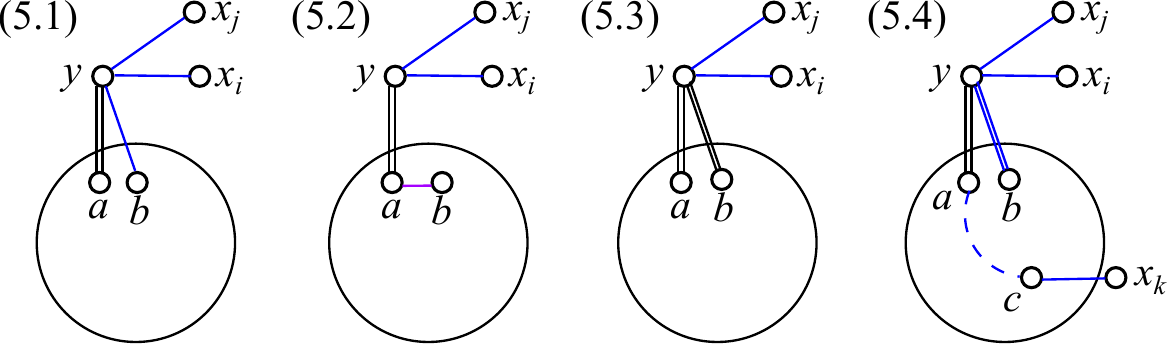}
\end{center}
First, if $(b,y)$ is replaced by a free edge, then it can play the same role as the free edge $(a,b)$ in the original graph. In the new graph (5.2), either $x_i$ or $x_j$ can be chosen as special redundant, and hence the property (b') holds. Second, suppose $(b,y)$ is replaced by a $\dashed$ edge, and $(a,b)$ is used as a black $\dashed$ edge in the original graph for the property (b') to hold. Then, putting $(a,y)$, $(b,y)$ into the black net gives a new graph (5.3) in which either $x_i$ or $x_j$ is special redundant, and hence the property (b') holds. Third, suppose $(b,y)$ is replaced by a $\dashed$ edge, and $(a,b)$ is used as a blue $\dashed$ edge in the original graph for the property (b') to hold. Without loss of generality, suppose $a$ connects to a special pivotal edge $(x_k,c)$ through a blue path as in graph (5.4). Then, putting $(a,y)$ into the black net and $(b,y)$ into the blue net gives a graph in which either $x_i$ or $x_j$ is special redundant, because $a$ still connects to $b$ through a blue path consisting of $(b,y)$, $(x_k,c)$, $(x_i,y)$ or $(x_j,y)$, and the blue path from $a$ to $c$. Hence, the graph (5.4) satisfies the property (b').
Together with the simple fact that $a$, $b$ are connected and $x_i$, $x_j$ are connected in the new graph, the above arguments show that the properties (b') and (c) hold for case (5).

\vspace{5pt}

\noindent{\bf Case III}: In this case, suppose we are performing local expansions on atoms in a maximal weakly isolated subgraph, say $\Gamma$, between a pair of internal and external molecules. Similar to Case II, we need to show that the operations (L1)--(L3) involving molecules outside $\Gamma$ do not break the properties (b') and (c). More precisely, we will show that after the following five kinds of operations, the new graphs still satisfy properties (b') and (c):
\begin{center}
    \includegraphics[width=11.5cm]{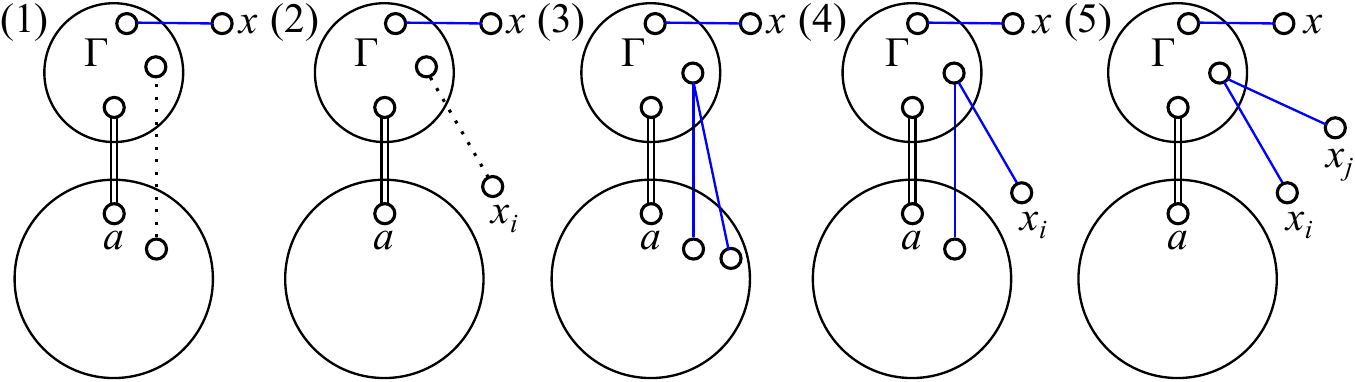}
\end{center}
In these graphs, $\Gamma$ indicates the maximal isolated subgraph, the lower circle indicates the subgraph induced on molecules that are not in $\Gamma$, and we have only drawn the relevant molecules and edges without showing all other details. In graphs (1) and (2), a molecule in $\Gamma$ is merged with an internal molecule and an external molecule, respectively. In graph (3), a blue solid edge between internal molecules is pulled to a molecule in $\Gamma$. In graph (4), a blue solid edge between an internal molecule and the external molecule $x_i$ is pulled to a molecule in $\Gamma$. In graph (5), a blue solid edge between the external molecules $x_i$ and $x_j$ is pulled to a molecule in $\Gamma$. Again, for simplicity of presentation, the ``original graph" in the following proof refers to the original graph where the closure of $\Gamma$ is replaced by a $\dashed$ or free edge $(x,a)$.

The cases (1) and (2) are almost trivial by using Claim \ref{trivial_graph1}, and the case (3) can be handled in a similar way as the above Case II-(4). If the edge $(x,a)$ is special redundant in the original graph, then treating the external molecule $x$ as an internal molecule, the current Case III becomes a special case of Case II above. Furthermore, if the blue solid edge between $x$ and $\Gamma$ is replaced by a free edge, then it can be used as a free edge $(x,a)$ in the original graph and the proof will be simple, so we omit the details. It remains to study cases (4) and (5) with $(x,a)$ being special pivotal and the blue solid edge between $x$ and $\Gamma$ replaced by a $\dashed$ edge.

In case (4), reducing $\Gamma$ into a single molecule $y$ gives the following graph (4.1), where a blue solid edge $(x_i,b)$ is pulled to $y$:
\begin{center}
    \includegraphics[width=10cm]{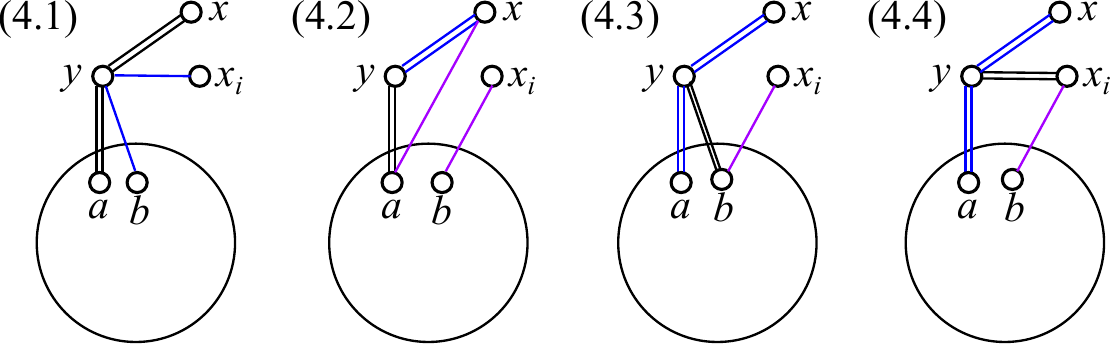}
\end{center}
First, if $(b,y)$ and $(x_i,y)$ are replaced by free edges, then they can play the same role as the free edges $(x,a)$ and $(x_i,b)$ in the original graph. Then, putting $(x,y)$ and $(a,y)$ into the blue and black nets, respectively, gives a new graph (4.2) that satisfies the property (b').   
Second, if $(b,y)$ is replaced by a diffusive edge and $(x_i,y)$ is replaced by a free edge, then the free edge plays the same role as the free edge $(x_i,b)$ in the original graph. Moreover, if $(x,a)$ is used as a blue diffusive edge in the original graph for the property (b') to hold, then we put $(x,y)$ and $(a,y)$ into the blue net and $(b,y)$ into the black net as in graph (4.3); otherwise, we put $(x,y)$ and $(a,y)$ into the black net and $(b,y)$ into the blue net. In this way, $(x,y)$ and $(a,y)$ together will replace the role of $(x,a)$ in the original graph, and hence the new graph will satisfy the property (b'). 
Third, if $(b,y)$ is replaced by a free edge and $(x_i,y)$ is replaced by a diffusive edge, then we can deal with it in a similar way as the second case; see graph (4.4).

Now, suppose $(b,y)$ and $(x_i,y)$ are both replaced by diffusive edges. Without loss of generality, we assume that $(x,a)$ is used as a blue diffusive edge in the original graph, while the case with $(x,a)$ used as a black $\dashed$ edge in the original graph can be handled in the same way by using a black-blue symmetry. On the one hand, if $(x_i,b)$ is used as a black diffusive edge in the original graph for the property (b') to hold, then we put $(x,y)$, $(a,y)$ into the blue net and $(x_i,y)$, $(b,y)$ into the black net. The new graph (4.5) will satisfy the property (b'). 
\begin{center}
    \includegraphics[width=5.5cm]{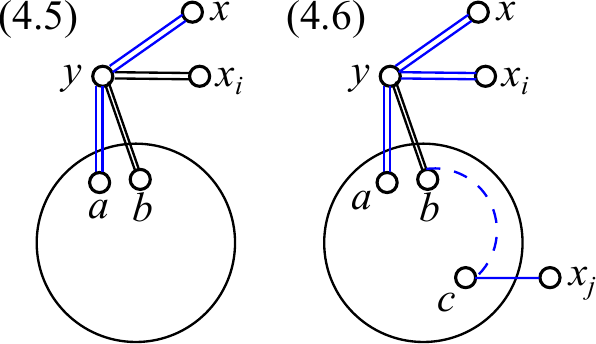}
\end{center}
On the other hand, suppose $(x_i,b)$ is used as a blue diffusive edge in the original graph for the property (b') to hold. Without loss of generality, assume that $b$ is connected with a blue special pivotal edge $(x_j,c)$ in the original graph, which is represented by the blue path in graph (4.6) (this pivotal edge can be $(x,a)$ in the original graph, in which case the following proof still works). Then, putting $(a,y)$, $(x,y)$ and $(x_i,y)$ into the blue net and $(b,y)$ into the black net as in graph (4.6), the resulting graph will satisfy the property (b'). This is because, in the new graph, $b$ still connects to $x_i$ (which belongs to the same equivalence class of external molecules as $x_j$) through a blue path consisting of the blue path from $b$ to $c$ and the edge $(x_j,c)$.

Combining the above arguments with the simple fact that $x$, $a$ are connected and $x_i$, $b$ are connected in the new graphs, we see that the properties (b') and (c) hold for case (4).

Finally, consider the case (5) with $(x,a)$ being special pivotal in the original graph and the blue solid edge between $x$ and $\Gamma$ replaced by a $\dashed$ edge. Reducing $\Gamma$ into a single molecule $y$ gives the following graph (5.1), where a blue solid edge $(x_i,x_j)$ is pulled to $y$:
\begin{center}
   \includegraphics[width=11cm]{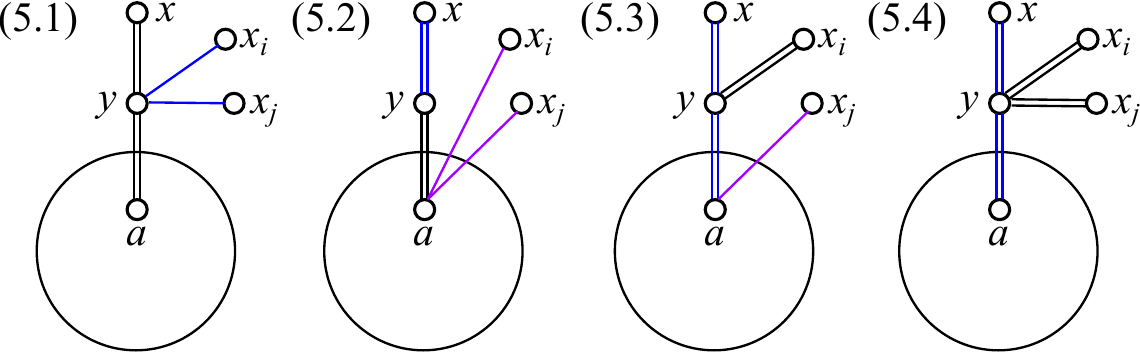}
\end{center}
If $(x,a)$ is used as a black diffusive edge in the original graph for the property (b') to hold, then it is obvious that either $x_i$ or $x_j$ can be chosen as special redundant. We now consider the harder case where $(x,a)$ is used as a blue diffusive edge in the original graph for the property (b') to hold.
First, if $(x_i,y)$ and $(x_j,y)$ are both replaced by free edges, then we treat them as free edges $(x_i,a)$ and $(x_j,a)$ as in graph (5.2). Note that one of them can play the same role as the free edge $(x,a)$ in the original graph, while the other one is redundant if we put $(x,y)$ into the blue net and $(y,a)$ into the black net. This shows that at least one of $x_i$ and $x_j$ is free redundant, and hence the new graph (5.2) satisfies the property (b'). 
Second, suppose one of $(x_i,y)$ and $(x_j,y)$ is replaced by a free edge, and the other one is replaced by a diffusive edge. Without loss of generality, let the free edge be $(x_j,y)$. If we treat it as a free edge $(x_i,a)$ and put $(x,y)$ and $(a,y)$ into the blue net and $(x_i,y)$ into the black net as in graph (5.3), then the edge $(x_j,a)$ is special redundant. On the other hand, if we put $(x,y)$ into the black net and $(y,a)$ into the blue net, then $(x_j,a)$ can replace the role of $(x,a)$ in the original graph and the edge $(x_i,y)$ becomes special redundant. In either case, at least one of $x_i$ and $x_j$ is special redundant, and hence the new graph (5.3) satisfies the property (b').
Finally, suppose $(x_i,y)$ and $(x_j,y)$ are both replaced by diffusive edges. Then, putting $(x,y)$, $(a,y)$ into the blue net and $(x_i,y)$, $(x_j,y)$ into the black net as in graph (5.4), one of $x_i$ and $x_j$ becomes special redundant, and hence the property (b') still holds. Together with the simple fact that $x$, $a$ are connected and $x_i$, $x_j$ are connected in the new graphs, we see that the properties (b') and (c) hold for case (5).  

\subsection{Global expansions of blue solid edges} \label{sec_case3}

In Case 3 of Strategy \ref{strat_global_weak2}, the blue solid edge we are expanding is either between internal molecules or between a pair of external and internal molecules. First, replacing the relevant $t$-variable with $\Err_{\fa,\fb_1 \fb_2}$ in \eqref{mlevelTgdef weak} gives an error term $\OO(W^{-D})$. Second, using Claim \ref{trivial_graph1}, we see that replacing the relevant $t$-variable with other non-$Q$ graphs in \eqref{mlevelTgdef weak} does not break properties (b') and (c). Finally, suppose we replace the $t$-variable by a $Q$-graph in \eqref{mlevelTgdef weak}. Then, it is easy to see that the resulting graph before applying $Q$-expansions satisfies properties (b') and (c) by replacing the closure of the maximal isolated subgraph with a $\dashed$ or free edge. Since $Q$-expansions are local expansions, all new graphs after applying them satisfy properties (b') and (c) by the arguments in Section \ref{sec_case2}.

\end{appendix}
\begin{acks}[Acknowledgments]
 We would like to thank Paul Bourgade for the helpful discussions. We are very grateful to an anonymous referee for the helpful comments, which have resulted in a significant improvement of the paper. Fan Yang would like to thank the support of the Beijing Institute of Mathematical Sciences and Applications. 
\end{acks}

\begin{funding}
The work of Horng-Tzer Yau is partially supported by the NSF grants DMS-1855509 and DMS-2153335 and a Simons Investigator award. The work of Jun Yin is partially supported by the NSF grant DMS-1802861.
\end{funding}


\end{document}